\documentclass[12pt]{article}
\usepackage{amssymb,amsmath,amsthm,sectsty,url,mathrsfs,graphicx}
\usepackage[letterpaper,hmargin=1.0in,vmargin=1.0in]{geometry}
\usepackage[colorlinks,linkcolor=black,citecolor=black,filecolor=black,urlcolor=black]{hyperref}
\usepackage{color}

\sectionfont{\large}
\subsectionfont{\normalsize}
\numberwithin{equation}{section}

\newtheorem{maintheorem}{Theorem}

\newtheorem{cor}{Corollary}[section]
\newtheorem{con}{Conjecture}[section]
\newtheorem{lem}{Lemma}[section]
\newtheorem{defn}{Definition}[section]
\newtheorem{prop}{Proposition}[section]
\newtheorem{rem}{Remark}[section]

\newtheorem{prob}{Problem}
\allowdisplaybreaks

\numberwithin{equation}{section}

\DeclareMathSymbol{\eset}{\mathalpha}{AMSb}{"3F}

\def\RR{{\mathcal R}}

\def\TT{{\mathbb T_d}}
\def\Tk+{{\mathbb T_d^+}}

\def\SS{{\mathcal S}}

\def\GG{{\mathcal G}}

\def\VV{{\cal V}}
\def\EE{{\cal E}}

\newcommand{\la}{{\lambda}}
\newcommand{\CT}{\mathrm{CT}}

\newcommand{\PP}{{\mathbb P}}

\def\eps{\varepsilon}

\newcommand{\sfrac}[2]{\mbox{\small $\frac{#1}{#2}$}}
\newcommand{\ssfrac}[2]{\mbox{\footnotesize $\frac{#1}{#2}$}}
\newcommand{\half}{\ssfrac{1}{2}}

\newcommand{\rr}{\mathbf{r}}
\newcommand{\oo}{\mathbf{o}}
\renewcommand{\Pr}{ \mathrm P}
\newcommand{\E}{ \mathbb E}

\newcommand{\Pois}{\mathrm{Pois}}
\newcommand{\as}{\mathrm{a.s.}}
\newcommand{\whp}{\mathrm{w.h.p.}}

\newcommand{\gd}{\delta}
\newcommand{ \cL}{ \mathcal L }
\newcommand{ \cT}{ \mathcal T }

\newcommand{\SC}{ \mathrm{SC} }

\let\phi=\varphi

\def\qed{\hfill $\square$}

\def\HHH{{\cal H}}

\def\dist{\mathop{\rm dist}}

\newcommand{\N}{\mathbb N}
\newcommand{\R}{\mathbb R}
\newcommand{\Z}{\mathbb Z}

\newcommand{\W}{\mathcal W}
\newcommand{\plant}{w_{\mathrm{plant}}}

\newcommand{\G}{{\cal G}}

\newcommand{\w}{{\mathbf w}}

\usepackage[normalem]{ulem}

\begin{document}

\title{Frogs on trees? }

\author{Jonathan Hermon \thanks{
University of Cambridge, Cambridge, UK. E-mail: {\tt jonathan.hermon@statslab.cam.ac.uk}. Financial support by
the  EPSRC grant EP/L018896/1.}}

\date{}
\maketitle

\begin{abstract}
We study a system of simple random walks on  $\mathcal{T}_{d,n}=(\VV_{d,n},\EE_{d,n})$, the  
 $d$-ary tree of depth $n$, known as the frog model. Initially there are Pois($\lambda$)  particles at
each site, independently, with one additional particle planted at some vertex $\mathbf{o}$. Initially all particles are inactive,
except for the ones which are placed at $\mathbf{o}$.  Active particles perform independent simple random walk on the tree of length   $ t \in \N \cup \{\infty \}  $, referred to as the particles' lifetime. When an active particle hits an inactive particle, the
latter becomes active. The model is often interpreted as a model for a spread of an epidemic. As such, it is natural to investigate whether the entire population is eventually infected, and if so, how quickly does this happen. Let $\mathcal{R}_t$ be the set of vertices which are visited by the process (with lifetime $t$). The susceptibility $\SS(\cT_{d,n}):=\inf \{t:\mathcal{R}_t=\VV_{d,n}  \} $ is the minimal lifetime required for the process to visit all sites. The cover time $\mathrm{CT}(\cT_{d,n})$ is the first time by which every vertex was visited at least once, when we take $t=\infty$.  We show that there exist absolute constants $c,C>0$ such that for all $d \ge 2$ and all  $\lambda = \la_n >0$ which does not diverge nor vanish too rapidly as a function of $n$,  with high probability $c \le \lambda\SS(\cT_{d,n}) /[n\log (n / \la )] \le C$ and $\mathrm{CT}(\cT_{d,n})\le  3^{4\sqrt{ \log |\VV_{d,n}| }} $. 
\\[.3cm]
\end{abstract}

\section{Introduction and results}
\label{intro} We study a system of  random walks known
as the \emph{frog model}. The frog model on infinite graphs has received much attention, e.g.~\cite{telcs1999branching,alves2002phase,alves2002shape,popov2001frogs,hoffman,hoffman2,01}. The aim of this paper is to launch the investigation of the model on finite graphs with an emphasis on finite trees.\footnote{In the year following the time at which the first draft of this paper was posted online, two other papers concerning the frog model on finite graphs were posted online. The first is \cite{frogs2}, which is some sense a continuation of this work. The base graphs considered in \cite{frogs2} are (1) $d$-dim tori of side length $n$ for all $d \ge 1$ and (2) expanders. The second is \cite{hoffman3} which contains some related results to our main results, which we shall discuss in more detail.} 

In this paper we study the model in the case that the underlying graph $\GG=(\VV,\EE)$ is some finite connected simple undirected graph. More specifically, we focus mainly on the case that $\GG$ is the finite $d$-ary tree of depth $n$, denoted by $\cT_{d,n}=(\VV_{d,n},\EE_{d,n})$. We employ the convention that the root, which throughout is denoted by $\mathbf{r}$  has degree $d$ (while the rest of the non-leaf vertices are of degree $d+1$). As we soon explain in more detail, we study some natural parameters associated with the frog model on finite graphs, which are meaningless in the infinite setup.

\medskip

 The frog model on $\GG$ with density $\la$  can be described as
follows.  Initially there are Pois($\lambda$)  particles at
each vertex of  $
\GG $, independently (where Pois($\la$) is the Poisson distribution with mean $\la$).\footnote{One can also consider different initial configurations of particles. See \cite{stochastic} for stochastic comparison between different random initial configurations.} A site of $ \GG $ is singled out and called its \emph{origin}, denoted by $\oo$ (when $\G$ is a tree, we do not assume that the origin is the root of the tree $\rr$). An additional particle, denoted by $w_{\mathrm{plant}}$, is planted at $\oo$. This is done in order to ensure that the process does not instantly die out. All
particles are inactive (sleeping) at time zero, except for those occupying the origin. Each active particle performs a
discrete-time simple random walk (SRW) of length $\tau $ (for some $\tau \in \N \cup \{\infty \} $) on the vertices of $ \GG $ (i.e.~at each step, it moves to a random neighbor of its current position, chosen from the uniform distribution over the neighbor set) after which it cannot become reactivated.
We refer to $\tau$ as the particles' \emph{lifetime}. Up to the time a particle dies (i.e.~during the $\tau$ steps of its walk), it activates all
sleeping particles it hits along its way. From the moment an
inactive particle is activated, it performs
the same dynamics over its lifetime $ \tau $, independently of
everything else. We denote the corresponding probability measure by $\PP_{\la}$. Note that there is no interaction between active particles, which
means that, once activated, each active particle moves independently of everything
else. 

\medskip

We now define two natural parameters for the frog model on a finite graph $\GG$.
Note that (in contrast to the setup in which $\GG$ is infinite) $\as$ there exists a finite minimal lifetime $\tau$ (which is a function of the initial configuration of the particles and the walks they pick) for which every vertex is visited by an active particle before the process ``dies out". We define this lifetime as the \emph{susceptibility} $\mathcal{S}(\GG)$. Another interesting quantity is the \emph{cover time} $\mathrm{CT}(\GG)$, defined as the minimal time by which every vertex is visited by at least one active particle when $\tau=\infty$ (i.e.\ when particles never die). More explicit definitions of the susceptibility and the cover time are given in \eqref{e:SandCT}.   

\medskip

The  $A+B \to 2B$ family of models, often
motivated as models for the spread of a rumor or infection,  are  defined by the following rule: there are type $A$ and $B$ particles occupying a graph $G$, say with densities $\la_A,\la_B>0$. They perform independent either discrete-time SRWs with holding probabilities $p_A,p_B \in [0,1]$ or continuous-time SRWs with rates $r_A,r_B \ge 0$  (possibly depending on the type). When a type $B$ particle collides with a type $A$ particle, the latter transforms into a type $B$ particle.
The frog model, whose name was coined in 1996 by
Rick Durrett, can be considered as a particular case of  the above dynamics in which the type $A$ particles are immobile ($p_A=1$ or $r_A=0$). Keeping the aforementioned interpretation of the model in mind, the cover time  and the susceptibility are indeed natural quantities. The former is roughly the minimal time by which all individuals are infected and the latter is the minimal lifetime $\tau$ of an individual infected by a virus,   sufficient for wiping out the entire population.
In this interpretation, the more likely  $\SS(\GG)$ is to be small, the more susceptible the population is.

 In a series of papers Kesten and Sidoravicius \cite{kesten2005spread,kesten2006,kesten2008shape} studied (in the continuous-time setup) the set of vertices visited by time $t$  by a type $B$ particle in the   $A+B \to 2B$  process when the underlying graph is the $d$-dimensional Euclidean lattice $\Z^d$,  $r_A,r_B>0$ and initially there are $B$ particles only at the origin. In particular, they proved a shape theorem for this set when $r_A=r_B$ and $\la_A=\la_B $ \cite{kesten2008shape} (and derived bounds on its growth in the general case \cite{kesten2005spread}). An analogous shape theorem for the frog model on $\Z^d$ was
proven by Alves, Machado, and Popov in discrete-time \cite{alves2002shape,alves3} and by Ram\' irez
and Sidoravicius in continuous-time \cite{RS}.

\medskip

 Most of the literature on the model is focused on the case that the underlying graph on which the particles perform their random walks is $
\Z^d$ for some $d \ge1$, e.g.\ \cite{alves2002phase,alves2002shape,popov2001frogs,RS} (in \cite{GS,dobler,dobler2017recurrence,ghosh,rosen} the case that the particles preform walks with a drift is considered). Beyond the Euclidean setup, there has been much interest in understanding the behavior of the model in the case that the underlying graph is a $d$-ary tree, either finite or infinite (we denote the infinite $d$-ary tree by $\TT$). More specifically, Itai Benjamini asked about the cover time and susceptibility of $\cT_{d,n}$ (see Problems \ref{o:1} and \ref{o:2} below) and about the existence of a phase transition for the model on $\TT$, as the density of walkers varies.  Until recently there was no progress in neither the finite nor infinite setups, which earned the problem its reputation as a hard problem. Finally, recently Hoffman, Johnson and Junge in a sequence of dramatic papers \cite{hoffman,hoffman2,johnson,hoffman3} showed that the frog model on $\TT$ indeed exhibits a phase transition as the initial state
becomes more saturated with particles. Namely, below a critical density of particles it is $\as$ transient, and above it is $\as$ recurrent. Johnson and Junge \cite{johnson} showed that the critical density grows linearly in $d$. We believe that the frog model should exhibit such a phase transition for all non-amenable graphs. See Conjecture \ref{con:nonam}.

Very recently, Hoffman et al.\ \cite{hoffman} showed that for $\la \ge Cd^2$ the frog model on $\TT$ is strongly recurrent (i.e.\ that the occupation measure of the origin at even times stochastically dominates some homogeneous Poisson process). As an application they proved sharp bounds (up to a constant factor) on $\CT(\cT_{d,n})  $ for $\la \ge Cd^2$. We discuss their results in more details later on. 

The purpose of this paper is to study the frog model on finite $d$-ary trees. We study two problems presented to us by I.~Benjamini \cite{Benjaminipc} (see also \cite[Open Question 5]{hoffman2}). In these two problems we seek an estimate which holds uniformly in the identity of the origin $\oo$.   
\begin{prob}
\label{o:1}
Show that for every $d \ge 2$  when $\la=1$ there exist some $C_{d},\ell >0$ such that \[\lim_{n \to \infty} \PP[\SS(\cT_{d,n})>C_d n^{\ell} ] = 0. \]
\end{prob}
\begin{prob}
\label{o:2} Show that for every $d \ge 2$  and $\la>0$ there exists some $f_{d,\la}:\N \to \N$ satisfying that\footnote{We write $o(1)$ for terms which vanish as $n \to \infty$ (or some other index, which is clear from context). We write $f_n=o(g_n)$ or $f_n \ll g_n$ if $f_n/g_n=o(1)$. We write $f_n=O(g_n)$ and $f_n \precsim g_n $ (and also $g_n=\Omega(f_n)$ and $g_n \succsim f_n$) if there exists a constant $C>0$ such that $|f_n| \le C |g_n|$ for all $n$. We write  $f_n=\Theta(g_n)$ or $f_n \asymp g_n$ if  $f_n=O(g_n)$ and  $g_n=O(f_n)$. When we have a sequence of functions $f^{(n)}:\{a_{n},a_{n}+1,\ldots,b_n\} \to \N$ we often suppress the dependence on $n$ from the notation and write $f(i) =O(g_i) $ (similarly, $f(i) \asymp g_i$, etc.)   to indicate that there exists some $C>0$ such that for all $n$ and $a_{n} \le i \le b_{n}$ we have that $f^{(n)}(i) \le C g_i $ (resp.\  $(1/C)g_i \le f^{(n)}(i) \le C g_i $).} $f_{d,\la}(n)=o(n^{\eps})$ for every $\eps>0$ such that \[\lim_{n \to \infty} \PP_{\la}[\CT(\cT_{d,n}) \le f_{d,\la} (|\cT_{d,n}|) ] = 1. \]
\end{prob}
Deterministically, $\CT(\cT_{d,n}) \ge n $. Conversely, $\CT(\cT_{d,n})$ can be bounded from above by the cover time of $\cT_{d,n} $ by a single SRW, which Aldous \cite{aldouscover} showed is $\asymp n^{2}d^n \log d $. To the best of the author's knowledge, the best previously known upper bound on $\CT(\cT_{d,n})$ is exponential in $n$, i.e.~polynomial in the volume of the tree, $|\cT_{d,n}| $ (see e.g.~the paragraph preceding Open Question 5 in \cite{hoffman2}). In light of the aforementioned phase transition it seems plausible that the answer to both Problem \ref{o:1} and \ref{o:2} may depend on $\la$. Our main result, Theorem \ref{thm:1}, resolves Problem \ref{o:1} by determining for all fixed $\la>0$ the (``typical" w.r.t.~$\PP_\la$) value of $\SS(\cT_{d,n})$ up to a constant factor. In particular, it asserts that  $\SS(\cT_{d,n})$ does not exhibit a phase transition w.r.t.~$\la$.  Theorem \ref{thm:2} provides an affirmative answer
to Problem \ref{o:2}. All estimates in Theorems \ref{thm:1} and \ref{thm:2} hold uniformly in the identity of the origin  $\oo $.
\begin{maintheorem}
\label{thm:1}
There exist some absolute constants $C,c>0$ such that for all $\la>0$ and $d \ge 2$  \[ \lim_{n \to \infty} \PP_{\la}[c \le \la \SS(\cT_{d,n})/(n \log n) \le C ] = 1. \]
In fact, as long as $ \la_n=o(n) $ and $\la_n \ge d^{-n}  n \log d $ we have that
\[ \lim_{n \to \infty} \PP_{\la_{n}}[   \SS(\cT_{d,n}) \ge c \sfrac{n}{\la_n}  \log \sfrac{n}{\la_n}] = 1. \]
Conversely, there exist some $c',C>0$ such that for all $d \ge 2$, if $ d^{-n}  n^2 \log d \le \la_{n} \le c' \log n $ then
\[ \lim_{n \to \infty} \PP_{\la_{n}}[   \SS(\cT_{d,n}) \le C \sfrac{n}{\la_n}  \log \sfrac{n}{\la_n}] = 1. \] 
\end{maintheorem}
\begin{rem}
We believe that the condition that  $\la_{n} \le c'\log n $ in the last equation can be weakened to the condition that $ \la_n =o( n)$. Only minor adaptations to the proof are necessary in order to weaken the condition that  $\la_n \ge d^{-n}  n^2 \log d $ to $ \la_n \ge d^{-n}  n \log d$.
\end{rem}
\begin{maintheorem}
\label{thm:2}
There exists some $C>0$ such that for all $\la_{n}  \ge d^{-n}n^2 \log d$ and $d \ge 2$ \[ \lim_{n \to \infty} \PP_{\la_{n}}[  \CT(\cT_{d,n})\le Cn  \max \{1, \sfrac{1}{\la_n}  \log \sfrac{n}{\la_n}\}  3^{3\sqrt{ \log |\VV_{d,n}|  }} ] = 1. \]  
\end{maintheorem}
\begin{rem}
It follows from Theorem \ref{thm:2} that if $  \log  \max \{ \la_n^{-1},1 \} =O(\sqrt{ \log |\VV_{d,n}|}) $ then $\whp$\footnote{We say that a sequence of events $A_n$ defined with respect to some probabilistic model on a sequence of graphs $G_n:=(V_n,E_n) $ with $|V_n| \to \infty $ holds $\whp$ (``with high probability")  if the probability of $A_n$ tends to 1 as $n \to \infty$. } $  \log \CT(\cT_{d,n}) \le O(\sqrt{ \log |\VV_{d,n}|})$. 
\end{rem}  
It follows from the assertion of Theorem \ref{thm:1} that $\SS(\cT_{d,n}) $ does not exhibit a phase transition w.r.t.~$\la $. This is in contrast with the aforementioned results in \cite{hoffman,hoffman2,johnson,hoffman3} concerning $\TT$. Observe that the constants in Theorems \ref{thm:1} and \ref{thm:2} are independent of $d$ (however the term $\log |\VV_{d,n}|$ in the exponent in Theorem \ref{thm:2} obviously depends on $d$). This is somewhat surprising. To get the correct dependence (or more precisely, lack of dependence) on $d$ we obtain estimates on the transition probabilities $p^t(\bullet,\bullet) $ for SRW on $\cT_{d,n}$ which are sharp up to a constant factor, uniformly in $t$ (see Appendix \ref{s:aux}; In particular, Corollary \ref{c:mixing} asserts that the $L_{\infty}$ mixing time of the continuous-time SRW  on $\cT_{d,n}$ is $\asymp d^{n-1} \log d$, uniformly in $d$). We believe that these estimates are of self interest.
 
 In a previous version of this paper the author posed the following problem. 
 \begin{prob}
\label{o:4} Show that for every $d \ge 2$ there exist $\la_1(d) \le \la_2(d)$ and some $C_{d,\la},c_{d,\la},\eps_{d},\ell >0$ such that (regardless of the identity of the origin $\oo$) \[ \forall\, \la> \la_2(d), \quad \lim_{n \to \infty} \PP_{\la}[\CT(\cT_{d,n}) \le C_{d,\la} n^{\ell} ] = 1. \]
\begin{equation}
\label{e:p4r} \forall \la< \la_1(d), \quad \lim_{n \to \infty} \PP_{\la}[\CT(\cT_{d,n}) > c_{d,\la}2^{n^{\eps_{d}}} ] = 1.
\end{equation}
\end{prob}
It was suggested in that older version of this paper that $\la_2(d)$ might correspond to a value below which the occupation measure (for the frog model on $\TT$) converges to 0 pointwise, while above it, the joint distribution of the number of particles at each site converges to a non-trivial invariant measure. As mentioned earlier, the existence of a related phase transition was proved recently by Hoffman et al.\ \cite{hoffman3}. As an application, they showed that  $\whp$ $\CT(\cT_{d,n}) \le C_{d}\la^{-1} n \log n $ for $\la \ge C_{0}d^2$, while $\CT(\cT_{d,n}) \ge \exp(c_{d,\la}\sqrt{n} )$ for $\la \le d/100$, resolving Problem \ref{o:4}.\footnote{We use $C,C',C_0,C_1,\ldots$ (resp.~$\delta,\eps, c,c',c_0,c_1,\ldots $) to denote positive absolute constants which are sufficiently large (resp.~small) to ensure that a certain inequality holds. Similarly, we use $C_{d},C_{\la,d}$ (resp.~$c_{d},c_{\la,d}$) to refer to sufficiently large (resp.~small) positive constants, whose value depends on the parameters appearing in subscript. Different appearances of the same constant at different places may refer to different numeric values.} Note that by our Theorems \ref{thm:1} and \ref{thm:2}  for all fixed $\la$ we have that $\whp$ \[  \la^{-1} n \log n \le \CT(\cT_{d,n}) \le \exp(C\sqrt{n \log d} ),  \]
which by the results in \cite{hoffman3} is sharp up to a constant factor for $\la > C_0d^2$ and up to the value of the constant in the exponent for $\la \le d/100$. It is interesting to note that $\log \CT(\cT_{d,n})  $ is ($\whp$) of the same order for $\la=d/100 $ as it is for $\la = \la(n,d) =e^{- \sqrt{n \log d } }$.

\medskip

The following theorem offers two general lower bounds on the susceptibility of a regular graph. The second is defined in terms of $P$, the transition kernel of SRW on the underlying graph $G=(V,E)$ (which is clear from context).  
We denote the $t$-steps transition probability from $u$ to $v$ by $p^t(u,v):=P^t(u,v)$. \begin{maintheorem}
\label{thm:lower}
Let $L(n,\la):=\log n -4 \log \log n- \log ( \max \{\sfrac{1}{\la},1 \} ) $. For every finite regular simple graph\ $G=(V,E) $ and all $\la_{|V|}>0$ such that $ \lambda_{|V|}^{-1} ( \log |V|)^5 \le |V| $ we have that
\begin{equation}
\label{e:lower1}
\PP_{\la_{|V|}}[\la_{|V|} \SS(G) \ge L(|V|,\la_{|V|}) ] \to 1, \quad \text{as }|V| \to \infty.
\end{equation}
Moreover, for  all  $\delta \in (0,1)$ and all  $\la_{|V|}\gg |V|^{-\delta/7}   $  we have that
\begin{equation}
\label{e:lower2}
\PP_{\la_{|V|}}[\SS(G) \ge t_{\la_{|V|},\gd}(G) ] \to 1, \quad \text{as }|V| \to \infty,
\end{equation}
where  $t_{\la,\delta}(G):=\min \{s: \frac{2s \la}{\kappa_s}  \ge (1-\delta)\log |V| \} $ and
$\kappa_t:=\min_v \sum_{i=0}^{t}p^i(v,v)$.  
\end{maintheorem}
\begin{rem}
Theorem \ref{thm:lower} seems to be especially useful when $G$ is vertex-transitive (Definition \ref{def: VT}; see Conjecture \ref{con:VT} for more on this point). The bound offered by \eqref{e:lower2} is sharp up to a constant factor in the cases that $G$ is either a $d$-dimensional torus ($d \ge 1$) of side length $n$ or a regular expander \cite{frogs2}.
Proposition \ref{prop:lower} gives a certain extension of \eqref{e:lower2}\footnote{By allowing one to take the minimum in the definition of $\kappa_t$ only w.r.t.~the leaf set.} which yields the lower bound on $\SS(\cT_{d,n})$ in Theorem \ref{thm:1} when $\la_n \ge d^{-(\sfrac{1}{3}-\eps)n}  $ for some $\eps \in (0,\sfrac{1}{3}] $. 
\end{rem}
Observe that the requirement $\lambda_{|V|}^{-1} ( \log |V|)^5 \le |V|$ ensures that $L(|V|,\la_{|V|}) \ge  \log \log |V|$. Provided that $\la_{|V|} \ge |V|^{-o(1)} $, we have that $L(|V|,\la_{|V|}) \ge (1-o(1))\log |V| $. 
We conjecture that (for a regular graph $G$) as long as $|V | \la_{|V|} \to \infty $, if $f(n) \gg 1$ then \[\PP_{\la_{|V|}}[ \la_{|V|}\SS(G) \ge  \log |V| -f(|V|) ] \to 1, \quad \text{as }|V| \to \infty.\] 

The argument in the proof of Theorem \ref{thm:lower} is inspired from those of Theorems 2 and 4.4 in \cite{SN}. It is possible to extend Theorem \ref{thm:lower} to the case that $G$ is non-regular. Adapting the argument from Theorem 4.3 in \cite{SN} gives
\begin{equation}
\label{eq: lower'}
\forall \alpha \in (0,1), \exists c_{\alpha}>0, \quad \PP_{\la}[\la \SS(G) \le c_{\alpha} \log |V|  ] \le \exp [-c'_{\la} r_{*}^2|V|^{\alpha}] , 
\end{equation}
where $r_*:=\min_{u,v \in V}  \frac{\deg (u)}{\deg (v)} $. This is meaningful as long as $r_* \ge |V|^{\beta-0.5}$ for some $\beta >0$. We do not prove \eqref{eq: lower'}.\footnote{The details involved in the translation of (1.2) in \cite{SN} to \eqref{eq: lower'} above are similar to the ones involved in the translation of Theorems 2 and 4.4 in \cite{SN} to Theorem \ref{thm:lower} above.}

We note that for every regular graph $G=(V,E)$ we have that $t_{\la,0}(G) \le C \la^{-2}\log^2| \la V| $ (e.g.~\cite[Lemma 2.4]{PS}), which is tight up to a constant factor as can be seen by considering the $n$-cycle, $\mathrm{C}_n $, for which $t_{\la,1/2}(\mathrm{C}_n) \ge c \la^{-2}\log^2| \la V| $. Theorem 1 in \cite{frogs2} asserts  that there exist $c_1,C_1>0$ such that    $\lim_{n \to \infty} \PP_{\la}[ c_{1} \le  \la^2 \SS(\mathrm{C}_n) /\log^{2} n \le C_{1}]=1$, for all fixed $\la>0$ (see \S~\ref{s:related} for more details on this point). We conjecture that the cycle is in some sense extremal (up to degree dependence). For further details see \S~3.3 in \cite{frogs2}. 

\begin{prop}
\label{prop:Kn}
Let $K_n$ be the complete graph on $n$ vertices. Let $(\la_n)_{n \in \N}$ be such that $\lim_{n \to \infty} \la_n n = \infty $. Then
\begin{equation}
\label{e:Knintro1}
\forall \eps \in (0,1), \quad \lim_{n \to \infty} \PP_{\la_{n}}[(1-\eps)\la_{n}^{-1}\log n  \le  \SS(K_n) \le \lceil (1+\eps)\la_{n}^{-1}\log n \rceil ]=1. \end{equation}
Moreover, there exists some $C>1 $ such that for every $(\la_n)_{n \in \N}$ such that $ \lim \sup_{n \to \infty } \la_n \le 2 $ 
\begin{equation}
\label{e:rem1.2}   \lim_{n \to \infty} \PP_{\la_{n}}[   \CT(K_n) \le C  \la_{n}^{-1} \log n ]=1,
  \end{equation}
while if $ \liminf_{n \to \infty} \la_n \ge 2 $ then  
\begin{equation}
\label{e:rem1.22}   \lim_{n \to \infty} \PP_{\la_{n}}[   \CT(K_n) \le C \lceil  \log_{\la_{n}} n \rceil]=1.
  \end{equation}
\end{prop}
In light of Theorem \ref{thm:lower} and Proposition \ref{prop:Kn}, it follows that $K_n$ is the regular graph with asymptotically the smallest susceptibility. This is consistent with the fact that
 $K_n$ is the regular graph with asymptotically the smallest cover time \cite{feige} and also the asymptotically smallest social connectivity time in the random walks social network model \cite{SN} (see \S~\ref{s:related}).

Observe that the dependence of $\SS(K_n)$ on $\la_n^{-1}$  is linear. As the number of particles is concentrated around $\la_n n$ (provided that $\la_n n \gg 1 $) it follows from \eqref{e:Knintro1} that the total number of steps performed by the particles when their lifetime is taken to be $\SS(K_n)$,  is roughly $n \log n$, the coupon collector time. 
\subsection{Organization of the paper}
In \S~\ref{s:construction} we present an explicit construction of the frog model and give more detailed definitions of $\SS(G)$ and $\CT(G)$. In \S~\ref{s:keySRW} we present two key auxiliary results concerning SRW on $\cT_{d,n}$, the $d$-ary tree of depth $n$. In \S~\ref{s:thm1}-\ref{s:thm2} we prove Theorems \ref{thm:1}-\ref{thm:2}, respectively. In \S~\ref{s:lower} we prove our general lower bounds on the susceptibility (Theorem \ref{thm:lower}). In \S~\ref{s:lowersustree} we prove the lower bound on $\SS(\cT_{d,n}) $ from Theorem \ref{thm:1}. In \S~\ref{s:Kn} we analyze the model on the complete graph and prove Proposition \ref{prop:Kn}. In \S~\ref{s:aux} we derive sharp heat kernel and hitting time estimates for SRW on $\cT_{d,n}$, which are then used in order to prove the two key results from \S~\ref{s:keySRW}.   
\subsection{Related models and further questions}
\label{s:related}

\medskip

In \cite{SN} Benjamini and the author study the following model for a social network, called the random walks social network model, or for short, the \emph{SN model}. Given a graph $G=(V,E)$,  consider Poisson($|V|$) walkers performing independent lazy simple random walks on $G$ simultaneously, where the initial position of each walker is chosen independently w.p.\ proportional to the degrees.
When two walkers visit the same vertex at the same time they are declared to be
acquainted. The social connectivity time, $\SC(G)$, is defined as the first time by which there
is a path of acquaintances between every pair of walkers.  The main result in  \cite{SN} is\  that when the maximal degree of $G$ is $d$, then $c \log |V| \le \SC(G) \le C_d \log^3 |V|$ $\whp $. Moreover, $\SC(G)$ is determined up to a constant factor in the cases that $G$ is a regular expander or a $d$-dimensional torus ($d \ge 1$) of side length $n$. 

Note that in the SN model all walkers are initially activated and we are not requiring the sequence of times at which some walkers met along a certain path of acquaintances to be non-decreasing. Hence one should expect the SN model to evolve much faster than the frog model, in a sense that (when $\la$  is fixed), for many graphs $\mathbb{E}[\SC(G) ] \ll\mathbb{E}[\CT(G) ]$ (one reason is the fact that  $\CT(G) \ge \mathrm{diameter}(G)$). We note that this fails when $G$ is either the complete graph, or a regular expander. In these cases both terms are of order $\log |V|$ (the expander case is in fact quite involved).

However, in many examples, even when  $\mathbb{E}[\SC(G) ] \ll\mathbb{E}[\CT(G) ]$, it is still the case that   $\mathbb{E}[\SC(G) ]$ and $\mathbb{E}[\SS(G) ]$ are of the same order, and several techniques from \cite{SN} can be applied successfully to the frog model. Namely, the same technique used in  \cite{SN}   to prove general lower bounds on $\SC(G)$ are used in the proof of Theorem \ref{thm:lower}. Moreover, the analysis of the two models on  expanders and on $d$-dimensional tori ($d \ge 1$) are similar (in all of these cases  $\SS(G)$ and  $\SC(G)$ are $\whp$ of the same order, \cite{frogs2,SN}).

In light of the above discussion the following conjecture is natural.
\begin{con}[Benjamini]\label{con:polylog} There exist some $C_{d,\la},\ell>0$ such that for every sequence of finite connected graphs $G_{n}=(V_{n},E_{n})$ with $|V_n| \to \infty$ we have that $\lim_{n \to \infty} \PP_{\la}[\SS(G_n) \le C_{d,\la}  \log^{\ell} |V_n| ]=1$.
\end{con}  
We suspect that one can take above $\ell=2$ and $C_{\la,d}=C \la^{-2}d^2 $ for some absolute constant $C>0$. Moreover, we suspect that for regular or vertex-transitive  $G$, one can even take above, respectively, $C_{\la,d}=C \la^{-2}d$ or  $C_{\la,d}=C \la^{-2}$ for some absolute constant $C>0$ (cf.~\cite[Conjecture 8.3]{SN}). 

 \begin{defn}
\label{def: VT}
We say that a bijection $\phi:V \to V$ is an automorphism
of a graph $G=(V,E)$ if $\{u , v \} \in E $ iff $\{ \phi(u) , \phi(v)\} \in E$. A graph $G$ is said to be \emph{vertex-transitive} if the action of its automorphisms group, $\mathrm{Aut}(G)$, on its vertices is transitive (i.e.~$\{\phi(v):\phi \in \mathrm{Aut}(G) \}=V $ for all $v$).
\end{defn}
\begin{con}
\label{con:VT}
Let  $G=(V,E)$  be a finite connected vertex-transitive graph.
Consider the cover time of $G$ by $m$ particles performing independent SRWs, each starting at a random initial position chosen from the uniform distribution on $V$ independently (where the cover time is defined as the first time by which every vertex is visited by at least one of the particles). Denote it by $\mathrm{Cov}_{m} $.  Then as long as $ \la |V| \gg 1  $ we have that \[\mathbb{E}_{\la}[\SS(G)] \asymp \mathbb{E}[\mathrm{Cov}_{\lceil \la |V| \rceil}] \asymp t_{\la}(G), \]
where  $t_{\la}(G):=\min \{s:2s/\kappa_s  \ge \la^{-1} \log |V| \} $ and
$\kappa_t:=\min_v \sum_{i=0}^{t}p^i(v,v)$. 

Finally, we conjecture that in the above setup $\SS(G) $ is concentrated around its mean (i.e.\ for a sequence of finite connected vertex-transitive graphs of diverging sizes $G_n=(V_n,E_n)$ and $\la_n \gg \sfrac{1}{|V_n|}  $\math $,  for all $\eps>0$ we have that $\mathbb{P}_{\la_n}[|\sfrac{\SS(G_{n})}{\mathbb{E}_{\la_n}[\SS(G_{n})]} -1| \ge \eps ] \to 0 $ as $n \to \infty$).
\end{con}
\begin{rem}
Recall that $\mathbb{E}_{\la}[\SS(G)] \succsim  t_{\la}(G)  $ by Theorem \ref{thm:lower} (when $\la$ does not tend to 0 too rapidly). We note that under transitivity the inequality $\mathbb{E}_{\la}[\SS(G)] \succsim \mathbb{E}[\mathrm{Cov}_{2\lceil \la |V| \rceil}] $ is fairly easy.
\end{rem}
In the $A+B \to 2B$ model, even when $p_A<1$, one can consider the case in which the $B$ particles have lifetime $t$ and initially only the particles at some vertex $\oo $ are of type $B$ (where an additional $B$ particle is planted at $\oo$). One can then define the susceptibility in an analogous manner, where the corresponding quantity is the minimal lifetime of a $B$ particle required so that every vertex is visited by a type $B$ particle. Alternatively, one can consider the minimal lifetime required so that all particles are transformed into part $B$ particles before the process dies out. We strongly believe that all of the results presented in this paper can be transferred into parallel results for the case $p_A<1$ (recall that $p_A=1$ is the frog model). Moreover, we also believe that the corresponding versions of  Conjectures \ref{con:polylog} and \ref{con:VT} are true also in the case of $p_A<1$.

In \cite{SN2} it is shown that the SN model on infinite graphs exhibits the following dichotomy. For transitive amenable graphs for every particle density a.s.\ every pair of particles eventually have a path of acquaintances between them. For non-amenable graphs of bounded degree there is always a (strictly positive) critical particle density above which the aforementioned behavior holds a.s. and below which it a.s.\ fails.
\begin{con}
\label{con:nonam}
Let $G$ be a connected non-amenable graph of bounded degree. Then there exists some $\la_{\mathrm{c}}>0$ such that for all $\la > \la_{\mathrm{c}}$ (resp.\ $\la< \la_{\mathrm{c}} $) the frog model on $G$ is $\mathbb{P}_{\la}$ a.s.\ recurrent (resp.\ transient). 
\end{con} 

\subsection{Notation}
\label{s:notation}

Let $G=(V,E)$ be some finite graph. For SRW on a graph $G$, the \emph{hitting time} of a set $A \subset V$ is $T_A:=\inf \{t \ge 0 :X_t \in A \}$. Similarly, $T_A^+:=\inf \{t \ge 1 : X_t \in A \} $. When $A=\{x\}$ is a singleton, we instead write $T_x$ and $T_x^+$.  Let $P$ be the transition kernel of SRW on $G$. We denote by $p^t(u,v):=P^t(u,v)$ the $t$-steps transition probability from $u$ to $v$. We denote by $\Pr_u$ the law of the entire walk, started from vertex $u $. We denote the stationary distribution by $\pi$. The same definitions apply when we consider an arbitrary Markov chain on a finite or countable state space, rather than SRW.

\medskip

Consider the frog model on $G=(V,E)$ with particle density $\la$ and lifetime $\tau$. We denote the collection of vertices which are visited by an active particle before the process corresponding to lifetime $\tau$ dies out by  $\RR_{\tau}$.  Denote the collection of particles whose initial position belongs to a set $U \subseteq V $ (resp.~is $v \in V$) by $\W_U$ (resp.~$\W_v$). Then $(|\W_v|-1_{v = \oo})_{v:v \in V }$ are i.i.d.~$\Pois(\la)$. Denote the collection of all particles by $\W=\W(V)$. Denote the range of the length $\ell$ walk picked (in the sense of \S~\ref{s:construction}) by a particle $w$  by $\RR_{\ell}(w)$. We denote the union of the ranges of the length $\ell$ walks picked  by the particles belonging to some set of particles $\mathcal{U} \subseteq \W $ by $\RR_{\ell}(\mathcal{U}) $. Denote the union of the ranges of the length $\ell$ walks picked  by the particles whose initial positions lie in $U \subseteq V $ (resp.~is $v$) by $\RR_{\ell}(U):=\RR_{\ell}( \W_U)$ (resp.~$\RR_{\ell}(v):=\RR_{\ell}( \W_v)$). For every event $A$ we denote its complement by $A^c$.  

\medskip

The distance $ \dist(x,y)$ between vertices~$ x
$ and~$ y $ is the number of edges along a shortest path from $ x $ to $ y $. Vertices
are said to be neighbors if they belong to a common edge. We denote $[k]:=\{1,2,\ldots,k \}$ and $]k[:=\{0,1,\ldots,k \}$. We denote the cardinality of a set $A$ by $|A|$. By abuse of notation, for a finite tree $\cT=(V,E)$ we write $|\cT|$ and $v \in \cT$ instead of $|V|$ and $v \in V$, respectively. 
We write \emph{w.p.}~as a shorthand for ``with probability".

\medskip

  We denote the $d$-ary tree of depth $n$ by  $\mathcal{T}_{d,n}=(\VV_{d,n},\EE_{d,n})$ (when $d$ and $n$ are clear from context we write  $(\VV,\EE)$). We denote its root by $\mathbf{r}$. Denote by $\cL_i$ the $i$th level of the tree and the leaf set by $\cL$. For $x \in \VV_{d,n}$ we denote the distance of $x$ from the root and from the leaf set, respectively, by \[|x|:=\dist(x,\mathbf{r}) \text{ (i.e.~$x \in \cL_{|x|}$) } \quad \text{and}\quad \| x \|:=n-|x|=\dist(x,\cL). \]  
The root induces the following partial order, $\le$, on $\VV_{d,n}$.
We write $x \le y$ if the (unique) path from $y$ to the root goes through $x$. If $x \le y$ we say that $x$ is the $\dist(x,y)$th \emph{ancestor} of $y$. If $x \le y$ and $\dist(x,y)=1$, we say that $x$ is the \emph{parent} of $y$. We denote the $i$th ancestor of a vertex $x$ by $\overleftarrow{x_i} $. For  $x,y \in \VV_{d,n}$ we denote by $x \wedge y $ their first common ancestor (the vertex $z$ such that $z \le x$, $z \le y$ and $|z|$ is maximal). The \emph{induced tree} at $x $, denoted by $\mathcal{T}_{x}$, is the induced tree on the set of $\{y:x \le y \}$. For $x \in \VV_{d,n}$ and $| \mathcal{T}_{x}| \le t \le |\VV_{d,n}| $ we denote by $\cT_x(t)$ the smallest induced tree $\cT_y $ such that $y \le x$ and $|\cT_y| \ge t$. Note that $\cT_x(t) $ is uniquely defined!   
We reserve the term ``induced tree" for trees of the form $\cT_x$ as these are the only type of subtrees we shall consider.

\medskip

Apart from $\cT_{d,n}$ we also consider some of its induced trees and the infinite $d$-regular tree $\TT:=(V(\TT),E(\TT)) $. When referring to some tree $\cT$ other than $\cT_{d,n}$, we write $\cL_i(\mathcal{T})$ and  $\cL(\mathcal{T})$ for the $i$th level and the leaf set of $\mathcal{T} $, respectively.  We denote the law of SRW on a tree $\cT$ (other than $\cT_{d,n}$) started from vertex $v $ by $\Pr_v^{\cT}$ and similarly, denote the transition probabilities of such a walk by $P_{\cT}^t(\cdot,\cdot) $.

\section{A formal construction of the model}
\label{s:construction}
In this section we present a formal  construction of the frog model. In particular  the susceptibility and cover time are defined explicitly in \eqref{e:SandCT}. Also, in what comes we shall frequently refer to ``the walk picked by a certain particle". This notion is explained in the below construction. We also recall the notion of Poisson thinning, which is used repeatedly throughout the paper.

Clearly, in order for the susceptibility to be a random variable, the probability space should support the model simultaneously for all particle lifetimes. In order to establish the fact   that the laws of $\SS(\GG)$ and $\CT(\G) $ are stochastically decreasing in $\la$,
 below we show that it can be taken to  support the model simultaneously also for all particle densities. As this is a fairly standard construction, most readers may wish to skim this subsection. 

 We denote the set of  $\Pois(\la)$ (or $1+\Pois(\la)$ for the origin)  frogs occupying vertex $v$ at time 0 by $\W_{v}=\{w_1^v,\ldots w_{|\W_v|}^v \}$. We can assume that at time 0 there are infinitely many particles occupying each site $v$, $\mathcal{J}_v:=\{w_i^v:i \in \N \}$ (where $w_i^v$ is referred to as the $i$-th particle at $v$), but only the first $|\W_{v}|$ of them are actually involved in the dynamics of the model. We may think of each particle $w_i^v \in \mathcal{J}_v $ as first \emph{picking}  an infinite SRW $(S_t^{v,i})_{t \in \Z_+}$ according to $\Pr_v$ and only in the case that $i \le |\W_v| $  and $v $ is visited by some active particle, say at time $s$ (for the first time), does $w_i^v$ actually performs the first $\tau$ steps of the SRW it picked (i.e.~its position at time $s+t$ is $S_t^{v,i}$ for all $t \in ] \tau [$).  

\medskip

Let $G=(V,E)$ be a graph.  A \emph{walk} of length $k$ in $G$ is a sequence of $k+1$ vertices $(v_0,v_1,\ldots,v_{k})$ such that for all $0 \le i <k$ or $\{ v_{i} , v_{i+1}\} \in E $. Let $\Gamma_k$ be the collection of all walks of length $k$ in $G$. We say that $w_i^v \in \W_v$ \emph{picked} the path  $\gamma=(\gamma_0,\ldots,\gamma_k) \in \Gamma_{k}$ if $S_t^{i,v}=\gamma_t $ for all $t \in ]k[ $. For each $\gamma \in \Gamma_k $ let $X_{\gamma}$ denote the number of particles in $\W_{\gamma_0} $, other than the particle planted at the origin, which picked the walk $\gamma$.  For a walk $\gamma=(\gamma_0,\ldots,\gamma_k) \in \Gamma_{k}$ for some $k \ge 1$, we denote $p(\gamma):=\prod_{i=0}^{k-1}P(\gamma_i,\gamma_{i+1}) $. By Poisson thinning, we have that for every fixed $k$, the joint distribution of $(X_{\gamma} )_{\gamma \in \Gamma_k}$ (under $\PP_\la$) is that of independent Poisson random variables with $\mathbb{E}_{\la}[X_{\gamma}]=\la p(\gamma)$ for all $\gamma \in \Gamma_k$. 

\medskip

 Consider a collection   $((M_v(s))_{s \ge 0})_{v \in \VV}$  of homogeneous Poisson processes on $\R_{+}$ with rate $1$ and a collection of simple random walks on  $ \GG $,  $ \{ (S^{v,i}_t)_{t \in \Z_+}:v \in \VV, i \in \mathbb{N} \} $, where for all $i $ and $x$, the walk picked by the $i$th particle in $\mathcal{J}_x$ is $ (S^{x,i}_t)_{t \in \Z_+}$.  We take the walks and the Poisson processes to be jointly independent.  We take $|\W_u^{\la}| :=M_u(\la)+1_{u=\oo}$, where $\W_u^{\la}$ denotes the set of particles involved in the dynamics, whose initial position is $u$, when the density is taken to be $\la$. We take $\W_u^{\la}$ to be the collection of the first $|\W_u^{\la}|$ particles in  $\mathcal{J}_u$.    When clear from context, we omit the superscript $\la$ and write $\W_u$.  From this construction it is clear that the laws of $\SS(\GG)$ and $\CT(\G) $ are stochastically decreasing in $\la$ (since if $\la < \la' $, then for all $u$ we have that $ \W_u^{\la} \subseteq \W_u^{\la'} $).

  For $ x,y \in \VV(\GG) $ such that $ x\neq y $ and $\tau \in \N \cup \{\infty \}$ let
\begin{equation}
\label{e:elltau}
 \ell_{\tau}(x,y) := \inf\{j \le \tau : S_j^{x,i} = y \quad \text{for some} \quad i \le |\W_x| \}
\end{equation}
 (where $\inf \eset = \infty $). The \emph{activation time} of~$x$ (and also of $\W_x$) w.r.t.~lifetime $\tau$ is 
\begin{equation}
\label{e:AT}
\mathrm{AT}_{\tau}(x) := \inf\{\ell_{\tau}(x_0,x_1)+\cdots+\ell_{\tau}(x_{m-1},x_m)\},
\end{equation}
where the infimum is taken over all finite sequences $x_0=\oo, x_1,
\ldots, x_{m-1},x_m=x$ such that $ x_i \in \VV(\GG) $ for all $i$. The event
$\mathrm{AT}_{\tau}(x)=\infty$ is precisely the event that (for lifetime $\tau$) site~$x$ is never visited by an
active particle, while when finite, $\mathrm{AT}_{\tau}(x)$ is the first time at which vertex $x$ is visited by an active particle (for lifetime $\tau$).  The \emph{cover time} and \emph{susceptibility} of $\GG$ can be defined as
\begin{equation}
\label{e:SandCT}
\mathrm{CT}(\GG):=\max_{v \in \VV}\mathrm{AT}_{\infty}(v) \text{ and }\SS(\GG):=\inf \{\tau:\max_{v \in \VV}\mathrm{AT}_{\tau}(v)<\infty \}. 
\end{equation}
\section{Key auxiliary results concerning SRW on $d$-ary trees}
\label{s:keySRW}
In this section we state two key auxiliary results concerning SRW on $\cT_{d,n} $. Recall that $\|y\|$ is the distance of $y$ from the leaf set $\cL$. Recall that $x \wedge y$ is the common ancestor of $x$ and $y$ which is closest to $\cL$. Recall that for a tree $\cT$ its leaf set is denoted by $\cL(\cT)$. Recall that  $\cT_x(m)$ the smallest induced tree $\cT_y $ such that $y \le x$ and $|\cT_y| \ge m$.  Observe that if $\|y\| \le k $ then $\cT_y(d^k)$ is of depth $k$.

Part (i) of the following proposition assets that for every set  $A$ of size at least $\Omega (d^k) $ which is a subset of the leaf set of some fixed induced subtree (of $\cT_{d,n}$) of depth $k \le n $,  the probability that a walk of length $t \in  [d^{k-1},d^{k}] $  with initial state inside this subtree will terminate at this induced subtree and will visit $\Omega ( t/\log_d (dt)) $ vertices of $A$ is  bounded from below uniformly in $k,t,A$ and in the initial position of the walk inside the induced subtree.
\begin{prop}
\label{cor:range}
 Let $R_t:=\{X_i : 0 \le i \le t \}$. Let $k \le n$. Let $t \in [d^{k-1},d^{k}]$.  Denote \[g(t):= t/\log_d (dt). \] There exist $c,c_{0},c'>0$ such that the following hold for all $d \ge 2$ and $n \ge M$.
\begin{itemize}
\item[(i)]

For every $\delta \in (0,1)$, $y \in \VV_{d,n}$, $t \in [d^{k-1},d^{k}] $ such that   \[d^{\|y\|-1}+8\|y\| \le t \le d^{n} \] and $A \subseteq \cL(\cT_y(d^{k}))$ of size $|A| \ge \delta d^{k} $, for all $z \in \cT_y$ we have that
\begin{equation}
 \label{eq: Range3}
\begin{split}&  \Pr_z[|R_{t} \cap A |>c' \delta g(t) \text{ and }X_{t} \in \cT_y(t)  ] \ge c\delta^{2}.
 \end{split}
\end{equation}
\item[(ii)]
Let $k \le i \le n  $ and  $x,y  \in \cL_{n-i}$ be such that $32 \|x \wedge y\| \le d^{k}$. Let $t \in [32 \|x \wedge y\|,d^{k}]$.  Let $\cT_{y,k}^1,\ldots,\cT_{y,k}^{d^{i-k}} $ be the collection of all induced subtrees of $\cT_y$ of depth $k$. Then for all $z\in \cT_{x}$ we have that 
\begin{equation}
 \label{eq: Range3'}
 \begin{split}
& \Pr_z[\max_{\ell \in [d^{i-k}] } |R_{t} \cap \cT_{y,k}^{\ell} |> c'g(t)   \text{ and }X_{t} \in \cT_y  ] \ge \frac{c_{0}t}{d^{2\|x \wedge y\|-\|y\|-1}}.
\end{split}
\end{equation}
 \end{itemize}
\end{prop}
In the following remark we present a minor variant of the previous proposition which is needed for the proof of Theorem \ref{thm:2}. We omit its proof to avoid repetitions. 
\begin{rem}
\label{r:frozen}
Only minor adaptations to the proofs of \eqref{eq: Range3} and \eqref{eq: Range3'} are needed in order to show that in the setup of
\eqref{eq: Range3} \[\Pr_z[|R_{t} \cap A |>c' \delta g(t) \text{ and }X_{i} \in \cT_y(d^{k}) \text{ for all }i \le t  ] \ge \hat c\delta^{2} \]
and in the setup of \eqref{eq: Range3'} if $\hat u $ is the parent of $x \wedge y$ then
\[\Pr_z[\max_{\ell \in [d^{i-k}] } |R_{t} \cap \cT_{y,k}^{\ell} |> c'g(t)   \text{ and }X_{t} \in \cT_y \text{ and }T_{\hat u}>t ] \ge \frac{\hat c_{0}t}{d^{2\|x \wedge y\|-\|y\|-1}}. \]    
\end{rem}
We make two comments concerning Proposition \ref{cor:range} which may assist the reader.
\begin{itemize}
\item
In the setup of (i), by Lemma \ref{lem: SRWontree} and the fact that $d^{k-1} \le t \le d^k $ we have that  $\Pr_z[X_{t} \in \cT_y(d^{k})] \succsim 1$. It is intuitively clear that $|R_{t} \cap A |$ and $1_{X_{t} \in \cT_y(d^{k})} $ are positively correlated (this is established in the proof of Proposition \ref{cor:range}). Hence if $E$ is the event that $|R_{t} \cap A |>c' \delta g(t)$ and  $D$ is the event that the walk reached the root of $\cT_y(d^{k})$ and afterwards reached its leaf set by time $\sfrac{3}{4}t$ then by Lemma \ref{lem: SRWontree}  $ \Pr_z[ D ] \succsim 1 $ and \[\Pr_z[E \text{ and }X_{t} \in \cT_y(d^{k})  ] \asymp \Pr_z[E   ] \ge \Pr_z[E  \cap D ] \succsim \Pr_z[E  \mid D ]. \]  This shows that w.l.o.g.\ we can assume that at time 0 the initial distribution is the uniform distribution over the leaf set of $\cT_y(d^{k})$ (since on $D$, once the leaf set of  $\cT_y(d^{k})$  is hit after its root is visited, the walk has the uniform distribution on the leaf set of   $\cT_y(d^{k})$, and at least $t/4$ remaining steps by time $t$). This means that the expected time the walk spends in $A$ by time $t$ is $\succsim t \frac{|A|}{d^k} \ge \delta t $. Since the expected number of returns by time $t$ to each leaf that is visited by time $t$ (resp.\ $t-\sqrt{t}$) is by  Lemma \ref{lem: SRWontree2} $\precsim \log_d (dt)$ (resp.\ $\succsim \log_d \sqrt{dt}=\half \log_d (dt)$) it is not hard to deduce that $\mathbb{E}[|R_{t}  \cap A |  ] \asymp  \delta t/\log_d(dt)$.

 To conclude the argument from this point we use the Paley-Zygmund
inequality together with the estimate $\mathbb{E}[|R_{t}  |^{2}  ] \precsim ( \mathbb{E}[|R_{t}  |  ])^{2} \precsim [t/\log_d(dt)]^{2} $, where the first inequality holds by general considerations, while the estimate $\mathbb{E}[|R_{t}  |  ] \precsim t/\log_d (dt) $ is again derived from Lemma \ref{lem: SRWontree2} in a similar manner to the derivation of  $\mathbb{E}[|R_{t}  \cap A |  ] \asymp  \delta t/\log_d(dt)$. For the actual details see the proof of part (i).   
\item
The term $td^{-\|x \wedge y\|+1} $ in  \eqref{eq: Range3'} corresponds to the (order of the) probability that a walk starting from $x$ reaches $x \wedge y$ by time $t/4$, while the term $ d^{-\|x \wedge y \|+\|y\|}$ corresponds to the (order of the) probability that starting from $x \wedge y$ a walk of length $ t/4$ will reach the leaf set of $\cT_y$. Thus with probability $\succsim td^{-\|x \wedge y\|+1} \cdot d^{-\|x \wedge y \|+\|y\|}  $ the walk reaches the leaf set of $\cT_y$ by time $t/2$. Thus \eqref{eq: Range3'} follows from \eqref{eq: Range3} via the Markov property by noting that the restriction that $t \ge d^{k-1}$ in \eqref{eq: Range3} is not needed when initially the walk starts from the uniform distribution on $\cL(\cT_y(d^k)) $ (nor when $\delta=1$, i.e.\ $A=\cL(\cT_y(d^k))$).
\end{itemize}
The following corollary shall be central in the proofs of Theorems \ref{thm:1} and \ref{thm:2}. Its part (ii) asserts that for particle density $\la/2$  if the set of activated sites at time 0 contains at least a quarter of the leaf set and satisfies a certain spatial homogeneity condition, then the set of initially activated particles cover the entire graph in  $s=s_{n,\la}:=\lceil C_{1} \sfrac{n}{\la}\log \sfrac{n}{\la}  \rceil $ steps $\whp
$. The   spatial homogeneity condition is that inside each of the $d^{n-k} $ induced subtrees of depth     $k:=\lceil \log_d s \rceil$ (i.e.\ the induced subtrees whose roots belongs to $\cL_{n-k}$) initially there are at least $d^k/4$ activated sites.

Part (ii) is used to argue in part (iii) that by splitting the particles into two independent sets of density $\la/2$, if  the collection of  vertices which are activated by the first set of particles satisfies the spatial homogeneity condition, we can use the second set of particles to activate the rest of the vertices. This is similar to a standard technique from percolation theory called sprinkling. 

Part (i) asserts that if inside an induced tree $\cT_{y}$ of depth $k$ the set of activated vertices $B$ is of size at least $C_{2} \la^{-1}\log |\VV_{d,n}|$, then  the collection of all leafs of $\cT_y$ visited by a particle in $\W_B$ by time $s$ is of size at least $d^k/4 $ with probability at least $1-2|\VV_{d,n}|^{-4} $. In fact, part (i) asserts that this is the case even if we replace $\W_B$ by the particles in $\W_B$ whose length $s$ walk terminates at $\cT_y$. 
\begin{cor}
\label{cor:coverleaf}
There exist absolute constants $C_1,C_{2} ,  C_3(d)>0  \ge 1$ such that the following hold for all $d \ge 2 $ and $\la>0$. Let $o \in \cL$. Denote $s=s_{n}:=\lceil C_{1} \sfrac{n}{\la}\log \sfrac{n}{\la}  \rceil $ and    $k:=\lceil \log_d s \rceil$.  Let $y=\overleftarrow{o_k}$ be the $k$th ancestor of $o$. Assume that $\la$ and $n$ are such that $1 \ll s \le d^{n}$.
\begin{itemize}
\item[(i)] Let $B \subseteq \VV(\cT_y) $ be of size $|B| \ge C_{2} \la^{-1}\log |\VV_{d,n}| $. At each $b \in B $ there are initially $\Pois(\la /2) $ (independently for different vertices) particles performing simultaneously independent SRWs of length $s$ on $\cT_{d,n}$. Denote by $D$ the collection of leafs in $\cL( \cT_y) $ visited by at least one particle (from the aforementioned collection) whose length $t$ walk terminated inside $\cT_y$. Then the probability that $|D| <d^{k}/4 $ is at most $2|\VV_{d,n}|^{-4} $, for all $n \ge C_3(d)$.  
\item[(ii)]
Let  $m:=d^{n-k}$. Let $\cT^1,\ldots,\cT^m$ be the collection of all induced subtrees (in $\cT_{d,n}$) of depth $k$.  For all $i \in [m]$ let $A_i \subseteq \cL(\cT^i)$ be of size $|A_i| \ge d^{k}/4$. Let $A=\cup_{i \in [m]}A_i $. If at each $a \in A $ there are initially $\Pois(\la /2) $ (independently for different vertices) particles performing simultaneously independent SRWs of length $s $  on $\cT_{d,n}$, then the probability that there is some leaf which is not visited by any of the walks is at most $|\VV_{d,n}|^{-4}$. \item[(iii)] Let $\mathcal{R}_s$ be the set of activated vertices when the particle lifetime is $s$. In the notation of part (ii) let $\mathrm{Hom} $ be the event that $|\mathcal{R}_s \cap \cL(\cT^i)| \ge d^k/4 $ for all $i \in [m]$. Then \[\PP_{\la}[\SS(\cT_{d,n}) > s ] \le \PP_{\la/2}[\mathrm{Hom}^c ]+|\VV_{d,n}|^{-4}.   \]
\end{itemize}  
\end{cor}
\begin{rem}
\label{r:maincor}
Using Remark \ref{r:frozen}, only minor adjustments to the proof of part (i) of  Corollary \ref{cor:coverleaf}  are required in order to show that its assertion remains valid if each particle is killed upon escaping the tree $\cT_y$.  
\end{rem}
\begin{rem}
\label{r:Bdifsize}
If in part (i) of Corollary \ref{cor:coverleaf} we make no assumption on $|B|$, then the same proof shows that for some $c>0$ the probability that $|D| < \min \{\sfrac{d^k}{4},\sfrac{c \la |B| s}{\log_ds}  \}  $ is at most $ \sfrac{1}{c} e^{-c |B|}$.
\end{rem}
\section{Proof of Theorem \ref{thm:1}}
\label{s:thm1}
Below we write $\VV $ instead of $\VV_{d,n}$.
In this section we show that there exist some $C,\hat c>0$ such that for all  $d \ge 2$, as long as $ \la_n \le \hat c \log n $ and $\la_n d^{n} \ge n^2 \log d $ we have that
\begin{equation}
\label{e:WWW}
\lim_{n \to \infty} \PP_{\la_{n}}[  \SS(\cT_{d,n})\le C \lceil \sfrac{n}{\la_n}  \log \sfrac{n}{\la_n} \rceil   ] = 1.
 \end{equation} 
  The proof of the existence of some $c>0$ such that $\lim_{n \to \infty} \PP_{\la_{n}}[  \SS(\cT_{d,n}) \ge c \sfrac{n}{\la_n}  \log \sfrac{n}{\la_n}] = 1$ for   all $d \ge 2$ and all $\la_{n}=o(n) $ such that  $\la_n d^{n} \ge n \log d $ 
is deferred to \S~\ref{s:lowersustree}.

%

Let $C_1>0$ be as in Corollary \ref{cor:coverleaf}.  As in Corollary \ref{cor:coverleaf} let
\begin{equation}
\label{eq:s}
s=s(n,\la):=\lceil C_{1} \sfrac{n}{\la}\log \sfrac{n}{\la}  \rceil  \quad \text{and} \quad    k=k(n,\la,d):=\lceil \log_d s \rceil.
\end{equation}
\subsection{Reducing to the case that the process starts at a leaf}

The following lemma asserts that it suffices to consider the case that $\oo \in \cL $ (where $\oo$ is the initial position of $\plant$). We note that \eqref{e:red0} is useful as long as $C_1 \sfrac{n}{\la_n}  \log \sfrac{n}{\la_n}  \ge 16 n   $, which is the case (provided that $C_1$ is taken to be sufficiently large) by our assumption that $\la_n < \hat c \log n $.

\begin{lem}
\label{lem:red1}
Let $\PP_\la^v$ be the distribution of the frog model when $\oo=v$, where $\oo$ is the initial position of the planted particle $\plant$. Let $u=u_n \in \cL $. Then for all $t(n,\la_{n})$ 
\begin{equation}
\label{e:red0}
\liminf_{n \to \infty} \min_{v \in \VV_{d,n}}  \PP_{\la_{n}}^{v}[  \SS(\cT_{d,n})\le  t(n,\la_{n})  +8n  ] \ge \liminf_{n \to \infty} \PP_{\la_{n}}^{u}[  \SS(\cT_{d,n})\le t(n,\la_{n})  ]. 
\end{equation}
\end{lem}
\emph{Proof.} For \eqref{e:red0} observe that $\max_v \PP^{v}[\plant \text{ did not reach }\cL \text{ by time } 8n]=o(1)$\footnote{We omit the dependence on $\la$ when we  consider events depending only on $\plant$.} \[\PP_{\la_{n}}^{v}[  \SS(\cT_{d,n})> t(n,\la_{n}) +8n  ] \le \PP_{\la_{n}}^{u}[  \SS(\cT_{d,n})> t(n,\la_{n})  ]\] \[+ \PP^{v}[\plant \text{ did not reach }\cL \text{ by time } 8n]. \quad \text{\qed} \] 
\subsection{Notation, definitions and exposition of our proof strategy} 
Let $\mathcal{R}_s$ be the set of activated vertices when the particles' lifetime is $s$. As in part (iii) of Corollary \ref{cor:coverleaf} let $\mathrm{Hom} $ be the event that $|\mathcal{R}_s \cap \cL(\cT^i)| \ge d^k/4 $ for all $i \in [m]$. By part (iii) of Corollary \ref{cor:coverleaf} to prove \eqref{e:WWW} it suffices to show that $\PP_{\la/2}[\mathrm{Hom}^c ]=o(1) $.  Recall that the idea behind this reduction is that we may split the particles into two independent sets of density $\la/2$, with the planted particle belonging to the first set. If the event $\mathrm{Hom}$ happens w.r.t.\ the first set of particles, then by part (ii) of  Corollary \ref{cor:coverleaf} the particles from the second set of particles, whose initial positions are the ones activated by the first set of particles, will together activate the rest of the vertices.  For notational convenience we shall  show that $\PP_{\la}[\mathrm{Hom}^c ]=o(1) $ (ignoring the $1/2$ term). As the theorem is proved for a large range of $\la$ there is no loss in replacing $\la/2$ by $\la$.  
 
 \medskip
 
Let $\cT^1,\ldots,\cT^m$ be the collection of all $m=d^{n-k} $ induced subtrees (in $\cT_{d,n}$)  of depth $k$. We denote the root of $\cT^i$ by $x_i \in \cL_{n-k} $ (i.e.\ $\cT^i=\cT_{x_i}$). We label the trees from right to left ($\cT^1$ being the rightmost). 

Recall that by Lemma \ref{lem:red1} we may assume that $\oo \in \cL $ (i.e.\ that the initial position of $\plant$ is a leaf). By symmetry we may further assume that $\oo \in \cL( \cT^1)$. Let $z_0=x_{1}$ be the root of $\cT^1$, i.e.~$\cT_{z_0}=\cT^1$. Denote the $i$th ancestor of $z_0$ by $z_i:= \overleftarrow{\oo_i}$ for all $i \in [n-k]$ (i.e.\ $z_1$ is the parent of $z_0$, $z_2$ is the parent of $z_1$, etc.\ and $z_{n-k}=\rr $). Then $z_i \in \cL_{n-k-i}$ and each of its children contains $d^{i-1} $ trees from $\cT^{1},\ldots,\cT^{m}$. We label the children of $z_i $ by $y_i^0,y_i^1,\ldots,y_i^{d-1} $ so that $z_{i-1}=y_i^0$ and  $\cT^\ell \subseteq \cT_{y_i^j} $ for all $i \in [n-k]$, $j \in ]d-1[$  and $jd^{i-1} < \ell \le (j+1)d^{i-1} $.

\medskip

We employ a recursive divide and conquer approach. Namely, we analyze the set of activated vertices of $\cT_{z_i}$ for each $i \in [n-k]$ by exploiting the earlier analysis of  $\cT_{z_{i-1}}= \cT_{y_i^0}$ and the fact that $\cT_{y_i^0} $ is isomorphic to $\cT_{y_i^j}$ for all $j \in [d-1]$. To set up the recursion we first need a couple of definitions. We first recall some notation.

Recall that we denote the collection of particles whose initial position belongs to a set $U \subseteq V $ (resp.~is $v \in V$) by $\W_U$ (resp.~$\W_v$). Recall that we denote the range of the length $\ell$ walk picked (in the sense of \S~\ref{s:construction}) by a particle $w$  by $\RR_{\ell}(w)$. Recall that we denote the union of the ranges of the length $\ell$ walks picked  by the particles belonging to some set of particles $\mathcal{U}  $ by $\RR_{\ell}(\mathcal{U}) $.  Denote the union of the ranges of the length $\ell$ walks picked  by the particles whose initial positions lie in $U \subseteq V $ (resp.~is $v$) by $\RR_{\ell}(U):=\RR_{\ell}( \W_U)$ (resp.~$\RR_{\ell}(v):=\RR_{\ell}( \W_v)$). When $\ell=s$, where $s$ is as in \eqref{eq:s}, we omit it from the subscript and write  $\RR(w),\RR(\mathcal{U}),\RR(U)$ and $\RR(v) $.

\begin{defn}
Let $B_i$ be the collection of leafs of $\cT^i$  which are activated  before the process dies out, when the particles' lifetime is $s$. We say that  $\cT^i$ is \emph{conquered} if $|B_i| \ge d^{k}/4$. We say that a tree  $\cT_x$ is \emph{conquered} if the tree  $\cT^i$  is conquered for all $i$ such that $\cT^i \subseteq \cT_x$.
\end{defn} 
Observe that the event $\mathrm{Hom} $ is the event that $\cT_{d,n}$ is conquered. One difficulty that arises when attempting to set up the recursion  is that possibly the tree $\cT_{z_{i}} $ was conquered with the assistance of some particles which originally lie outside of $\cT_{z_i}$. It seems possible to overcome this difficulty using Poisson thinning and the FKG inequality (namely, that  functions which are non-decreasing w.r.t.~the point configuration associated with a Poisson random measure are positively correlated). Instead, we define a refined notion of conquering a tree which avoids this complication. 

Before giving that definition we need the following two auxiliary definitions.
\begin{defn}
For $i \in [m]$ (where $m=d^{n-k}$) let $\HHH_i$ be the collection of all particles such that the length $s$ walk they picked  ends at $\cT^i$ (where as above $\cT^1,\ldots,\cT^m$ are all the induced subtrees of $\cT_{d,n}$ of depth $k$). 
\end{defn}
 Recall that $\|v\| $ is the distance of $v$ from $\cL$. Recall that $x \wedge y$ is the common ancestor of $x$ and $y$ which is closest to the leaf set. 
\begin{defn}
We say that a vertex $v \in \cT^j $ for some $j \in [m]  $ is \emph{good} for lifetime $s$ if
there exist a collection of sites $v_0 = \oo,v_1,\ldots,v_{\ell+1}=v $ and a collection of particles $w_{0} \in \W_{v_0} ,w_{1} \in \W_{v_1},\ldots,w_{\ell}  \in \W_{v_\ell}$ (i.e.\ for all $i \in ]\ell[$ the initial position of $w_i$ is $v_i$) such that the following hold.
\begin{itemize}\item[(i)]   $v_0,\ldots,v_{\ell} $ all lie in the last $k$ levels of $\cT_{v \wedge \oo } $.
\item[(ii)] For all $i \in [\ell]$ we have that $\|v_{i} \wedge v_{i+i'} \|$ is non-decreasing w.r.t.~$i' \in ]\ell -i [$ and that also  $\|v_{i} \wedge v_{i-i'} \|$ is non-decreasing w.r.t.~$i' \in ] i [$.

\item[(iii)] For all $i \in ]\ell [$ we have that   $v_{i+1} \in \RR(w_i) $ (i.e.\ $v_{i+1}$ is visited by $w_i$). \item[(iv)] For all $i \in ]\ell [$ unless $w_{i}=w_{\mathrm{plant}}$ (which is possible only for $i=0$)     $w_{i} \in \cup_{\ell:\cT^{\ell} \subseteq \cT_{v \wedge \oo }  } \HHH_{\ell}$. That is, for all $i \in ]\ell [$  the length $s$ walk of $w_i \neq \plant $ ends at the last $k$ levels of $\cT_{v \wedge \oo } $.  
\end{itemize}
\end{defn}
Observe that (iii) implies that all of the vertices $v_0 ,v_1,\ldots,v_{\ell+1} $ got activated. Conditions (i), (ii) and (iv) imply that this is done using particles from $ \W_{\cT_{v \wedge \oo }} $ such that the length $s$ walk they picked begins and terminates in the last $k$ levels of $\cT_{v \wedge \oo }$. This avoids the aforementioned  complication in setting up the recursion.   

We now define the notion of \emph{internally conquering a tree} which is a key concept in our divide and conquer approach.

\begin{defn}
\label{def:internally}
We say that  $\cT^j$ (where $j \in [m]$) is \emph{internally conquered} if the collection $ D^j $ of leafs of $\cT^{j}$ which are good for lifetime $s$   satisfies
\[|D^j  | \ge d^{k}/4. \]
We say that $\cT_{z_i}$ is \emph{internally conquered}, if $\cT^j $ is internally conquered  for all $j \in [d^{i}] $ (i.e.\ all $d^i$ induced subtrees of depth $k$ which are subtrees of $\cT_{z_i}$ are internally conquered).
\end{defn}
The usefulness of the last two definitions is demonstrated by Lemma \ref{lem:maketherefereehappy}.

Loosely speaking, we now give a simple condition which ensures that a certain $\cT^i$ gets internally conquered w.p.~at least $1-2|\VV|^{-4}$. In practice, we only use this estimate for $\cT^1$. Let  \[r=r(n,\la):=\lceil C_{2} \la^{-1}\log |\VV| \rceil,\] where $C_2$ is as in Corollary \ref{cor:coverleaf}. We take $C_1$ from the definition of $s$ to be sufficiently large so that $c's/\log_{d} s>r$ (where $c'>0$ is as in part (ii) of Proposition \ref{cor:range}).    
\begin{lem}
\label{lem:red2}
For every $i \in [m]$, every set $D \subseteq \cT^i  $ such that $|D| \ge r $ satisfies that \[ |\RR(\W_D \cap \HHH_i) \cap \cL(\cT^i) | \ge d^{k}/4 \] w.p.~at least $1-2|\VV |^{-4}$, provided that $n$ is sufficiently large.  \end{lem}
\begin{proof}
This follows immediately from part (i) of Corollary \ref{cor:coverleaf}.
\end{proof}
\subsection{The  base case}
\begin{lem}
\label{cor:warmstart}
By replacing the particle density $\la$ by $\la/2$ we may assume that at time 0 there is a set $D \subseteq \cT^1 $ of size at least $r$ which is activate.\end{lem}
\begin{proof}
Split the particles into two independent sets of the same density (with $\plant$ belonging to the first set). It suffices to show that the dynamics corresponding only to the first set of particles satisfies that $\whp$ at least $r$ vertices of $\cT^1$ are activated before this process dies out. As before, instead of showing this for particle density $\la/2$ we show this for density $\la$.

 First expose $\RR(w_{\mathrm{plant}}) \cap \cT^1 $. By Lemma \ref{lem:rangeofplant}  $|\RR(w_{\mathrm{plant}}) \cap \cT^1| \ge \sqrt{s}  $ $\whp$. We now expose $\RR(\W_{\RR(w_{\mathrm{plant}})\cap \cT^1 } \cap \HHH_1) $. Denote $U_1:=\RR(w_{\mathrm{plant}})\cup \RR(\W_{\RR(w_{\mathrm{plant}})\cap \cT^1 } \cap \HHH_1) $. Conditioned on $|\RR(w_{\mathrm{plant}}) \cap \cT^1|> \sqrt{s}  $,  by Remark \ref{r:Bdifsize} we get that   $\whp$  $|U_{1}| \ge r$.
\end{proof} 
\begin{lem}
\label{lem:red3}
Assume that initially there is a set $D \subseteq \cT^1 $ of size at least $r$ which is activate. Then
the probability that $\cT^1 $ does not get internally conquered is at most $2|\VV |^{-4}$.
\end{lem}
\begin{proof}
This follows at once from Lemma \ref{lem:red2}.
\end{proof}

\begin{lem}
\label{lem:rangeofplant}
$\whp$  $|\RR(w_{\mathrm{plant}}) \cap \cT^1| \ge \sqrt{s}  $.
\end{lem}  
\begin{proof}
By \eqref{eq: SRWtree5}  $\whp$ $\plant $ does not reach level $n-k $ by time  $s^{3/4}$ (or any other time which is $o(s)$), where as above $k=\lceil \log_d s \rceil $. Thus by a straightforward coupling argument it suffices to show that the size of the range of SRW by time $s^{3/4}$ on $\cT_{d,k}$ is $\whp$ at least $\sqrt{s}$.  The claim now follows by applying \eqref{e:submulrangtree} with $n=k $ and $t=s^{3/4}$.   
\end{proof}
 \begin{lem}
\label{lem:easyrange}
Let $(X_t)_{t=0}^{\infty}$ be a SRW on some graph $G$. Let $R(t)=\{X_i:i \in [t] \}$ be the range of its first $t$ steps. Then for all $s,t',\ell \in \N $ and $x \in V$ we have 
\begin{equation}
\label{e:submulrang}
\Pr_x[|R(st')| \le \ell  ] \le (\max_{v \in V} \Pr_x[|R(t')| \le \ell ])^{s}.
\end{equation}
In particular, for $G=\cT_{d,n}$ we have that for all $x \in \VV_{d,n}$ and all $t \le d^{n-1} $ that
\begin{equation}
\label{e:submulrangtree}
\Pr_x[|R(t)| \le t^{2/3} ] \le \exp (-c_{0}t^{1/4}).
\end{equation}
\end{lem}
\begin{proof}
For \eqref{e:submulrang} apply the Markov property $s-1$ times at times $it'$ for $i \in [s-1]$. Equation \eqref{e:submulrangtree} follows from \eqref{e:submulrang} by picking $s= \lfloor t^{1/4} \rfloor $, $t'=\lfloor t^{3/4} \rfloor $ and noting that by the results from \S~\ref{s:aux} we have that $ \min_{x \in \VV_{d,n}} \Pr_x[|R(\lfloor t^{3/4} \rfloor)| \le t^{2/3}  ]  $ is bounded away from 1, uniformly in $t$. 
\end{proof}

\subsection{The recursion step}
\label{s:basis}
We are now in the position to start the recursion.

\begin{prop}
\label{p:rec}
Let  $D^{\ell}$ be as in Definition \ref{def:internally}. Consider the case that initially the set of activated vertices is $U_1 \subseteq \cT^1 $ and $|U_1 | \ge r$. Let $p_j=p_j(n,d,\la)$ be the probability  that  for all $i \in ]j[ $ we have that $\cT_{z_i}$ is not internally conquered.
Then for all $0 \le j <n-k$ we have that
\[p_{j+1} \le d p_{j} +(d-1)|\VV|^{-4}.\] 
\end{prop}
Before proving Proposition \ref{p:rec} we explain how it implies the assertion that  $\PP_{\la}[\mathrm{Hom}^c ]=o(1)$. Indeed, by Lemma \ref{lem:red3} we have that $p_0 \le 2|\VV |^{-4} $. It follows by induction that $p_{j} \le 2 |\VV |^{-4}  \sum_{i=1}^{j+1}d^{i}$ and thus $p_{n-k} \precsim |\VV |^{-3}  $. Finally, using Lemma \ref{cor:warmstart}  we may indeed assume that initially the collection of activated vertices is some set  $U_1 \subseteq \cT^1 $ such that $|U_1 | \ge r$.

Before proving  Proposition \ref{p:rec} we need the following definition.
\begin{defn}
For $i \in [m]$ let \[\mathcal{Q}_i:=\{w \in \HHH_i : |\RR(w) \cap \cT^i | \ge c's/ \log_{d} s\}, \] be the collection of all particles in $\HHH_i$ such that the intersection of the range of their length $s$ walks with $\cT^i$ is of size at least $c' s/ \log_{d} s$, where $c'>0$ is as in part (ii) of Corollary \ref{cor:range}.
\end{defn}
\emph{Proof of Proposition \ref{p:rec}:}  
Let $i \in [n-k]$ and $D \subseteq \cL( \cT_{z_{i-1}})$ be so that \[|D | \ge \sfrac{d^{k}}{4}|\{\ell:\cT^\ell  \subseteq \cT_{z_{i-1}} \}|=\sfrac{d^{k+i-1}}{4} .\] Let $j \in [d-1]$. Let $\xi_{D,y_i^j}:=|\W_D \cap (\cup_{\ell:\cT^\ell \subset \cT_{y_i^j}} \mathcal{Q}_{\ell} )| $ be the number of particles whose initial position is in $D$ which picked a length $s$ walk ending at one of the trees $\cT^\ell \subseteq \cT_{y_i^j}$, whose range contains at least $c's/\log_{d} s \ge r$ vertices of $\cT^\ell$. Recall that we may choose $C_1$ from the definition of $s$ so that $c's/\log_{d} s \ge r$. By part (ii) of Proposition \ref{cor:range} (which is applicable, provided that $\hat c$ and $C_1$ are chosen so that $s \ge 32 n$), provided that $C_1$ is taken to be sufficiently large and that $n$ is sufficiently large, we have that  $\xi_{D,y_i^j} $ has a Poisson distribution whose mean is at least 
\begin{equation*}
 \la  |D|\times (\hat c_0 sd^{-(i+k+1)}) \ge c_{0} d^{-2}\la s \ge 4 \log |\VV| 
\end{equation*}
 (for some $\hat c_0,c_0>0$, where in the last inequality we have used the assumption that $\la \le \hat c \log n $ and hence $ \la s \asymp n \log ( n/\la) \gg n $), and so we have that 
\begin{equation}
\label{e:recstep}
\PP_{\la}[\xi_{D,y_i^j}=0] \le |\VV|^{-4}.
 \end{equation}
 Let $D^{\ell}$ be as in Definition \ref{def:internally}. Denote the set of leafs of  $\cT_{z_i}$ which are good for lifetime $s$ by  \[G_i:=\cup_{\ell:\cT^\ell \subseteq \cT_{z_i} } D^\ell.  \]
\begin{lem}
\label{lem:maketherefereehappy}
The set $G_j
$ is independent of the walks picked by the particles in $\W_{\VV \setminus \cT_{z_{j}}} $ and also of the particles  in  $\W_{ \cT_{z_{j}}} $ whose length $s$ walk terminates outside $\cT_{z_j}$.\footnote{More precisely, it is independent of the number of particles who picked each such path.}
\end{lem}
\begin{proof}
 Note that $G_{j} $ is a random set which is measurable w.r.t.~the $\sigma$-algebra generated by $\RR(w_\mathrm{plant}) $ and $\RR(\W_{\cup_{\ell:\cT^\ell \subset \cT_{z_j} }\cT^\ell} \cap (\cup_{\ell:\cT^\ell \subset \cT_{z_j} }\HHH_{\ell} ) )$. Using this, the lemma follows from Poisson thinning.       
\end{proof}
We condition on the event that $\cT_{z_j}$ is  internally conquered.
Conditioned on the identity of $G_j$ being a certain (fixed) set of size at least $\frac{d^{k}}{4}|\{\ell:\cT^\ell  \subset \cT_{z_{j}} \}| $ (as implied by  $\cT_{z_j}$ getting internally conquered),  it follows from \eqref{e:recstep} together with a union bound that w.p.~at least $1-(d-1)|\VV|^{-4}$, for all $\ell \in [d-1]$ there is some $\cT^{a_\ell} \subset \cT_{y_{j+1}^{\ell}}$ such that $\W_{G_j} \cap \mathcal{Q}_{a_{\ell}} $ is non-empty. 

By the definition of the $\mathcal{Q}_i$'s and the choice of constants (used to argue that $c's/ \log_{d} s \ge r$), if this occurs for  $ \cT_{y_{j+1}^{\ell}} $ for some $\ell \in [d-1] $, then there is some  $\cT^{a_\ell} \subset \cT_{y_{j+1}^{\ell}}  $ which contains at least $r$ sites which are good for lifetime $s$. By Lemma \ref{lem:maketherefereehappy}  this is precisely what is needed in order to apply the recursion, as the tree $\cT_{y_{j+1}^{\ell}}$ is isomorphic to $\cT_{y_{j+1}^{0}}=\cT_{z_{j}}$ and on the aforementioned event, there is some induced subtree of  $\cT_{y_{j+1}^{\ell}}$  of depth $k$ that has at least $r$ vertices which are activated using particles from $\W_{\cT_{z_{j}}} $.\footnote{We need to have at some induced subtree of  $\cT_{y_{j+1}^{\ell}}$  of depth $k$   at least $r$ vertices which are activated, as this is part of the assumption we had for the case $j=0$ in the statement of Proposition \ref{p:rec}.}  

Applying the recursion on each of the $d-1$ trees $\cT_{y_{j+1}^{\ell}}$ (where $\ell \in [d-1]$) we get that the conditional probability of $\cT_{z_{j+1}} $ not getting internally conquered (conditioned on  $\cT_{z_{j}} $ getting internally conquered)   is by a union bound at most $(d-1)|\VV|^{-4}+(d-1)p_j$. This concludes the proof.  \qed

\section{Proof of Theorem \ref{thm:2}}
\label{s:thm2}
We employ some of the definitions and notation from the proof of Theorem \ref{thm:1}.  As opposed to the proof  of Theorem \ref{thm:1}, here we consider walks of increasing lengths. Recall that for $x \in \VV_{d,n}$ and $| \mathcal{T}_{x}| \le t \le |\VV_{d,n}| $ we denote by $\cT_x(t)$ the smallest induced tree $\cT_y $ such that $y \le x$ and $|\cT_y| \ge t$.
Let $s:=\lceil C_{1} \sfrac{n}{\la}\log \sfrac{n}{\la}  \rceil,k:=\lceil \log_d s \rceil,m:=d^{n-k},r:=\lceil C_{2} \la^{-1}\log |\VV| \rceil,$  $\cT^1,\ldots,\cT^m $ and $z_0,\ldots,z_{n-k}=\rr, (y_{i}^j)_{i \in [n-k],j \in ]d-1[} $ be as in the proof of Theorem \ref{thm:1}.
  
\begin{proof}
  As in the proof of Theorem \ref{thm:1}, we may assume that $\oo \in \cL( \cT^1) $ (using an analog of \eqref{e:red0} for the cover time instead of the susceptibility; Namely, using the fact that $\plant$ hits the leaf set by time $8n$ $\whp$). We say that a tree $\cT^i$ (where $i \in [m]$) is \emph{conquered by time} $t$ if at least $d^k/4$ of its leafs are activated by time $t$. Here and throughout this section,  time is measured  w.r.t.~the entire process, not w.r.t.\ a walk of a particular particle. We say that a tree $\cT_x$ is \emph{conquered by time} $t$ if $\cT^i $ is conquered by time $t$ for all $i$ such that $\cT^i \subseteq \cT_x $. 

Similarly to the proof of Theorem \ref{thm:1}, the cover time is the time at which the last leaf is activated.
 Much like in the proof of Theorem \ref{thm:1}, using part (ii) of Corollary \ref{cor:coverleaf} it suffices to analyze the time it takes  to conquer $\cT_{d,n}$. Like in the proof of Theorem \ref{thm:1}, the idea behind this is that we can split the particles into two independent sets of equal density, use the first set to  conquer $\cT_{d,n}$ and then use the second set to activate the rest of the leafs. Again, for the sake of notational convenience we still refer to the particle density as $\la$ rather than $\la/2$. 

We consider in the proof below a modified dynamics. As in the proof of Theorem \ref{thm:1}, we may assume that at time 0 there is a set $D \subseteq \cT^1 $ of size at least $r$ which is activated.

\medskip

Below $(s_{i})_{i=0}^{\ell_*}$  will be an increasing sequence of sizes (of some induced trees) while $(t_i)_{i=0}^{\ell_*}$ will be an increasing sequence of times ($\ell_*$ shall be defined shortly). Let $s_0:=s $ and $t_{0}:=2s $. We define recursively \[t_{\ell}:=t_{0}3^{\ell} \quad \text{and} \quad s_{\ell+1}:=\min \{ s_{\ell}\sqrt{ \sfrac{\la t_{\ell}}{ Md^{4} \log |\VV|}} , |\VV|\} =\min \{ s_{\ell}\sqrt{ \sfrac{C_{1}3^{\ell} \log (n/ \la) }{ Md^{4}\log d}} , |\VV|\},\] for some constant $M>0$ to be determined later. By monotonicity w.r.t.\ the particle density it suffices to prove the theorem in the case that $ \la_n \le \hat c \log  n$.  Observe that (provided that $n$ is sufficiently large) for all $\ell$ we have, \begin{equation*}
\label{e:sell}
s_{\ell} \ge  3^{\half \sum_{i=1}^{\ell}i } \ge 3^{\sfrac{1}{8} \ell^2 }.
\end{equation*}
To see this, observe that for all $1 \le \ell< \ell_* $ we have that  \[\log_3 s_{\ell}=\log_3 s_{\ell-1}+\sfrac{\ell}{2}+\sfrac{1}{2} \log_3 ( c_{2}(d)\log (n/ \la)  ) \ge \log_3 s_{\ell-1}+\sfrac{\ell}{2}.  \]
(where the last inequality holds for sufficiently large $n$).
Denote $\ell_*:=\min \{\ell :s_{\ell}:=|\VV| \}$. Note that $\sfrac{1}{8} \ell_*^2 \le \log_3 |\VV| $ and so $\ell_* \le 3  \sqrt{\log |\VV|}$. Consequently \[t_{\ell_*} \le t_{0} 3^{3 \sqrt{\log |\VV|}}.\] 

  We consider
the dynamics  in which for all $i \in ]\ell_*-1 [ $, once an activated particle leaves $\cT_{\oo}(s_{i} )$ before time $t_i$, it is frozen up to time $t_i$, at which time it continues its walk. We may think of the (infinite) walk picked by each particle in the aforementioned \emph{frozen model} as the same one it picks in the original model, but in the frozen model, the walk is delayed in the case the particle leaves $\cT_{\oo}(s_{i} )$ before time $t_i$ for some $i$. Note that (in the aforementioned coupling) the cover time for the  frozen frog model
  is at least as large as the cover time of the original model. 

\medskip

Assuming that at time 0 there is a set $D \subseteq \cT^1 $ of size at least $r$ which is activated, for all $0 \le i \le \ell_{*} $  let $\alpha_i $   be the probability that the frozen frog model does not satisfy conditions (1) and (2) below. 

\begin{itemize}
\item[(1)]
 For every $\cT^j \subseteq \cT_{\oo}(s_{i} )$ the set of leafs of $\cT^j$ which were activated by time $t_{i}$  is of size at least $d^k/4$.
\item[(2)]
For every $\cT^j \subseteq \cT_{\oo}(s_{i} )$ the collection of  particles whose initial position belongs to the set of leafs of $\cT^j$ which were activated by time $t_{i}$   is of size at least $\la d^{k}/8$.
\end{itemize}
 Denote the  event that (1) and (2) hold (for $i$) by $A_i$. Since  $t_{\ell_*} \le t_{0} 3^{3 \sqrt{\log |\VV|}}$  it suffices to show that $\alpha_{\ell_*}=o(1) $ in order to deduce that $\whp$ $\cT_{d,n}$ is conquered by time  $t_{0} 3^{3 \sqrt{\log |\VV|}}$.    

\medskip

Denote the depth of $\cT_{\oo}(s_{i} ) $  by $\rho_i $.
We will prove the following recursion 
\begin{equation}
\label{e:recthm2}
\alpha_{i+1} \le d^{\rho_{i+1}-\rho_i}(|\VV|^{-4}+\alpha_i). 
\end{equation}
For the induction basis we argue that (possibly by increasing the constant $C_1$ if necessary)
\begin{equation}
\label{e:recthm22}
\alpha_{0} \le 2|\VV|^{-4}. 
\end{equation}

\medskip

This follows from the analysis in \S~\ref{s:basis} from the proof of Theorem \ref{thm:1} together with Remark \ref{r:maincor}.\footnote{Here we use the fact that for each single particle, the additional requirement of not escaping $\cT_{\oo}(s_0)$ by time $t_0$ (or doing so only after activating at least $c' s/\log_ds$ leafs of $\cT_{\oo}(s_0)$) changes all relevant probabilities regarding that particle only by a constant factor.}  

\medskip

Combining \eqref{e:recthm2} with \eqref{e:recthm22} we get that $\alpha_{\ell_*} \precsim 1/|\VV |$ which concludes the proof. We now prove
 \eqref{e:recthm2}. 
\medskip

We condition on the event $A_j$ and bound from above the probability  $\PP_{\la}[A_{j+1}^c \mid A_j ] $. To conclude the induction step it suffices to show that   
\begin{equation}
\label{e:recthm23}
\PP_{\la}[A_{j+1}^c \mid A_j ] \le (d^{\rho_{j+1}-\rho_j} -1) [|\VV|^{-4}+\alpha_{j}]  .
 \end{equation}

On the event  $A_{j}$, at time $t_{j}$ there are at least $c_0' \la|\cT_{\oo}(s_{j} ) | $  activated particles which are either inside $\cT_{\oo}(s_j) $ or at the parent of the root of $\cT_{\oo}(s_j) $ (these are the particles which are unfrozen at time $t_j$). Consider the collection of $d^{\rho_{j+1}-\rho_j}  \precsim (\frac{ \la t_{\ell}}{M d^{4}  \log |\VV|})^{1/2}$ induced subtrees of  $\cT_{\oo}(s_{j+1})$ of size  $|\cT_{\oo}(s_j)| $. Denote them by $\cT_{j+1}^0=\cT_{\oo}(s_j),\cT_{j+1}^1,\ldots,\cT_{j+1}^{d^{\rho_{j+1}-\rho_j}-1} $. Using the definition of $s_{j+1}$ and part (ii) of Proposition \ref{cor:range} in conjunction with Remark \ref{r:frozen}, for each $\ell \in [d^{\rho_{j+1}-\rho_j}-1]$, the probability of each single  activated particles at time $t_{j}$ (out of the $ \ge  c_0' \la|\cT_{\oo}(s_{j} ) |$ activated particles at that time, conditioned on $A_j$) of visiting at least $r $ distinct vertices of at least one $\cT^{\ell'} \subset \cT_{j+1}^\ell $ by time $2t_j$ (without escaping $\cT_{\oo}(s_{j+1}) $ before doing so) is at least\footnote{The applicability of part (ii) of Proposition \ref{cor:range} follows from the assumption that $\la \le \hat c \log n$ and so $t_{\ell} \ge 32 n$ for all $\ell$, provided that $C_1$ is sufficiently large).} \[ \tilde c_0t_jd^{-2\rho_{j+1}+\rho_j+1} \ge \bar c_0t_jd|\cT_{\oo}(s_{j})||\cT_{\oo}(s_{j+1})|^{-2} \ge \bar c_0M \la^{-1} \log |\VV|/|\cT_{\oo}(s_{j})|,\] 
where the second inequality follows from the definition of $s_j$.

Pick $M$ so that $(c_0' \la|\cT_{\oo}(s_{j} ) |) \times ( \bar c_0M \la^{-1} \log |\VV|/|\cT_{\oo}(s_{j})|) \ge 32 \log |\VV|  $. By Chernoff's bound, it follows that on the event $A_j$,  for each $\ell \in [d^{\rho_{j+1}-\rho_j}-1]$ the probability that there is at least one activated particle at time $t_{j}$ that visits at least $r $ distinct vertices of at least one $\cT^{\ell'} \subset \cT_{j+1}^\ell $ by time $2t_j$ is at least $1-|\VV|^{-4}$. As there are $t_j$ time units between time $2t_j$ and $t_{j+1}=3t_j$ this allows one to perform the recursion step. 

The proof of \eqref{e:recthm23}  is concluded  similarly to the proof of Theorem \ref{thm:1} by a union bound over all  $\cT_{j+1}^1,\ldots,\cT_{j+1}^{d^{\rho_{j+1}-\rho_j}-1} $, using the aforementioned estimate obtained from the application of Chernoff's bound and using the recursion for $i=j$ on each of these $d^{\rho_{j+1}-\rho_j}-1 $ trees.
\end{proof}

\section{Lower bounds on the susceptibility}
\label{s:lower}
The approach taken here is borrowed from \cite{SN}. We start with some notation and  auxiliary calculations.
Fix some  graph $G=(V,E)$.  Consider the frog model on $G$ with parameter $\la$. Let $A \subseteq V$.   Denote
\[e_{u,v}(s):=\sum_{i=0}^{s}p^i(u,v) \quad \text{and} \quad m_A(s):=\min_{u \in A}e_{u,u}(s). \]
Let $Y_a(t)$ be the number of particles not occupying $a$ at time 0, other than the planted particle $w_{\mathrm{plant}} $, which visit vertex $a$ by time $t$, if at time 0 all of the vertices are activated. In the notation of \S~\ref{s:construction}, if the index of the planted walker at $\oo$ is 1, then  \[Y_a(t):=|\{(i,v) \in \N \times (V \setminus \{a\}) : i \le |\W_v| \text{, }a \in \{S_j^{v,i}: j \in [t] \},(i,v)\neq (1,\oo) \}|.\] Let $F_t $ be the collection of vertices visited by $\plant$ by time $t$. Clearly $|F_t| \le t+1 $. Let $Z_a(t):=1_{Y_a(t)=0}$. Observe that 
\[\SS(G) > \max \{t:\sum_{v \in V \setminus F_{t}} Z_v(t)>0\}. \]
By Poisson thinning 
\begin{equation}
\label{eq:eat0}
Y_a(t) \sim \Pois(\mu_a(t)),  \quad \text{where} \quad \mu_a(t):= \la \sum_{v \in  V \setminus \{a \}}\Pr_v[T_a \le t], \end{equation}
for all $a \in V$ and $t$. Note that if $G$ is regular then by reversibility $\Pr_v[T_a \le t] \le e_{v,a}(t)=e_{a,v}(t) $ for all $v$ and $t$ and hence
\begin{equation}
\label{e:muat1}
\mu_a(t) =\la \sum_{v \in V \setminus \{a \} }\Pr_v[T_a \le t] \le \la \sum_{v \in V \setminus \{a \} }e_{a,v}(t) =\la [t+1-e_{a,a}(t)] \le \la t  .
\end{equation} 
 When $G$ is regular we also have that $\Pr_v[T_a \le t] \le \frac{e_{v,a}(2t)}{e_{a,a}(t)}  \le\frac{e_{v,a}(2t)}{m_{A}(t)} = \frac{e_{a,v}(2t)}{m_{A}(t)} $ for all $a \in A ,v \in V$ and $t>0$.
Whence for all $a \in A$ and $t>0$
\begin{equation}
\label{eq:eat1}
\mu_a(t) \le \frac{\la}{m_A(t)} \sum_{v \in V \setminus \{a \} }e_{a,v}(2t) \le 2\la t/m_{A}(t).
\end{equation}
Similarly, for $G=\cT_{d,n}$ for all $a \in A \subseteq \cL $ and $t>0$ by reversibility we have that
\begin{equation} 
\label{e: treeexp}
\Pr_v[T_a \le t] \le \frac{e_{v,a}(2t)}{e_{a,a}(t)} \le \frac{e_{a,v}(2t)}{m_{A}(t)},  
\end{equation}
for all $v \in V$. 
Whence for $a \in A  \subseteq \cL $  as in \eqref{eq:eat1} we have 
\begin{equation}
\label{eq:eat1tree}
\mu_a(t) \le 2\la t/m_{\cL}(t).
\end{equation}
\begin{prop}
\label{prop:lower}
Consider the case that $G=(V,E)$ is a finite regular graph and  $A \subseteq V \setminus \{\oo \} $ is arbitrary  or that $G=\cT_{d,n} $ and $A \subseteq \cL \setminus \{\oo\} $. In the above notation, let \[ Z_{A}(t):=\sum_{a \in A \setminus \{\oo \}}Z_a(t)  \quad \text{where} \quad Z_a(t):=1_{Y_a(t)=0}, \]
 \[\alpha_{t}(A):=  \min \{2/m_{A}(t),1 \} \quad 
\text{and} \quad p_A(t):=\max_{a,b \in A:a \neq b}\Pr_a[T_b < t].\]
\begin{itemize}
\item[(1)] For every $a \in A $, $\la>0$ and $t \in \N$ we have that $\E_{\la}[Z_a(t)] \ge e^{- \la t \alpha_t(A)}$. 
\item[(2)] For every $\la>0$ and $t \in \N$ we have that 
\[ \PP_{\la}[Z_{A}(t) \le \sfrac{1}{2} \E_{\la}[Z_{A}(t)] ] \le \sfrac{4}{\E_{\la}[Z_{A}(t)]}+4(e^{2 \la tp_A(t)}-1).   \]
\item[(3)] For every $s>0$ and $t \in \N $ there exists $B_{s} \subseteq A $ such that  \begin{equation}
 \label{e:Bs3}
 |B_{s}| \ge \sfrac{|A|}{1+st^{2}} \quad \text{and} \quad p_{B_s}(t) \le \max_{a,b \in B_{s}:a \neq b}e_{a,b}(t) < \sfrac{1}{st}.
\end{equation}
In particular, 
\begin{equation}
\label{e:Zbs1}
\E_{\la}[Z_{B_{s}}(t)] \ge \frac{|B_s(t)|}{ e^{ \la t \alpha_t(B_{s}(t))}} \ge \frac{|A  |}{(1+st^{2}) e^{ \la t \alpha_t(A)}}, \end{equation}
\begin{equation}
\label{e:Zbs12}
\PP_{\la}[Z_{B_{s}}(t) \le \sfrac{1}{2} \E_{\la}[Z_{B_{s}}(t)]] < \sfrac{4}{\E_{\la}[Z_{B_{s}}(t)]} +4(e^{2\la/s}-1).
\end{equation}
\end{itemize}
\end{prop}
\begin{proof}
For $\E_{\la}[Z_a(t)] \ge e^{- \la t \alpha_t(A)}$ use \eqref{eq:eat0}-\eqref{eq:eat1tree}. Before proving part (2) we need an auxiliary calculation. Denote the  collection of particles, other than $\plant$, whose initial position is neither $a$ nor $b$, which picked a path which reaches both $a$ and $b$ by time $t$ by
\begin{equation*}
\begin{split}
Y_{a,b}(t):=&|\{(i,v) \in \N \times (V \setminus \{a,b\}) \\ & : i \le |\W_v| \text{, }a,b \in \{S_j^{v,i}: j \in [t] \},(i,v)\neq (1,\oo) \}|.
\end{split}
\end{equation*}
Let $Y_a(t) $ and $Y_b(t)$ be as in the previous page. Then by Poisson thinning we have that $W_{1}:=Y_a(t)-Y_{a,b}(t),W_{2}:=Y_b(t)-Y_{a,b}(t) $ and $W_{3}:=Y_{a,b}(t) $ are independent Poisson random variables. Thus
\begin{equation}
\label{eq:cov}
 \frac{ \E_{\la}[Z_a(t)Z_b(t)]}{\E_{\la}[Z_a(t)]\E_{\la}[Z_b(t)]}=\frac{\mathbb{P}_{\la}[W_{1}+W_{2}+W_{3}=0]}{\PP_{\la}[W_{1}+W_{3}=0]\PP_{\la}[W_{2}+W_{3}=0]} =e^{\E_{\la}[W_{3}]} \le e^{2 \la tp_A(t)},    \end{equation}
where in the last inequality we have used the fact that (under $\PP_{\la}$) for all $u \in A$ and $s \ge 1 $ the number of particles that have $u$ as the $s$th position of the walk they picked \[A_{u}(s):=|\{(i,v) \in \N \times V:i \le |\W_v| \text{, }u \in \{S_s^{v,i} \},(i,v)\neq (1,\oo) \}|\] has a Poisson distribution with mean $ \le \la$ (in the regular case it equals $\la$, while for the tree the inequality follows since $u$ is a leaf, as can be seen by increasing the density of particles at each site $w$ site to be $\la \deg (w)$ and observing that generally, starting at each vertex $w$ with a random number of particles which is distributed as  $\mathrm{Pois}( \la \deg (w))$,  independently of the rest of the vertices, is by Poisson thinning a stationary measure when the particles perform simultaneously independent SRWs)  and thus  \[\E_{\la}[W_{3}]= \mathbb{E}_{\la}[Y_{a,b}(t)] \le \sum_{s=1}^t \mathbb{(E}_{\la}[ A_{a}(s)]\Pr_a[T_b \le t-s ]+\mathbb{E}_{\la}[ A_{b}(s)]\Pr_b[T_a \le t-s ]) \le 2 \la t p_A(t). \]

We now prove part (2). We will show that
\[\mathrm{Var}_{\la} (Z_{A}(t)) \le \E_{\la}[Z_{A}(t)]+  (\E_{\la}[Z_{A}(t)])^{2}(e^{2\la tp_A(t)}-1), \] which by Chebyshev's inequality implies that
\[\PP_{\la}[Z_{A}(t) \le \sfrac{1}{2} \E_{\la}[Z_{A}(t)] ] \le \sfrac{4\mathrm{Var}(Z_{A}(t))}{(\E_{\la}[Z_{A}(t)])^2} \le \sfrac{4}{\E_{\la}[Z_{A}(t)]}+4(e^{2 \la tp_A(t)}-1), \]
as desired. Indeed
 $\mathrm{Var}_{\la} (Z_{A}(t)) \le \E_{\la}[Z_{A}(t)]+\rho $, where $\rho:=\sum_{a \neq b \in A\setminus \{\oo \}}\mathrm{Cov}(Z_a(t),Z_b(t)) $. By \eqref{eq:cov} we have that $ \rho  \le (\E_{\la}[Z_{A}(t)])^{2}(e^{2\la tp_A(t)}-1)$. 

 We now prove the existence of the set $B_s$ from part (3).  Note that $\sum_{b \in A \setminus \{a \} }e_{a,b}(t)\le t$ for every $a \in A$ and so
\begin{equation}
\label{e:stsquare}
\{b \in A \setminus \{a \}: e_{a,b}(t) \ge 1/(st)  \}| \le st^{2}.
\end{equation} 
Hence we obtain a set $B_{s}$ satisfying \eqref{e:Bs3} by systematically deleting some vertices from $A$ in a naive manner as follows:  At each stage $r$, if the current set is $A_r$ and there is some $a=a(r) \in A_r$ such that  $\{b \in A_{r}\setminus \{a \}: e_{a,b}(t) \ge \frac{1}{st}  \} \neq \eset $ then we set \[A_{r+1}:=A_r \setminus \{b \in A_{r} \setminus \{a \}: e_{a,b}(t) \ge \sfrac{1}{st}\}.\] Denote the set obtained at the end of this procedure by $B_s$.  Since for all $a,b$ we have that $e_{a,b}(t)=e_{b,a}(t)$ (when $G=\cT_{d,n} $ this follows from the fact that $A \subseteq \cL$) we get that $a(r) \in A(s) $ for every $s \ge r $. Using \eqref{e:stsquare}, it is easy to see that this implies that $|B_s| \ge |A|/ (1+st^2) $ and that $e_{a,b}(t) < \frac{1}{st} $ for all $a \neq b \in B_s$. Finally, we note that \eqref{e:Zbs1} and \eqref{e:Zbs12} are simple consequences of \eqref{e:Bs3} and parts (1) and (2).
\end{proof}

\medskip

\emph{Proof of Theorem \ref{thm:lower}.} We first prove \eqref{e:lower1}.  Let $F_{i} $ be the collection of sites visited by $\plant$ by time $i$. Recall that $\la$ may depend on $|V|$. However, we suppress this dependence from the notation.
Set $k:= \lceil \la^{-1} (\log |V| -4\log \log |V| - \log ( \max \{\sfrac{1}{\la},1 \} ) ) \rceil $. Recall that in the setup of \eqref{e:lower1} $ \lambda^{-1} ( \log |V|)^5 \le |V| $.  Hence we have that $|F_{k}| \le k+1 \le  \half |V|$, with the last inequality holding for all sufficiently large $|V|$. We may assume that $\la \le \log |V| $ as otherwise there is nothing to prove.
  Applying Proposition \ref{prop:lower} with $A=V \setminus F_{k} $,  $t=k$ and $s:=\la \log |V|$ yields that (by \eqref{e:Zbs1}) $\E_{\la}[Z_{B_{s}}(k)] \succsim \frac{|V |-|F|}{sk^{2} e^{ \la k}}  \succsim \log |V|   $  and so (by \eqref{e:Zbs12})  \[\PP_{\la}[Z_{B_{s}}(t) = 0]  \precsim 1/\log |V|+(e^{2\la/s}-1) \le  1/\log |V|+(e^{2/\log |V|}-1)   ,\] which tends to 0 as $|V| \to \infty$. This concludes the proof of \eqref{e:lower1}.

 We now prove \eqref{e:lower2}. Let $\delta \in (0,1)$.  Recall that   $t_{\la,\delta}(G):=\min \{s: \frac{2s \la}{m_V(s)}  \ge (1-\delta)\log |V| \}$, where 
$m_{V}(t)=\min_v \sum_{i=0}^{t}p^i(v,v)$. We seek to show that $\PP_{\la}[\SS(G) < t_{\la,\gd}(G) ]  \le \eps_{|V|} $ for some $\eps_n $ such that $\eps_n \to 0$ as $n \to \infty$, provided that    $\la=\la_{|V|}\gg |V|^{-\delta/11}$.

We first note that we may assume that $\la < \log |V| $ since otherwise we have that $t_{\la,\delta}=1$ (so there is noting to prove). 
 We apply Proposition \ref{prop:lower} with $A=V \setminus  F_t $ (with $F_t$ as above),  \[t= t_{\la,\gd}(G)-1 \quad \text{and} \quad  s:= \sfrac{ |V|^{\delta/2}}{8t^2}  .\] Observe that by assumption on $\la$ we have that $|F_t| \le t+1 \le |V|/2 $.
 As mentioned earlier,  for every regular graph $G=(V,E)$ we have that $t_{\la,0}(G) \le C \la^{-2}\log^2|V| $ (e.g.~\cite[Lemma 2.4]{PS}), and thus by the assumption that   $\la\gg |V|^{-\delta/11}   $   we have that $\la/ s \to 0$ as $|V| \to \infty$. Note that by the definition of  $t_{\la,\gd}(G)$ and our choice of parameters $2\la t/m_V(t) \le (1-\delta) \log |V| $.  Thus  $\E_{\la}[Z_{B_{s}}(t)] \succsim \frac{|V|}{st^{2} e^{ 2\la t/m_V(t) }}  \succsim |V|^{\delta/2}  $.  Hence by \eqref{e:Zbs12} \[\PP_{\la}[Z_{B_{s}}(t) =0] \precsim |V|^{-\delta/2}+(e^{2\la/s}-1)  ,\] which tends to 0 as $|V| \to \infty$.  \qed

\medskip

\subsection{Proof of the lower bound on $\SS(\cT_{d,n})$}
\label{s:lowersustree}
We first give a short proof for the case that $\la_n^{-1} \ll d^{(\sfrac{1}{3} - \eps)n} $ for some fixed $\eps>0$.  In the notation of Proposition \ref{prop:lower}, by \eqref{eq: SRWtree7} we have that $m_{\cL}(\ell) \asymp \log_{d} (d\ell) $ for all $\ell \in \N $ such that $\ell \le d^{n} $. Then for some choice of $c>0$ we have that  $t=t(n,\eps,\la_n):=\lfloor c \eps \la_{n}^{-1} n  \log ( \la_{n}^{-1} n)\rfloor$ satisfies  $2\la_{n} t/m_{\cL}(t) \le ( \sfrac{\eps}{4} \log d )n $ for all sufficiently large $n$.

 As before, let $F_t$ be the set of vertices visited by $\plant$ by time $t$. The proof in the case where  $\la_n^{-1} \ll d^{(\sfrac{1}{3} - \eps)n} $ is concluded by applying part (3) of Proposition \ref{prop:lower},  with $A=\cL \setminus F_{t} $, $t$ as above and $s=d^{n/3}$.  With these choices we have that $|A| >\half d^{n}$, $|B_s(t)| \ge \frac{|A|}{2st^2} \gg d^{\eps n}  $, $\mathbb{E}_{\la_n}[ Z_{B_{s}}(t)] \ge |B_s(t)| e^{-2\la_{n} t/m_{\cL}(t)} \gg d^{\eps n} \times d^{- \eps n/4}=d^{3 \eps n/4} $ and (by \eqref{e:Zbs12}) \[\PP_{\la_n}[ Z_{B_{s}}(t)=0] \precsim \frac{1}{\mathbb{E}_{\la_n}[ Z_{B_{s}}(t)] }+(e^{\la_{n}/s}-1) =o(1) .\]

\medskip   

We now consider the general case   $\la_n d^{n} \ge n \log d $ and $\la_n =o(n) $. By the previous case we can assume that $\la_n^{-1} \gg d^{n/4} $.  We shall reduce the problem to the following.

\begin{prop}
\label{p:killledwalk}
Let $d<s \le nd^{n}$.
Consider a random walk $(\widehat X_i)_{i \ge 0} $ on $\cL$ (the leaf set of $\cT_{d,n}$) which follows the following role. At each step it moves to a random leaf (chosen uniformly at random) with probability $\sfrac{1}{2s}$. Otherwise, if its current position is $x$ it moves to $y$ with probability $\Pr_x[X_{T_{\cL}}^+=y]$, where $(X_i)_{i \ge 0} $ is SRW on $\cT_{d,n}$. Let $\tau_{\mathrm{cov}} $ be the cover time (i.e.\ the first time by which every leaf is visited by the walk at least once). Denote expectation w.r.t.\ $\mathbf{\widehat X}:= (\widehat X_i)_{i \ge 0} $ by $\mathbb{\widehat E}$ and its distribution by $\mathbb{\widehat P} $. Then the following hold
\begin{itemize}
\item[(i)] $\max_{x,y \in \cL} \mathbb{\widehat E}_{x}[T_y] \precsim d^{n}\log_{d} s$. Conversely, if $x,y$ are two leafs    whose graph distance w.r.t.\ $\cT_{d,n}$ is at least $n$ (i.e.\ $\|x \wedge y\| \ge n/2$) then $ \mathbb{\widehat E}_{x}[T_y] \succsim d^{n}\log_{d} s$ (uniformly over all such $x,y$).
\item[(ii)] $\mathbb{\widehat E}[\tau_{\mathrm{cov}}] \asymp d^{n}n\log s  $.
\item[(iii)] $\tau_{\mathrm{cov}}/\mathbb{\widehat E}[\tau_{\mathrm{cov}}] \to 1  $ in probability as $n \to \infty$.
\item[(iv)] There exists some $c_0>0$ such that the number of times that the walk $\mathbf{\widehat X}$ moves to a random leaf chosen uniformly at random before time $\tau_{\mathrm{cov}}$ is $\whp$ at least $c_0 s^{-1} nd^{n}\log s$. 
\end{itemize}   
\end{prop}
Let   $t'=t'(n,\la_n):=\lfloor c ' \la_{n}^{-1} n  \log ( \la_{n}^{-1} n)\rfloor$ for some $c' >0$ to be determined later.
Before proving Proposition \ref{p:killledwalk} we first explain how it implies the assertion that $\SS(\cT_{d,n}) \ge t' $ $\whp$. 
We will show that if initially all particles are activated and each particle walks for $t'$ steps, then $\whp$ some leaf will not be visited by any particle.

 Recall that conditioned on the total number of particles being $m$, each particle, other than $\plant$, starts at a random vertex picked (independently) uniformly at random. In particular, the first leaf it visits has the uniform distribution on the leaf set (i.e.\ starting from the uniform distribution on the vertex set, the hitting distribution of the leaf set is the uniform distribution on $\cL$). Thus  we may assume that each particle starts at a random leaf chosen (independently) uniformly at random. Indeed, when a particle reaches the leaf set it has at most $t'$ remaining steps to perform by time $t'$. Increasing the number of remaining steps to be precisely $t'$ can only increase the probability that every leaf is visited by some particle by time $t'$. 

Similarly, observe that assuming that the initial position of $\plant$ is a leaf can only increase the chance that all leafs are visited by some particle. Hence we may assume that the initial position of $\plant$ is a leaf. By symmetry, we may assume further that its initial position is chosen uniformly at random from $\cL$. 

Next, we argue that by the concentration of the Poisson distribution we may assume that there are  $\lceil 2 \la_{n} d^n \rceil $ particles (as increasing the number of particles can only increase the probability that every leaf is visited by time $t'$) each picking its initial position to be a leaf chosen uniformly at random. 

We now further increase the number of particles to   $\lceil 8 \la_{n} d^n \rceil $ (each particle still starts independently at some leaf chosen  uniformly at random), but now each particle  performs a random number of steps, which is geometrically distributed with mean $2t'$.  We may do so as  each particle has probability at least $1/2$ of performing at least $t'$ steps, and so $\whp$ at least    $\lceil 2 \la d^n \rceil $ particles will walk for at least $t'$ steps. 

We may further increase the duration of their walks by making the number of visits to the leaf set of each walk (rather than the total number of steps of a walk) be geometrically distributed with mean $2t'$. For our purposes we may (and shall) omit all steps in which a particle is not at the leaf set, to obtain a new walk on the leaf set.  
Observe that we may sample the walk  $(\widehat X_i)_{i \ge 0} $ from Proposition \ref{p:killledwalk} with $s=t'$ by concatenating together such walks. By part (iv) of Proposition \ref{p:killledwalk} the number of such walks one has to concatenate until every leaf is visited at least once is $\whp$ at least $ c_0 (t')^{-1}n d^{n}\log t' \ge    \lceil 8 \la_{n} d^n \rceil   $, where the last inequality holds (by the choice of $t'$ and some algebra) provided that $c'$ is taken to be sufficiently small. Hence $\lceil 8 \la_{n} d^n \rceil$ such particles  $\whp$ do not visit every leaf.             
 \qed 

\subsection{Proof of Proposition \ref{p:killledwalk}} 
\begin{proof}
We first note that part (ii) follows from part (i) via Matthew's method (e.g.\ \cite[Theorem 11.2 and Proposition 11.4]{levin2009markov}; For $\mathbb{\widehat E}[\tau_{\mathrm{cov}}] \succsim d^{n}n\log s$ use \cite[Proposition 11.4]{levin2009markov} with $A$ being a collection of $d^{\lfloor n/2 \rfloor }$ leaves whose $\lceil n/2 \rceil$th ancestors are all distinct). Combining (i) and (ii) we get that $\mathbb{\widehat E}[\tau_{\mathrm{cov}}] \gg \max_{x,y}\mathbb{\widehat E}_{x}[T_y]  $, which by a classic result of Aldous \cite{covercon} implies (iii). Part (iv) in turn follows by combining parts (ii) and (iii) (this is left as an exercise). It remains to prove part (i). 

Denote the transition matrix of  $\mathbf{\widehat X}:= (\widehat X_i)_{i=0}^{\infty}$ by $\widehat P $.  The  chain $\mathbf{\widehat X}$ is transitive (i.e.\ for all $x,y \in \cL$ there exists a bijection $\phi:\cL \to \cL$ such that $\phi(x)=y $ and for all $a,b \in \cL$ we have that $\widehat P (a, b)=\widehat P (\phi (a), \phi (b)) $). In particular, its stationary distribution is the uniform distribution on $\cL$. Let $Z_{a,b}:=\sum_{i \ge 0} (\widehat P^{i} (a, b)-d^{-n}) $. It is classical (\cite[Lemma 2.12]{aldous}) that
\begin{equation}
\label{e:ZabZ} \mathbb{\widehat E}_{x}[T_y]=d^n(Z_{y,y}-Z_{x,y})=d^{n}\sum_{i \ge 0} (\widehat P^{i} (y, y)-\widehat P^{i} (x, y)) \end{equation}
(where the last equality follows from the fact that by transitivity\footnote{Transitivity is used to argue that $|\widehat P^{2i} (a, b)-d^{-n}| \le |\widehat P^{2i} (a, a)-d^{-n}|  $  for all $i$ (e.g.\ \cite[Equation (3.59)]{aldous})}  for all $a,b \in \cL$ we have  \[\half \sum_{i \ge 0} |\widehat P^{i} (a, b)-d^{-n}| \le  \sum_{i \ge 0} |\widehat P^{2i} (a, b)-d^{-n}| \le~ \sum_{i \ge 0} |\widehat P^{2i} (a, a)-d^{-n}| \le \sum_{i \ge 0} \beta^{i} < \infty,  \] where $\beta$ is the second largest  eigenvalue of $\widehat P^2$).  Let $\rho$ be the first time at which the walk moves to a random position chosen uniformly at random. We couple $\mathbf{\widehat X} $ with SRW on $\cT_{d,n}$, $(X_i)_{i=0}^{\infty}$, by time $\rho-1$ in an obvious manner. Namely, we let $\widehat X_i $ be the location of the $i$th visit of  $(X_i)_{i=0}^{\infty}$ to $\cL$ for $i < \rho$, where $\rho$ is independent of  $(X_i)_{i=0}^{\infty}$. Let $N(a) $ be the number of visits of  $\mathbf{\widehat X} $  to $a$ by time $\rho-1 $.

 Then $\rho \sim \mathrm{Geometric}(\sfrac{1}{2s}) $ and by the memoryless property of the Geometric distribution and the Markov property (used in the second equality) we get that
\[\sum_{i \ge 0} (\widehat P^{i} (y, y)-\widehat P^{i} (x, y))=\mathbb{\widehat E}_y[N(y)]-\mathbb{\widehat E}_x[N(y)]= \widehat \Pr_x[  T_{y} > \rho ] \mathbb{\widehat E}_y[N(y)].  \]
Using \eqref{e:ZabZ},  to conclude the proof it remains to show that
 \begin{equation}
 \label{e:Wald4} \mathbb{\widehat E}_y[N(y)] \asymp \log_d s, \end{equation}
and  that for every pair of leafs $x,y $ such that $\|x \wedge y\| \ge n/2 $ we have that  $\mathbb{\widehat P}_x[\rho < T_y ] $ is (uniformly) bounded away from 0. We first verify that for such  $x,y $  indeed $\mathbb{\widehat P}_x[\rho < T_y ] \succsim 1 $. 

Using network reductions  it is not hard to verify that  for SRW on $\cT_{d,n}$ we have that the effective resistance $x \neq y \in \cL$ is $2\|x \wedge y \| $ and so  $\mathbb{E}_x[T_y] \asymp d^n \|x \wedge y \|  $, uniformly  for all  $x \neq y \in \cL$.\footnote{Implicitly, we are using the fact that by symmetry, $\mathbb{E}_x[T_y]=\mathbb{E}_y[T_x]= \sfrac{1}{2}\mathbb{E}_x[T_{yx}] $  for all  $x,y \in \cL $, where $T_{yx}$ is the commute-time from $x$ to $y$ (i.e.\ the time to hit $y$ and then return to $x$; Indeed, in general the expectation of the commute-time from $x$ to $y$    can be written in terms of the effective resistance between $x$ and $y$, e.g.\ \cite[Chapter 9]{levin2009markov}).} 

Using the Markov property $\ell-1$ times in conjunction with Markov's inequality we get that (for SRW on $\cT_{d,n}$) $\Pr_x[T_y >2 \ell \max_{a,b} \mathbb{E}_a[T_b] ] \le 2^{-\ell} $ for all $\ell \in \N$ and hence for all $x,y$ such that $\|x \wedge y \|$ we have that  $\mathbb{E}_x[T_y^{2}] \asymp \max_{a,b} (\mathbb{E}_a[T_b])^2$ and so   $\mathbb{E}_x[T_y^{2}] \asymp (\mathbb{E}_x[T_y])^2$. It follows from the Paley-Zygmund inequality that $\Pr_x[T_y> c_0 d^n n] \ge c_1>0 $ for some absolute constants $c_0,c_1>0$. It is easy to see  that this implies that starting from $x$  for some $c_0',c_{1}'>0$ the probability that the number of returns to the leaf set prior to $T_y$ is at least $c_0' d^n n$ is at least $c_1'$. In particular, (using $s \le nd^{n}$ and the aforementioned coupling)  the walk  $(\widehat X_i)_{i \ge 0} $ satisfies that started from $x$ the probability that  $ \rho \le  T_{y} $ is bounded from below (uniformly in all $x,y \in \cL$ with  $\|x \wedge y \| \ge n/2$).

We now prove \eqref{e:Wald4}. 
 As we later verify, since $\rho$ is independent of $(X_i)_{i \ge 0} $ (in the aforementioned coupling) it suffices to show that for all $t \ge 1$ we have that
 \begin{equation}
 \label{e:Wald4} M(t):= \mathbb{ E}_y[\text{number of visits to $y$ by the } t\text{-th visit to the leaf set}] \asymp \log_d (dt)+d^{-n}t. \end{equation}
By \eqref{eq: SRWtree6''} we have for SRW on $\cT_{d,n}$ that $A(t):=\mathbb{ E}_y[\text{number of visits to $y$ by time } t] \asymp \log_d(dt)+d^{-n}t$ for all $t \ge 1$.  Thus for all $t \ge 1$ 
\[M(t)\ge A(t)\asymp  \log_d (dt)+d^{-n}t .  \]
Using the fact that the total amount of time that SRW on $\cT_{d,n}$ starting from $\cL$ spends at $\cL$ by time $t$ is highly concentrated around its mean (this follows from the decomposition from the proof of \eqref{eq: SRWtree5}) which is $\ge t/4$ (see the paragraph following \eqref{e:leafmon} for a justification of this lower bound), it is not hard to verify that for all $t \ge 1 $ 
\[M(t) \precsim A(8t)\asymp \log_d(dt)+d^{-n}t .  \]
Finally, using the independence of $\rho $ and $(X_i)_{i \ge 0}$ we get that
\[\mathbb{\widehat E}_y[N(y)]= \sum_{\ell=1}^{\infty}\mathbb{\widehat P}_y[\rho = \ell] M( \ell) \asymp  \sum_{\ell=1}^{\infty}\mathbb{\widehat P}_y[\rho = \ell]  \log_d (d \ell)+  \sum_{\ell=1 }^{\infty}\mathbb{\widehat P}_y[\rho = \ell]d^{-n} \ell \asymp \log_d s , \]
where we used $ A(t) \le M(t) \precsim A(8t) $.
\end{proof} 
 
 \section{The frog model on the complete graph}
\label{s:Kn}
Recall that in the coupon collector problem (e.g.~\cite[Proposition 2.4]{levin2009markov}) there are $n$ coupons labeled by $[n]$. At each step a coupon is picked from the uniform distribution on $[n]$. 
If $\tau_i$ is the first time at which $i$ different coupons were collected,  then for all $1 \le i < n $ the distribution of $\eta_i:=\tau_{i+1} - \tau_i $ is Geometric with parameter $\frac{n-i}{n}$. Moreover,   $\eta_0,\eta_1,\ldots,\eta_{n-1}$ are independent. The coupon collector time is defined to be $\tau_n$. It is  concentrated around time $n \log n $.

 It is not hard to see that for all $i<n/6$ we have that the probability that $\tau_{3i}-\tau_i> 4i $ decays exponentially in $i$. Similarly, if $\eps<1/10$ is fixed then there exists some $C'>0$ such that the probability that $\tau_{\lceil (1-\eps/6)n \rceil}-\tau_{\lceil n/100 \rceil}>C'n | \log \eps | $ decays exponentially in $n$.

We note that the same remains true if at each step the next coupon is picked from the uniform distribution over $[n] \setminus \{\text{the last coupon that was picked} \} $   (the only difference is that now the marginal distribution of  each $\eta_i$ is Geometric with parameter $\frac{n-i}{n-1}$ for $i \ge 1$).  

\subsection{Proof for Proposition \ref{prop:Kn}} 
Let $(\la_n)_{n \in \N}$ be such that $\lim_{n \to \infty} \la_n n = \infty $. We first show that
\[\forall \eps \in (0,1), \quad \lim_{n \to \infty} \PP_{\la_{n}}[(1-\eps)\la_{n}^{-1} \log n  \le  \SS(K_n) \le \lceil (1+\eps)\la_{n}^{-1}\log n \rceil ]=1. \]

Fix $\la= \la_n >0$ and $\eps \in (0,1/10)$. Let $\la$ be the density of particles. Recall that $\W_i$ is the collection of particles occupying site $i$ at time $0$ and that for $B \subseteq [n]$ the collection of particles which initially occupy $B$ is denoted by $\W_{B}:=\cup_{b \in B}\W_b $. Label the vertex set by $[n]$ such that $\oo=1 $. 

  Observe that by symmetry assuming that all of the $\Pois(\la n)$ particles in $\W_{[n]} \setminus \{\plant\} $ start from vertex $1$, and are thus activated at time $0$, can only increase the probability that $\SS(K_n) <(1-\eps) \la^{-1} \log n $.\footnote{Recall that $\SS(K_n)$ is defined as the minimal lifetime for which every site is visited before the process dies out, as opposed to the minimal lifetime for which each particle is activated.} To see this, observe that when a particle is activated in the usual setup, moving this particle to site $1$ does not change the law of $\SS(K_n)$ (as it does not change the law of the identity of  the vertices activated by this particle).

By the concentration of the Poisson distribution around its mean we have that $\whp$ $|\W_{[n]}| \le (1+\eps/2) \la n  $. The claim that $\whp$ $\SS(K_n) \ge (1-\eps) \la^{-1} \log n $ now follows from the solution to the coupon collector problem  and the fact that for all  $\eps \in (0,1/10)$ \[[(1-\eps) \la^{-1} \log n] \cdot [(1+\eps/2) \la n] < (1-\sfrac{\eps}{4} )n \log n.\]      

 We now show that   $\whp$ $\SS(K_n) \le \lceil (1+\eps) \la^{-1} \log n \rceil $.  Consider the case that the lifetime of the particles is \[s_n:=\lceil \la^{-1} (1+\eps)\log n \rceil .\]  
For some constant $C>0$ to be determined later, we set \[M=M(\eps,\la):=C\lceil \sfrac{1}{\la}   | \log \eps |\rceil .\]
We now consider a dynamics in which the particles in $\W_1$ are initially active and have lifetime $s_{n} $ while the rest of the particles have lifetime $M$. Crucially, in the  variant of the frog model we now consider, if a particle $w'$ is activated by a particle $w$, then $w'$ begins its walk only once the particle $w$ has finished performing its walk (which is either of length $s_{n}$ or of length $M$, depending on whether $w \in \W_1 $ or not).  We think of $\W_1$ as the $0$th generation. We think of the particles they activated during their length $s_n$ walks as the first generation.  Similarly, for all $i \ge 1$ we think of the particles activated by the particles from the $i$th generation (during their length $M$ walk) as the $(i+1)$th generation. 

Denote the $i$th generation by $ \mathcal{G}_i$ and set  $\mathcal{G}:=\bigcup_{i \ge 1} \mathcal{G}_i$. We will later argue that if we let the particles in $\mathcal{G} $ perform additional $s_n-M $ steps (so that the total number of steps they perform is $s_n$) then $\whp$ every vertex is visited by some particle in  $\mathcal{G} $. As we later explain in more detail, by a coupon collector argument this boils down to arguing that $\whp$ $(s_n-M) |\mathcal{G} | \ge (1+c \eps)n \log n $ for some constant $c>0$. By the concentration of the Poisson distribution this in turn boils down to arguing that $\whp$ the vertices which are visited by the above dynamics (i.e.\ the initial positions of the particles in $\GG$) is $\whp$ of size at least $(1-c' \eps)n $ for some small $0<c'<1$ (below we call this set $U$).  In order to do so we now describe the evolution of this model as an exploration process. The set $\W_{A_i}$ below is the $i$th generation in the above interpretation.

Stage 1: Let $A_0=\{1\}$. First expose the first $s_n $ steps of the walk of each $w \in \W_1=\W_{A_0}$.   Let $D_1$ be the union of the ranges of those walks. Set $U_1:=\{1\}=A_0 $ and $A_1:=D_1 \setminus U_1$. 

Stage 2:  Expose the first $M$ steps of each $w \in \W_{ A_1}$. Denote the union of the ranges of these walks by $D_2$. Set $U_2:=U_1 \cup A_1 = A_0 \cup A_1 $ and $A_2:= D_2 \setminus U_2 $. 

Stage $i+1$: As long as $A_i $ is non-empty continue as in stage 2. Namely,  expose the first $M$ steps of each $w \in \W_{A_i}$. Denote the union of the ranges of these walks by $D_{i+1}$. Set $U_{i+1}:=U_{i} \cup A_i = \cup_{j=0}^i A_j $ and $A_{i+1}:= D_{i+1} \setminus U_{i+1} $.   

\medskip

Let $i_*:=\min\{i:A_{i} = \eset \} $ be the first stage $i$ at which no new vertices (i.e.\ ones not belonging to $U_{i}  $) were discovered by the length $M$ walks of the particles in $\W_{A_{i}} $.
Let $U$  be the set of all vertices that are activated by the above dynamics. Crucially, the set $U$ does not change as a result of the convention that each particle begins its walk only after the particle which activated it has finished performing its walk (of length $s_n$ for generation 0 or length $M$ for the rest of the particles). However, with this convention we have that $U=U_{i_*}$.
We will show that $\whp$ we have that 
\begin{equation}
\label{e:Uislarge}
|U \setminus \{1\} |>(1-\eps/6)n \quad \text{and} \quad |\W_{U \setminus \{1\} } | \ge(1-\eps/6)|U \setminus \{1\} | >(1-\eps/3)n \la.
\end{equation}
We first explain how this implies the conclusion that $\SS(K_n) \le s_n $ $\whp $. Indeed this follows  by letting the particles in $\W_{U \setminus \{1\} } $ perform additional $s_n-M$ steps and arguing that when both inequalities in \eqref{e:Uislarge} hold (using $\eps < 1/10$) for sufficiently large $n$ we have that \[(s_n-M)|\W_{U \setminus \{1\}} | \ge (s_n-M) [\la(1-\eps/3)n] \ge (1+\eps/6)n \log n, \] and so by comparison with the coupon collector problem, $\whp$ the particles in $\W_{U \setminus \{1\} } $  will together cover the graph in these additional $s_n-M$ steps.

We first treat the case that $\la^{-1}  \log n > n/10 $. In this case, it is easy to see that by the above discussion and by the concentration of the Poisson distribution around its mean, $\whp$ we have that $|A_1|>n/100$ and that  $|\W_{A_{1}}| >(1-\eps/6) \la |A_1|$. Recall that in the coupon collector problem the probability that $\tau_{\lceil (1-\eps/6)n \rceil}-\tau_{\lceil n/100 \rceil}>C'n | \log \eps | $ decays exponentially in $n$.
  Thus $C$ in the definition of $M=C\lceil \sfrac{1}{\la}   | \log \eps |\rceil$ is taken to be sufficiently large then $\whp$ (the randomness is over $|A_1|$) \[[(1-\eps/6) \la |A_1|] \cdot M > C'n | \log \eps |  \] and so $\whp$ $|A_1 \cup A_2|>(1-\eps/6)n $. Using again the concentration of the Poisson distribution, we also have that  $|\W_{A_{2}}| >(1-\eps/6)\la|A_2|$ $\whp$. This implies \eqref{e:Uislarge} when  $\la^{-1}  \log n > n/10 $. 

We now consider the case that  $\la^{-1}  \log n \le n/10 $ and $\la \ll \log n $.   
 Note that $\whp$ $|A_1|> \half \la^{-1} \log n$  and (as before) $|\W_{A_1} | \ge (1-\eps/6)\la |A_1|   $. Let $j^*:=\inf \{j : |A_j| \ge n/100 \}$. We now argue that if $C$ from the definition of $M$ is taken to be sufficiently large then $\whp$ the following hold (1) $j^{*} \le \log_2 n$, (2)   $|A_1| \le \half |A_2| \le \sfrac{1}{4}|A_3| \le \cdots \le 2^{-(j^*+1)}|A_{j^{*}}|$ and (3) for all $i \le j^*$ we have $|\W_{A_i} | \ge (1-\eps/6)\la |A_i|$. Once this is established, the proof of \eqref{e:Uislarge}  is concluded in a similar fashion to the case that  $\la^{-1}  \log n > n/10 $.

 To see that (1)-(3) indeed hold $\whp$ observe that (2) implies (1) and that conditioned on  $|A_1|> \half \la^{-1} \log n$,   $|\W_{A_1} | \ge  (1-\eps/6)\la |A_1|   $ and 
 on    $|A_1| \le \half |A_2| \le \sfrac{1}{4}|A_3| \le \cdots \le 2^{-(i+1)}|A_{i}|<n/100$ we have that the probability that $|A_{i+1}|<2  |A_{i}|  $ is exponentially small in $\la |A_i|  \ge \la 2^{i} |A_1| \ge 2^{i-1} \log n  $ (because in the coupon collector problem for all $k<n/6$ we have that the probability that $\tau_{3k}-\tau_k> 4k $ decays exponentially in $k$). Conditioned further also on  $|A_{i+1}| \ge 2  |A_{i}|$,  the probability that  $|\W_{A_{i+1}} | < (1-\eps/6)\la |A_{i+1}|$ is exponentially small in $\la |A_{i+1}|\le 2^{i} \log n$. Thus the probability that there is some $i<j^*$ which violets the conditions   $|A_{i+1}| \ge 2  |A_{i}|  $  and   $|\W_{A_{i+1}} | \ge (1-\eps/6)\la |A_{i+1}|$ is $o(1)$.

\medskip

We now treat the case that $\la \succsim \log n $. In fact, we shall treat the case that $\la \gg 1 $. In this case, we may replace $M$ with $\widehat M=1$ in the above analysis (this makes a difference if  $\la \succsim \log n $, and then possibly $M > \lceil (1+\eps) \la^{-1} \log n \rceil $). Fix some $\eps \in (0,1/10)$.
Let $\ell^*:= \ell^*(\eps,\la,n):=\inf \{j : |A_j| \ge C_0 \la^{-1} n \log (1/\eps) \}$. 

 Much as before, we have that $\whp$  the following hold (1) $\ell^*\le \log_{\la/2} n$, (2)   $|A_1| \le \sfrac{2}{\la} |A_2| \le \sfrac{2^{2}}{\la^2}|A_3| \le \cdots \le (\sfrac{2}{\la})^{j^*+1}|A_{\ell^*}|$ and (3) for all $i \le \ell^*$ we have $|\W_{A_i} | \ge (1-\eps/6)\la |A_i|$. Conditioned on this, we pick an arbitrary collection $A \subseteq A_{\ell^*} $ of size $\lceil C_0 \la^{-1} n | \log \eps| \rceil$. By the concentration of the Poisson distribution, $\whp$ $|\W_A|> \half C_0  n |\log \eps|$. We expose the first  step of the particles in $\W_A$ and denote by $J$ the collection of vertices  visited by at least one particle from  $\W_A$ during their first step. By the analysis of the coupon collector problem, provided that $C_0$ is taken to be sufficiently large, we have that $\whp$ $|J \setminus A | \ge (1-\eps/6)n$ and (as before)  $|\W_{J \setminus A  } | \ge (1-\eps/6)\la |J \setminus A |$. From here the proof is concluded as before.

\medskip

We now prove \eqref{e:rem1.2}. Fix $\eps=1/20$. Observe that in the regular frog model the time until all of the vertices in  $U=\bigcup_{i=1}^{j^*} A_i   $  are activated is at most $s_n+j^*M $.\footnote{Implicitly, we are considering a coupling of it with the aforementioned variant of the model, in which the number of particles at each site is the same for both models, and also the walk picked by each particle is the same.} It follows  from the above analysis that $\whp$ we have that $\CT(K_n) \le 2 s_n + j^*M  \le 2 s_n + M \log_2 n $. This concludes the proof when $\la \le 2 $. 

Now consider the case that $\la>2$. By monotonicity it suffices to consider the case that $\la \le \sqrt{n} $. Using the same reasoning as above one can show that in the above notation (if   $C$ in the definition of $M=C\lceil \sfrac{1}{\la}   | \log \eps |\rceil$ is taken to be sufficiently large) we have that  $\whp$ the following hold (1) $j^{*} \le \log_{\la} n$, (2)   $|A_1| \le \sfrac{1}{\la} |A_2| \le \sfrac{1}{\la^2}|A_3| \le \cdots \le \sfrac{|A_{j^{*}}|}{\la^{j^*+1}}$ and (3) for all $i \le j^*$ we have $|\W_{A_i} | \ge (1-\sfrac{\eps}{6})\la |A_i|$. Thus as above $\CT(K_n) \le 2 s_n + j^*M \le 2s_n + M \log_{\la} n $.
  \qed

\section*{Acknowledgements}
The author would like to thank Itai Benjamini for suggesting the problems studied in this paper and for many useful discussions. The author would also like to thank Matthew Junge and Tobias Johnson for helpful discussions. Finally the
author would like to express his gratitude to the anonymous referee for suggesting improvements to the presentation of this work. 

\bibliographystyle{amsplain}
\bibliography{Frog}

\section{Appendix: Transition probabilities, hitting-times and range growth estimates for SRW on finite $d$-ary trees}
\label{s:aux}
In this section we develop some auxiliary results concerning SRW on $\cT_{d,n} $. We start by giving some results concerning the distribution of the hitting time of the root (which can be used to study the hitting time distribution of an arbitrary $y \in \VV_{d,n}$ starting from any vertex in $\cT_y$). We later give some sharp estimates on $p^t(x,y)$ and then use them to study the range of SRW on $\cT_{d,n}$.
 
\medskip

We start by recalling some notation and introducing some further notation which shall be used throughout the section.
Fix some $n,d$. Recall that $\mathbf{r}$ is  the root of $\mathcal{T}_{d,n}=(\VV_{d,n},\EE_{d,n})$,  $\cL_i$ is its $i$th level and $\cL$ is its leaf set. Recall that the distances of $v \in \VV_{d,n}$ from the root and from the leaf set are denoted, respectively, by $|v|$ and $\|v\|$. 

Let $(X_s)_{s=0}^{\infty}$ be SRW on $\mathcal{T}_{d,n} $. Let $Y_s:=|X_s|$. Then $(Y_s)_{s=0}^{\infty}$ is a birth and death chain on $]n[:=\{0,1,\ldots,n\}$ with transition matrix $Q(0,1)=1=Q(n,n-1)$ and $Q(i,i-1)=(d+1)^{-1}=1-Q(i,i+1)$ for $i \in [n-1]$. Denote its law starting from state $i$ by $\Pr_i^{]n[}$, the corresponding expectation by $\mathbb{E}_i^{]n[} $ and  its stationary distribution by $\pi$, where 
\begin{equation}
\begin{split}
\label{e:piofY}
&\pi_0:=\sfrac{d}{2|\EE_{d,n}|}=d \cdot \sfrac{(d-1)}{2d(d^n-1)}, \\ & \pi_n:=\sfrac{|\cL|}{2|\EE_{d,n}|}=d^{n} \cdot \sfrac{(d-1)}{2d(d^n-1)} \in (1/4,1/2], 
\\ & \pi_j:=\sfrac{|\cL_{j}|(d+1)}{2|\EE_{d,n}|}=\sfrac{(d^{2}-1)d^{j}}{2d(d^n-1)} \quad \text{for } 1 \le j <n.
\end{split}
\end{equation}
Elementary considerations involving effective-resistance and network reductions (e.g.~\cite[Example 9.9]{levin2009markov}) show that for all $i \in ]n-1[$,
 \begin{equation}
\label{eq: SRWtree8} q_i:= \Pr_i^{]n[}[T_{0}<T_{n}]=\frac{d^{-i}-d^{-n}}{1-d^{-n}} \text{ and } \Pr_\mathbf{r}[ T_{\mathbf{r}}^+ < T_{\cL} ]=q_1= \frac{d^{-1}-d^{-n}}{1-d^{-n}}.
\end{equation}
\begin{lem} Let $o \in \cL $. In the above notation we have that
 \begin{equation}
 \label{eq: SRWtree3} \max_{v \in \VV_{d,n}}\mathbb{E}_v[T_{\mathbf{r}}]=\mathbb{E}_o[T_{\mathbf{r}}] = \frac{2d(d^n-1)}{(d-1)^2}-\frac{d+1}{d-1}n \pm O(1).
\end{equation}
\end{lem}
\emph{Proof:} Observe that if $v \in \cL_i $ then by symmetry $\mathbb{E}_v[T_{\mathbf{r}}]=\mathbb{E}_i^{]n[}[T_{0}] =\sum_{j=0}^{i-1} \mathbb{E}_{j+1}^{]n[}[T_{j}]$. Hence $\max_{v \in \VV_{d,n}}\mathbb{E}_v[T_{\mathbf{r}}]=\mathbb{E}_n^{]n[}[T_{0}] =\sum_{j=0}^{n-1} \mathbb{E}_{j+1}^{]n[}[T_{j}]$. Using the general formula \[\mathbb{E}_{j+1}^{]n[}[T_{j}]=\frac{\pi \{j+1,j+2,\ldots,n \} }{\pi(j)Q(j,j+1)}=\frac{1}{\pi(j)Q(j,j+1)}-\frac{\pi(]j[)}{\pi(j)Q(j,j+1)} \] for the expected crossing time of an edge in a birth and death chain (e.g.\ \cite[Proposition 5.3 part (b)]{aldous}) and the identities $\frac{1}{\pi(j)Q(j,j+1)}=\frac{d+1}{d\pi(j)}  =\frac{2(d^n-1)}{(d-1)d^{j}} $ for $j=0$ while $\frac{1}{\pi(0)Q(0,1)}=\sfrac{2(d^{n}-1)}{d-1} $ and  $|\frac{\pi(]j[)}{\pi(j)Q(j,j+1)}-\frac{d+1}{d-1}| \precsim d^{-(j+1)} $ for $j>0$ while  $\frac{\pi(]0[)}{\pi(0)Q(0,1)}=1$  (which follow from \eqref{e:piofY}) we have that \[ \sum_{j=0}^{n-1} \mathbb{E}_{j+1}^{]n[}[T_{j}]-O(1)=\sum_{j=0}^{n-1}\frac{2(d^n-1)}{(d-1)d^{j}}-\frac{d+1}{d-1}=  \frac{2d(d^n-1)}{(d-1)^2}-\frac{d+1}{d-1}n - O(1). \] 

 We give an alternative proof involving the Doob's transform corresponding to conditioning on $T_{\rr}^+<T_{\cL}^+$, as we shall later use the fact that this transform exhibits a bias towards the root (see \eqref{eq: SRWtree9} below for a precise statement).  By the Markov property and symmetry we have that $\mathbb{E}_\mathbf{r}[T_{\rr}^+ \mid T_{\mathbf{r}}^+ > T_{\cL} ]=\mathbb{E}_\mathbf{r}[T_{\cL} \mid T_{\mathbf{r}}^+ > T_{\cL} ]+\mathbb{E}_o[T_{\mathbf{r}}] $.  Thus 
 \[\frac{2(d^n-1)}{d-1}=2|\EE_{d,n}|/d=\mathbb{E}_\mathbf{r}[T_{\mathbf{r}}^+]=\mathbb{E}_\mathbf{r}[T_{\mathbf{r}}^+  \cdot 1_{ T_{\mathbf{r}}^+ < T_{\cL}} ]+\mathbb{E}_\mathbf{r}[T_{\mathbf{r}}^+  \cdot 1_{ T_{\mathbf{r}}^+ > T_{\cL}} ] \]
\[=\Pr_\mathbf{r}[ T_{\mathbf{r}}^+ < T_{\cL} ]\mathbb{E}_\mathbf{r}[T_{\mathbf{r}}^+ \mid T_{\mathbf{r}}^+ < T_{\cL} ]+\Pr_\mathbf{r}[ T_{\mathbf{r}}^+ > T_{\cL} ](\mathbb{E}_\mathbf{r}[T_{\cL} \mid T_{\mathbf{r}}^+ > T_{\cL} ]+\mathbb{E}_o[T_{\mathbf{r}}] ). \]
By \eqref{eq: SRWtree8} $\Pr_\mathbf{r}[ T_{\mathbf{r}}^+ < T_{\cL} ]=q_{1}=\frac{1}{d}-O(d^{-n}) $.  Thus to prove \eqref{eq: SRWtree3} it suffices to show that
\begin{equation}
\label{e:TLcon}
\mathbb{E}_\mathbf{r}[T_{\cL} \mid T_{\mathbf{r}}^+ > T_{\cL} ]=\E_0^{]n[}[T_n \mid T_0^+ > T_n ] = \frac{d+1}{d-1}n - O(1).
\end{equation}
Using the Doob's transform (e.g.~\cite[Section 17.6.1]{levin2009markov}) we get that conditioned on $T_{0}^+ < T_{n}^+$ the chain $(Y_i)_{i \in \Z_+} $ exhibits a bias towards $0$. More precisely, for every $j \in ]n[$, $ i \in [n-1]$ and $s \ge1$ we have that
 \begin{equation}
\label{eq: SRWtree9}
\begin{split} & \Pr_{j}^{]n[}[Y_{s+1}=i-1 \mid Y_{s}=i, s \le  T_{0}^+ < T_{n}^+]  \\ & =\frac{q_{i-1}}{(d+1)q_i}= \frac{d(1-d^{i-n-1})}{(d+1)(1-d^{i-n})}  \ge \frac{d}{d+1}.
\end{split}
 \end{equation}
Thus conditioned on $T_{0}^+ < T_{n}^+ $ we have that $Y_{\min \{i,T_0^{+} \}}+\frac{d-1}{d+1} \min \{i,T_0^{+} \} $ is a super-martingale.   By optional stopping  $\mathbb{E}_{1}^{]n[}[T_{0}^+ \mid T_{0}^+ < T_{n} ] \le \frac{d+1}{d-1} $  and so
\begin{equation}
\label{e:returnrr}
\mathbb{E}_\mathbf{r}[T_{\mathbf{r}}^+ \mid T_{\mathbf{r}}^+ < T_{\cL} ] = \mathbb{E}_{0}^{]n[}[T_{0}^+ \mid T_{0}^+ < T_{n} ]   \le 1+\sfrac{d+1}{d-1}=\sfrac{2d}{d-1} .
\end{equation}
Using the aforementioned formula for the expected crossing time of an edge in a birth and death chain,
 \[ \mathbb{E}_\mathbf{r}[T_{\cL}  ]= \E_0^{]n[}[T_n]=n+ \sum_{i=1}^{n}\frac{1-d^{-(k-1)}}{d-1} =\frac{d+1}{d-1}n-O(1). \]
  Since the number of visits to $\rr$ prior to $T_{\cL}$ (starting from $\rr$) follows a Geometric distribution with mean $1/\Pr_{\rr}[T_{\mathbf{r}}^+ > T_{\cL} ]$ using Wald's identity we get that 
\[  \mathbb{E}_\mathbf{r}[T_{\cL}  ] - \mathbb{E}_\mathbf{r}[T_{\cL} \mid T_{\mathbf{r}}^+ > T_{\cL} ] = \frac{\mathbb{E}_\mathbf{r}[T_{\mathbf{r}}^+ \mid T_{\mathbf{r}}^+ < T_{\cL} ]\Pr_{\rr}[T_{\mathbf{r}}^+ < T_{\cL} ] }{\Pr_{\rr}[T_{\mathbf{r}}^+ > T_{\cL} ]}=O(1/d) .\]
Combining the last two equations yields \eqref{e:TLcon}. \qed

\medskip

The following lemma asserts that the law of the hitting time of the root $\rr$ starting from the leaf set is in some sense similar to the Geometric distribution of mean $\asymp d^{n-1}$. Equations \eqref{eq: SRWtree5'up} and \eqref{eq: SRWtree5'} will not be used in what come. We prove them as we think they are interesting in their own right. 
\begin{lem}
\label{lem: SRWontree}
 For every $v \in \VV_{d,n}$ and $i \in \N$, 
 \begin{equation}
 \label{eq: SRWtree4}
  \Pr_v[T_{\mathbf{r}} \ge  16id^{n-1}] \le 2^{-i} .
\end{equation}
 Conversely, there exist some $c,c'>0$ such that
  \begin{equation}
 \label{eq: SRWtree5}
\forall \, 16n \le t \le d^{n-1}, \quad  ctd^{-(n-1)} \le \min_{v \in \VV_{d,n} } \Pr_v[T_{\mathbf{r}} \le t] \le td^{-(n-1)} .
\end{equation}
  \begin{equation}
 \label{eq: SRWtree5''}
\forall\, t \ge 16 n, \quad \exp(-ctd^{-(n-1)}) \le \min_{v \in \VV_{d,n} } \Pr_v[T_{\mathbf{r}} \le t] \le \exp(-c'td^{-(n-1)}) .
\end{equation}
Moreover, there exists some $C_1>0$ such that for all $v \in \cL$
  \begin{equation}
 \label{eq: SRWtree5'up}
\max_t \Pr_v[T_{\mathbf{r}} = t] \le C_{1}d^{-(n-1)} ,
\end{equation}
while for all  $C_2>0$ there exists some $c_1>0$, depending only on $C_2$, such that for every $v \in \cL$ and all $t$ such that $16n \le t \le C_2 d^{n-1}$ and $t-n $ is even we have that
  \begin{equation}
 \label{eq: SRWtree5'}
 \quad  \Pr_v[T_{\mathbf{r}} = t] \ge c_{1}d^{-(n-1)} .
\end{equation}
\end{lem}
\begin{proof}
We first prove \eqref{eq: SRWtree4}. By \eqref{eq: SRWtree3} $\max_{v \in \VV_{d,n}}\mathbb{E}_v[T_{\mathbf{r}}] \le 8d^{n-1} $. Hence by Markov's inequality $\max_{v \in \VV_{d,n}}\Pr_v[T_{\mathbf{r}} \ge16d^{n-1}] \le 1/2 $. Averaging over $(X_{16jd^{n-1}})_{j=1}^i $ and applying the Markov property $i$ times  yields \[\max_{v \in \VV_{d,n}}\Pr_v[T_{\mathbf{r}} \ge16id^{n-1}] \le \left( \max_{v \in \VV_{d,n}}\Pr_v[T_{\mathbf{r}} \ge16d^{n-1}] \right)^i \le 2^{-i}. \] 
We now prove \eqref{eq: SRWtree5}.
First, observe that it suffices to consider the case that $v=o \in \cL $ (by symmetry, any leaf achieves the minimum). Consider the aforementioned birth and death chain $(Y_s)_{s=0}^{\infty}$ started from state $n$. The excursions from $n$ to itself can be partitioned into ones during which the chain does not visit 0 and ones in which it does. Let $\xi_i$ be the length of the $i$th excursion of the former type and let \[ Z:=T_0-\sup \{s<T_0:Y_s=n \} \] be the length of the journey from $n$ to 0 in the first excursion of the latter type. Denote the number of excursions of the former type prior to $T_0$ by \[U:=|\{0<s<T_0 :Y_s=n \}|.\] By \eqref{eq: SRWtree9}, \[\Pr_n^{]n[}[Z>4n] \le \Pr_0^{]n[}[T_{n}>4n]  \le e^{-c_3n}.\] Let $t \ge 16n$ and  $\ell:= \lceil t/8 \rceil$. Let $S_{\ell}:=\sum_{i=1}^{\ell} \xi_i  $. Using the fact that for all $i$ we have that $\E_n^{]n[}[\xi_i] \le \frac{2d}{d-1} \le 4$ (cf.\ \eqref{e:returnrr}) and  $\E_{n}^{]n[}[e^{\alpha \xi_i}]<\infty $ for some $\alpha>0$ we get that \[\Pr_{n}^{]n[}\left[S_{\ell}> t-4n \right] \le \Pr_{n}^{]n[}\left[S_{\ell} \ge 6 \ell  \right] \le e^{-c_4 \ell} \le e^{-c_4t/8}. \] By \eqref{eq: SRWtree8},   Bonferroni inequality and the fact that $\ell= \lceil t/8 \rceil \le d^{n-1}/8$ we have that \[\Pr_{n}^{]n[}[U \le \ell] \ge \ell q_{n-1}-\binom{\ell}{2}q_{n-1}^2 \ge 0.5 \ell q_{n-1} \ge c_5 t d^{-(n-1)}.\] Finally,
\[\Pr_o[T_{\mathbf{r}} \le t] \ge \Pr_n^{]n[} \left[U \le \ell,\, S_{\ell}  \le t-4n,\, Z \le 4n \right] \] \[ =\Pr_n^{]n[}[U \le \ell]\Pr_n^{]n[}\left[S_{\ell}\le t-4n \right]\Pr_n^{]n[}[Z \le 4n] \ge ctd^{-(n-1)}.     \]
Conversely, as $\xi_i \ge 2$ for all $i$, if we write $m:=\left\lfloor \frac{t-n}{2} \right\rfloor$ we have that \[\Pr_o[T_{\mathbf{r}} \le t] \le \Pr_n^{]n[} \left[U \le m  \right] \le mq_{n-1} \le td^{-(n-1)}.\] The proof of \eqref{eq: SRWtree5''} is similar to that of  \eqref{eq: SRWtree5}, using the fact that the law of $U $ under $\Pr_n^{]n[} $ is Geometric with parameter $q_{n-1}$. It  is left as an exercise.

We now prove \eqref{eq: SRWtree5'up}. Let  $v \in \cL$. The law of $T_{\rr}$ starting from $v$ is the same as the law of $T_0$ starting from $n$ in the birth and death chain $(Y_i)_{i=0}^{\infty}$, where $Y_i:=|X_i|$.   We treat the case that $n$ is even. The case that $n$ is odd is analogous and is left as an exercise.

By Proposition \ref{p:log-con}  the law of $T_0/2 $ is a convolution of Geometric distributions, one of which is of parameter $\gamma \precsim d^{-(n-1)}$. Since the probability of the mode of a convolution is at most the probability of the mode of each term in the convolution, it follows that the mode of $T_0$ has probability at most $\gamma \precsim d^{-(n-1)}$.

We now show that for all $C_2 \ge 1 $ there exists some $c_1>0$ (depending only on $C_2$)  such that for all $16n \le t \le C_2 d^{n-1} $ such that $t-n$ is even we have that $\Pr_v[T_{\rr}=t] \ge c_1 d^{-(n-1)} $. Denote  the mode of $T_{\rr}$ by $m_*$. Let $t \equiv n$ mod 2. By Proposition \ref{p:log-con}   $\Pr_o[T_{\mathbf{r}} = t] $ is non-decreasing in $t$  for   $t \le m_*$ (such that  $t \equiv n$ mod 2) and non-increasing for $t \ge m_*$ (such that  $t \equiv n$ mod 2). Thus for $ t \ge 16 n$ such that  $t \equiv n$ mod 2 if    $t \le \min \{ m_*,C_{2}d^{n-1} \}$ then by \eqref{eq: SRWtree5''}  \[\Pr_o[T_{\mathbf{r}} = t] \ge \sfrac{2}{t}\Pr_o[T_{\mathbf{r}} \le t] \ge  \sfrac{2}{t}(\hat c_1t/d^{n-1}   ) =2 \hat c_1 d^{-(n-1)}       ,\] where $\hat c_1 $ may depend on $C_2$ (in practice $m_* = O(d^{n-1}) $ and so $\hat c_1$ need not depend on $C_2$), while if $t > m_*$ then for all $i \in \N$ 
\[ \Pr_o[T_{\mathbf{r}} = t] \ge \sfrac{1}{i}  \Pr_o[t \le T_{\mathbf{r}} < t+2i]    .\]
Substituting $i=\lceil C_3 d^{n-1} \rceil $ for some ($C_3$ sufficiently large in terms of $C_2$) and applying \eqref{eq: SRWtree5''} concludes the proof.
\end{proof}

Before obtaining estimates on the transition probabilities we need two additional lemmas.

\begin{defn}
\label{d:M}
Consider SRW on $\cT_{d,n}$. Let $\tau_i:=T_{\cL_{n-i}}$. Let  \[M_t:= \max \{i: \tau_i \le t \}=\max\{\|X_i \|:i \le t \}\]
  be the maximal distance of the walk from $\cL$ by time $t$.
\end{defn}
The following simple lemma describes the way in which we shall derive heat kernel estimates for SRW on $\cT_{d,n}$. 
\begin{lem}
\label{lem:auxcalculationtreeSRW}
Consider SRW on $\cT_{d,n}$. For every $t \in \N $ and $x,y \in \VV_{d,n}$,
\begin{equation}
\label{eq: SRWtree1}
p^t(y,x) = \sum_{i \ge \|x \wedge y \|}\Pr_y[ |X_t|=|x| , M_t=i]d^{-(i-\|x\|)}.
 \end{equation}
 \begin{equation}
\label{eq: P2txx}
p^{2t}(x,x)  \le \sum_{j \ge \|x\| }\Pr_x[| X_{2t} |=|x|, M_{t}=j]d^{\|x\|-j}. \end{equation}
\end{lem}
\emph{Proof:}
Consider the walk started from $y$. Recall that $\overleftarrow{y_\ell}$ is the $\ell$th ancestor of $y$.  Let $i \ge \|y\|$. Denote $z_i:=\overleftarrow{{y_{i-\|y\| }}} \in \cL_{n-i}$. If $M_t=i \ge \|y\| $, then  $T_{z_i} \le t$ and $\{X_s:s \le t \} \subseteq \cT_{z_i}$. If in addition  $X_t \in \cL_{n-j} $ for some $j \le  i$, then by symmetry $X_t$ is equally likely to be at each of the $d^{i-j} $ vertices of $\cL_{i-j}(\cT_{z_i}) $.     This clearly implies \eqref{eq: SRWtree1}. For \eqref{eq: P2txx} observe that similar reasoning yields that for  $j \ge \|x\| $, \begin{equation*} \begin{split} & \Pr_x[ X_{2t}=x \mid | X_{2t} |=|x|, M_{t}=j] \\ & = \sum_{\ell:\, \ell \ge j}d^{\|x\|-\ell}\Pr_x[ M_{2t}=\ell \mid | X_{2t} |=|x|, M_{t}=j] \le d^{\|x\|-j}. \quad \text{\qed}  \end{split} \end{equation*}
Recall that $\pi$ is the stationary distribution of $(Y_s)_{s=0}^{\infty}$, where $Y_s:=|X_s|$. 
\begin{lem}
\label{lem:auxcalculationtreeSRW2}
Consider SRW on $\cT_{d,n}$.  For all $0 \le i \le j \le n$ and $t \ge j-i $ we have that   
\begin{equation}
\label{eq: SRWtree2}
\forall y \in \cL_i, \quad \Pr_{y}[X_{t} \in \cL_{j} ]  \le  \Pr_{y}[T_{\cL}>t]+ 4 \pi_{j}  \Pr_{y}[T_{\cL} \le t]. 
\end{equation}
Let $f(t):=t^{-3/2}[2d/(d+1)^2]^t$.  Let $y \in \VV_{d,n}$. Then
\begin{equation}
\label{eq: SRWtree0}
\max_{u} P^{2t}(y,u)/\deg(u) \le Cf(t)+d^{-\|y\|}.  
\end{equation} 
\end{lem}
\begin{proof}
We first prove \eqref{eq: SRWtree2}. We may assume  $t-(j-i)$ is even, as otherwise $\Pr_{y}[X_{t} \in \cL_{j} ]=0$.   Since the $L_{\infty}$ distance\footnote{The $L_{\infty}$ distance of a distribution $\mu$ from $\pi$ is $\|\mu - \pi \|_{\infty,\pi}:= \max_{x}|\frac{\mu(x)}{\pi(x)}-1 |$. If $\Pr_{\mu}^t$ denotes the law of the chain at time $t$  when the initial distribution is $\mu$, then $\|\Pr_{\mu}^t - \pi \|_{\infty,\pi} $ is non-increasing.} from $\pi$ is non-increasing in time  \[\max_{0 \le \ell \le n}\sup_{s \ge 0}\Pr[|X_s|=\ell \mid |X_0|=n]/\pi_{\ell} =\max_{0 \le \ell \le n}\sup_{s \ge 0}\Pr[Y_s=\ell \mid Y_0=n]/\pi_{\ell} = 1/\pi_n \le 4. \]
Averaging over $T_{\cL}$ and applying the Markov property, we get that
\begin{equation*}
\begin{split}
& \Pr_{y}[X_{t} \in \cL_{j} ]-\Pr_{y}[T_{\cL}>t] \le\Pr_{y}[T_{\cL} \le t,X_{t} \in \cL_{j} ] 
\\ &  \le \Pr_{y}[T_{\cL} \le t] \sup_{s \ge 0}\Pr[|X_s|=j \mid|X_0|=n ]
\\ &  \le \Pr_{y}[T_{\cL} \le t]\pi_{j}[\max_{0 \le \ell \le n}\sup_{s \ge 0}\Pr[|X_s|=\ell \mid |X_0|=n ]/\pi_{\ell} ] \le  4\pi_j \Pr_{y}[T_{\cL} \le t] .
\end{split}
\end{equation*} 
We now prove \eqref{eq: SRWtree0}. Clearly,
\begin{equation}
\label{e:P2tuv}
P^{2t}(y,u)=\Pr_y[X_{2t}=u,T_{\cL}>2t]+\Pr_y[X_{2t}=u,T_{\cL} \le 2t].   \end{equation}
   By \eqref{e:P2tuv2} we have that $\Pr_y[X_{2t}=u,T_{\cL}>2t] \le C f(t) $. We now bound $\Pr_y[X_{2t}=u,T_{\cL} \le 2t]$.

Given $T_{\cL} \le 2t $ and that $M_{T_{\cL} }=j \ge \|y\|$ (where $M_{\bullet}$ is as in Definition \ref{d:M}), we have that $X_{T_{\cL}} $ has the uniform distribution over $\cL(\cT_{\overleftarrow{y_{j-\|y\|}}})$ and is independent of $T_{\cL}$ (conditioned on $M_{T_{\cL} }=j$).   Using the fact that the $L_{\infty}$ distance from the stationary is non-increasing in time, by averaging over $T_{\cL}$ and  $M_{T_{\cL} }$ and applying the triangle inequality we get that \[\max_{u \in V} \Pr_y[X_{2t}=u,T_{\cL} \le 2t]/\deg(u) \le 1/|\cL(\cT_y)|=d^{-\|y\|}.  \]This, in conjunction with
\eqref{e:P2tuv} and \eqref{e:P2tuv2}, concludes the proof of \eqref{eq: SRWtree0}.
\end{proof}

\begin{lem}
Let $\TT $ be the infinite $d$-ary tree. Let  $\mathbb{T}_{d+1}^{\mathrm{reg}}$ be the $(d+1)$-regular tree.\footnote{The only difference between $\TT $ and $\mathbb{T}_{d+1}^{\mathrm{reg}} $ is that the degree of the root of the former is $d$ while that of the latter is $d+1$.} Let $(Z_i)_{i=0}^{\infty}$ (resp.\ $(\widehat Z_i)_{i=0}^{\infty}$) be a SRW on $\TT$ (resp.\ $\mathbb{T}_{d+1}^{\mathrm{reg}}$). Denote the corresponding distributions, starting from vertex $z$ by $\Pr_z^{\TT}$ and $\Pr_z^{\mathbb{T}_{d+1}^{\mathrm{reg}}}$, respectively, and the corresponding transition kernels for time $i$ by $P_{i}^{\mathbb{T}_{d}}$ and $P_{i}^{\mathbb{T}_{d+1}^{\mathrm{reg}}}$, respectively. Then 
\begin{equation}
\label{e:P2tuv2}
\max_{y,u \in \VV_{d,n} } \Pr_y[X_{2t}=u,T_{\cL}>2t] \le \max_{z,u} \Pr_z^{\TT}[Z_{2t}=u] \le\frac{d+1}{d} \max_z \Pr_z^{\TT}[Z_{2t}=z] \le Cf(t). \end{equation}
\end{lem}
\emph{Proof of \eqref{e:P2tuv2}. }Indeed the first inequality follows by a straightforward coupling argument. The last inequality is well-known (e.g.\ \cite[Equation (2.5)]{ram}). We now prove the middle inequality. Let $v$ be an arbitrary vertex in $\mathbb{T}_{d+1}^{\mathrm{reg}}$. Then   \[\max_{z,u} \Pr_z^{\TT}[Z_{2t}=u] \le \frac{d+1}{d}  \max_{z,u} \Pr_z^{\mathbb{T}_{d+1}^{\mathrm{reg}}}[\widehat Z_{2t}=u] \] \[= \frac{d+1}{d}  \Pr_v^{\mathbb{T}_{d+1}^{\mathrm{reg}}}[\widehat Z_{2t}=v] \le \frac{d+1}{d} \max_{z} \Pr_z^{\TT}[Z_{2t}=z], \]
where the equality follows from the fact that  $\mathbb{T}_{d+1}^{\mathrm{reg}}$  is vertex-transitive and hence (\cite[Corollary 7.3]{aldous})  $\max_{x,y}P_{2t}^{\mathbb{T}_{d+1}^{\mathrm{reg}}}(x,y) =P_{2t}^{\mathbb{T}_{d+1}^{\mathrm{reg}}}(v,v) $   and the two inequalities are left as an exercise.\footnote{Hint: we may think of $\TT$ as a subgraph of $\mathbb{T}_{d+1}^{\mathrm{reg}} $ and couple  $(Z_i)_{i=0}^{\infty}$ with $(\widehat Z_i)_{i=0}^{\infty}$ starting from the same vertex, such that whenever $\widehat Z_i $ is in $\TT$ we have that $\widehat Z_i=Z_i$.} \qed

\begin{lem}
\label{lem: SRWontree2} 
Consider SRW on $\cT_{d,n}=(\VV_{d,n},\EE_{d,n})$. There exist some absolute constants $c,c_{1},C$ such that for all  $u,v \in \VV_{d,n} $ all  $1 \le k \le n$ and all $t \in [d^{k-1},d^k] $ we have that 
 \begin{equation}
 \label{eq: SRWtree6}
\frac{c_{1}p^t(u,v)}{\deg(v)} \le d^{-k}+d^{-(k-1)} \exp(-c td^{-(k-1)} ).
\end{equation}
Conversely, for every $u \in \VV_{d,n}, v \in \cL $, $ \max\{1, \| u \wedge v  \|\} \le k \le n $ and    $ t \in [d^{k-1},d^{k}] $  
\begin{equation}
 \label{eq: SRWtree7}
CP^{t}(u,v) \ge  [d^{-k}+  d^{-(k-1)} \exp(-c' td^{-(k-1)} )] \cdot 1_{t-\|u\|  \text{ is even} }
\end{equation}
Finally, there exists some absolute constants $C_{1},C_2,c_{0}>0$ such that for all $t \ge C_{1}d^{n-1} \log d $,  \begin{equation}
 \label{eq: Linftymixingtree}
\max_{u,v \in \VV_{d,n} :\, t-(\|u\|- \|v\| ) \text{ is even} }|\frac{p^t(u,v)}{\deg(v)/|\EE_{d,n}|}-1| \le C_{2} \exp [- c_{0}t d^{-(n-1)} ] . 
\end{equation} 
In particular, for every $v \in \cL$ and every $t \ge 1 $ we have that
 \begin{equation}
 \label{eq: SRWtree6''}
\sum_{i=0}^{t} p^i(v,v) \asymp  \log_d (dt) + t d^{-n}. 
\end{equation} 

\end{lem}
For our purposes, we shall only need to apply \eqref{eq: SRWtree7} when $u,v \in \cL$. We could not find a reference even for this case. We prove the more general case as we find this tail estimate to be of self interest.
\begin{rem}
Observe that SRW on $\cT_{d,n}$ has period 2. This is the reason for the form of the normalization in \eqref{eq: Linftymixingtree}. However, it follows straightforwardly from Lemma \ref{lem: SRWontree2}  that the $L_{\infty}$ mixing time of both rate 1 continuous-time and discrete-time lazy (obtained by replacing $P$ by $\half (I+P)$) SRWs on  $\cT_{d,n}$ is $\asymp d^{n-1} \log d $ (uniformly in $d$). We record this observation in Corollary \ref{c:mixing}, although it will not be used in what comes.
\end{rem}
\begin{rem}
Similar reasoning as in \eqref{eq: SRWtree7} shows that
for every $u \in \VV_{d,n}, v \in \cL $ and  $ t \le [16\| u \wedge v  \|,d^{\| u \wedge v  \|-1}] $ such that $t-\|u\|$ is even, we have that  $P^{t}(u,v) \succsim \sfrac{1}{d^{\| u \wedge v  \|}} \cdot \sfrac{d^{\| u  \|-1}+t}{d^{\| u \wedge v  \|-1}}$.  
\end{rem}
\emph{Proof of Lemma \ref{lem: SRWontree2}: }
We first prove \eqref{eq: SRWtree6}.  It is standard that $\max_{x,y} \frac{p^{2t}(x,y)}{\deg(y)} =\max_{x} \frac{p^{2t}(x,x)}{\deg(x)}$, for all $t \ge 0 $ (e.g.\ \cite[Equation (3.60)]{aldous}). Since $\max_{x,y} \frac{p^{t}(x,y)}{\deg(y)}  $ is non-increasing in $t$ it suffices to consider even times and the case $u=v=:x$.   Let $k \le n$ and $d^{k-1} \le t \le d^{k}$. By \eqref{eq: SRWtree0} we only need to treat the case that $k > \|x\|$. Recall that by \eqref{eq: P2txx}
\[p^{2t}(x,x)  \le\sum_{j \ge \|x\| }\Pr_x[X_{2t} =x, M_{t}=j] = \sum_{j \ge \|x\| }\Pr_x[| X_{2t} |=|x|, M_{t}=j]d^{\|x\|-j}.\]
We shall now show that the sum is dominated by the contribution coming from the terms corresponding to $j \in \{k-1,k\}$. We first argue that we may ignore all $j > k +C \log_d t $ in the above sum as by \eqref{eq: SRWtree5} we have that $\Pr_{x}[M_t \ge k +C \log_d t ] \le \frac{1}{t^2}$.  

Denote the $(j-\|x\|)$-th ancestor of $x$ by  $z_j:=\overleftarrow{x_{j-\|x\|}} \in \cL_{n-j} $.   Then by \eqref{eq: SRWtree4}  (applied to $\cT_{z_j} $ instead of $\cT_{d,n}$)   for every $ \|x\| < j<k$ we have that
\begin{equation}
\label{eq: Mtsmall}
\begin{split}
 \Pr_x[M_{t} = j-1 ] & \le \Pr_x[M_{t} < j ] =\Pr_x[T_{z_j}>t] \\ & \le \Pr_x[T_{z_j}>(\sfrac{dt}{16 d^{j}} )16d ^{j-1}]\le 2^{-\left\lfloor \frac{dt}{16d^{j}} \right\rfloor }.
\end{split}
\end{equation}

Using \eqref{eq: SRWtree1} and the fact that $d^{k-1} \le t \le d^k  $ we get that for  $j \le k +\lceil C \log_d t \rceil $ we have      \[\Pr_x[ X_{2t} \in \cL_{|x|} \mid M_{t}=j] \le \Pr_{n-j}^{]n[}[T_{n}>t]+ 4 \pi_{|x|}  \Pr_{n-j}^{]n[}[T_{n} \le t]=(4+o(1)) \pi_{|x|}.    \]   This together with \eqref{eq: Mtsmall},  \eqref{eq: P2txx}, \eqref{eq: SRWtree5} and the estimate $\pi_{|x|} =\Theta( \deg(x)d^{-\|x\|}) $ yield
\begin{equation}
\label{e:5.9apl}
\begin{split}
& cp^{2t}(x,x)  \le    \sum_{j = \|x\| }^{\min \{ k+\lceil C \log_d t \rceil , n\}} \Pr_x[M_{t}=j] \Pr_x[| X_{2t} |=|x| \mid M_{t}=j]d^{\|x\|-j}
\\ & \le (4+o(1)) \pi_{|x|} \left( \sum_{j = \|x\| }^{k-2} 2^{-\left\lfloor \frac{t}{16d^{j}} \right\rfloor } d^{\|x\|-j}+\sum_{j=k-1}^{\min \{ k+\lceil C \log_d t \rceil , n\}} \Pr_x[M_{t}=j]d^{\|x\|-j} \right) \\ & \le C'  \pi_{|x|} \max \{\Pr_x[M_{t}=k-1] d^{\|x\|-k+1}, d^{\|x\|-k} \} \\ & \le \bar C \deg(x)\max \{d^{-(k-1)}\exp(-c td^{-(k-1)} ),d^{-k}  \}.  
\end{split}
\end{equation}
This concludes the proof of \eqref{eq: SRWtree6}.
We now prove \eqref{eq: SRWtree7}. 

Fix some $u \in \VV_{d,n} , v \in \cL $. 
Now fix some $ \| u \wedge v  \| \le k \le n $ and some $ d^{k-1} \le t \le d^{k}$. For simplicity, we assume that $t$,  $\|u\|$ and $n$ are even.  Using \eqref{eq: SRWtree5''} it is not hard to show that there exist some $\delta,\gd'>0$ such that  \[\Pr_u[M_{t/2}=k=M_{t} ] \ge \delta' \Pr_u[M_{t/2} =k] \ge \gd  .\] 
\[\Pr_u[M_{t/2}=k-1=M_{t} ] \ge \gd' \Pr_u[M_{t} =k-1]  \ge \delta \exp(-c' td^{-(k-1)} ).\]
Let $i \in \{k-1,k \}$. Using the fact that (by a standard coupling argument) for all $\ell \in \N $ the conditional distribution of $(Y_s)_{s =0}^\ell$ started from $n-i$, conditioned on $T_{n-i-1}>\ell $, stochastically dominates its unconditional distribution (used in the second inequality below), by averaging over $T_{\cL_{n-i}} $  (used in the first inequality) and over $T_{n} $  (used in the penultimate inequality) we get that
\begin{equation}
\label{e:leafmon}
\begin{split} & \Pr_u[|X_{t} |=n\mid M_{t/2}=i=M_{t} ] \ge \inf_{s \le t/2:\, s-i \text{ is even}  } \Pr_u[|X_{t} |=n\mid T_{\cL_{n-i}}=s,T_{\cL_{n-i-1}}>t  ] \\ & \inf_{s \le t/2:\, s-i \text{ is even}  } \Pr_u[|X_{t} |=n\mid T_{\cL_{n-i}}=s ] = \inf_{\ell \ge t/2:\, \ell-i \text{ is even}  }\Pr_{n-i}^{]n[}[Y_{\ell}=n]   \\ & \ge \Pr_{n-i}^{]n[}[T_{n} \le t/2 ] \inf_{ \ell \ge0}\Pr_{n}^{]n[}[Y_{2\ell}=n]  \ge c_2, 
\end{split}
\end{equation}
where we have used the fact that $\Pr_{i}^{]n[}[Y_{2\ell}=i]$ is non-increasing in $\ell$ for all $i$ (this is true for every reversible Markov chain, e.g.\ \cite[Proposition 10.25]{levin2009markov}) and hence $\inf_{ \ell \ge0}\Pr_{n}^{]n[}[Y_{2\ell}=n]=\lim_{\ell \to \infty} \Pr_{n}^{]n[}[Y_{2\ell}=n]=2\pi_n$. Finally, since $d^{k-1} \le t \le d^k $ we get that
\begin{equation*}
\begin{split}
P^{t}(u,v) & \ge \sum_{i=k-1}^k \Pr_u[M_{t/2}=i=M_{t} ]\Pr_u[|X_{t} |=n\mid M_{t/2}=i=M_{t} ]d^{-i} \\ & \succsim \max \{d^{-(k-1)}\exp(-c' td^{-(k-1)} ),d^{-k}  \}.
\end{split}
\end{equation*}

We now prove \eqref{eq: Linftymixingtree}. Fix some  $t \ge 2s:= \lceil C_{3}d^{n-1} \log d \rceil $.  By conditioning on $X_1 $ (when $t$ is odd) it suffices to consider the case that $t$ is even. We may therefore consider the SRW whose transition kernel is $P^2$. It is not hard to verify that on each of its two irreducibility classes its stationary distribution is  $\widehat \pi(v):=\deg(v)/|\EE_{d,n}| $. By \eqref{eq: SRWtree6} $ \max_v \frac{p^{2s}(v,v)}{\widehat \pi(v)} \precsim 1$. Let $\gamma $ be the  largest non-unit eigenvalue of $P^2$ (all of its eigenvalues lie in $[0,1]$ and two of them equal $1$ as it has 2 irreducibility classes). Let $t_{\mathrm{mix}}^{(1)} $ and $t_{\mathrm{mix}}^{(2)}$ be the total-variation mixing times of the walk corresponding to $P^2$ on its two irreducibility classes (in practice, these two mixing times are close to one another). Let $t_{\mathrm{mix}} = \max \{ t_{\mathrm{mix}}^{(1)},t_{\mathrm{mix}}^{(2)}\} $. By reversibility $\frac{1}{1-\gamma} \precsim t_{\mathrm{mix}} $ (e.g.\ \cite[Theorem 12.4]{levin2009markov}). We argue that  $ t_{\mathrm{mix}} \precsim \max_x  \E_{x}[T_{\rr}] \precsim d^{n-1}$ (where $\E_{x}[T_{\rr}]$ is defined w.r.t.\ SRW on $\cT_{d,n}$).  This follows from a standard coupling argument (e.g.\ \cite[Example 5.3.4]{levin2009markov}) in which two walks (starting from levels of the same parity) are coupled so that after the first time they reach the same level at some time (which occurs typically in $\asymp n$ steps) both walks are always  at the same level. Thus they collide once reaching the root. Thus the hitting time of the root can be used to bound the coupling time (the time until the two walks collide in the aforementioned coupling) from above, which by the coupling characterization of the total-variation distance implies that $t_{\mathrm{mix}} \precsim \max_x  \E_{x}[T_{\rr}]$, as claimed (e.g.\ \cite[Theorem 5.4]{levin2009markov}).  Thus $ 1-\gamma\succsim \frac{1}{\max_x  \E_{x}[T_{\rr}]} \succsim  d^{-n+1}. $ By Poincar\'e inequality (applied to the walk corresponding to $P^2$) together with the estimates  $ \max_v \frac{p^{2s}(v,v)}{\widehat \pi(v)} \precsim 1$ and  $\gamma=1-(1-\gamma) \le  e^{-(1-\gamma)} \le \exp[-c_{0}d^{-(n-1)} ] $, we get that for all $i \ge 0$ \[ \max_v | \frac{p^{2(s+i)}(v,v)}{\widehat \pi(v)} -1| \le \max_v | \frac{p^{2s}(v,v)}{\widehat \pi(v)} -1| \gamma^{i} \precsim  \gamma^{i}   \precsim \exp [- c_{0}i d^{-(n-1)} ]. \] Finally, by reversibility $\max_v | \frac{p^{2(s+i)}(v,v)}{\widehat \pi(v)} -1| = \max_{v,u: \|u\|-\|v\| \text{ is even} } | \frac{p^{2(s+i)}(v,u)}{\widehat \pi(u)} -1|$ (e.g.\ \cite[Equation (2.2)]{spectral}) for all $i,s \ge 0$.  
\qed

\begin{cor}
\label{c:mixing}
Consider  continuous-time with exit rate 1 from each vertex. Denote its distribution at time $t$ started from $x$ by $\mu_t^x(\bullet)$ and let $\mu(\bullet) :=\sfrac{ \deg( \bullet )}{2 |\EE_{d,n}|} $ be its stationary distribution. Let $\tau_{\infty}(\eps):= \inf \{t: \max_{x,y}  |\mu_t^{x}(y)/\mu(y) - 1| \le \eps \}$ be its $\epsilon$ $L_{\infty}$ mixing time. Then there exists $C \ge 1 $ such that for all $d \ge 2 $ and $n$ we have that $\sfrac{1}{C} \le \tau_{\infty}(1/e)/(d^{n-1} \log d) \le C $.
\end{cor}
\begin{proof}
We first note that $\max_{x,y}  |\mu_t^{x}(y)/\mu(y) - 1|=\max_{x}  \mu_t^{x}(x)/\mu(x) - 1 $ (e.g.\ \cite[Equation (2.2)]{spectral}).
If $(X_i)_{i = 0}^{\infty}$ is a discrete-time SRW on $\cT_{d,n} $ and $(N(t))_{t \ge 0 } $ is an independent homogeneous rate 1 Poisson process on $\R_+$, then $(X_{N(t)})_{t \ge 0}$ is a continuous-time SRW on $\cT_{d,n} $. It follows that $\mu_t^{x}(x)= \sum_{i \ge 0}\mathbb{P}[N(t)=2i]p^{2i}(x,x)$. The claim that  $C\tau_{\infty}(1/e) \ge d^{n-1} \log d $ now follows from \eqref{eq: SRWtree7}, together with the estimate $|\tau_{\infty}(1/e) - \tau_{\infty}(\eps)| \precsim d^{n-1} |\log \eps |   $  for all $\eps$, which follows from the Poincar\'e inequality (cf.\ the proof of \eqref{eq: Linftymixingtree}).  We now prove that  $\tau_{\infty}(1/e) \le C d^{n-1} \log d $. Let $t \in \N$. Since $P^{2i}(x,x) $ is decreasing in $i$ (\cite[Proposition 10.25]{levin2009markov})
\[\mu_{4t}^{x}(x) \le P^{2t}(x,x) \mathbb{P}[N(4t) \ge2 t, N(4t) \text{ is even}  ]+  \mathbb{P}[N(4t) <2 t] \] \[ \le P^{2t}(x,x) \mathbb{P}[ N(4t) \text{ is even}  ] + (e/4)^{-2t} \le (\sfrac{1}{2}+\sfrac{C}{\sqrt{t}} )P^{2t}(x,x) + (e/4)^{-2t} ,  \]
where we have used the estimates $\mathbb{P}[N(4t) <2 t] \le (e/4)^{-2t} $ (\cite[Exercise 20.6]{levin2009markov}) and $| \mathbb{P}[ N(4t) \text{ is even}  ]-\half  |\precsim t^{-1/2}$.\footnote{In fact, using Poisson thinning, one can generate $(N(s))_{s \ge 0}$ by sampling a rate 2 Poisson process $(M(s))_{s \ge 0}$  and then keeping each point independently w.p.\ $1/2$. Hence for all $s \ge 0$ we have that $\mathbb{P}[ N(s) \text{ is even}  ]=\mathbb{P}[ M(s)=0 ]+\half \mathbb{P}[ M(s)>0 ]= \half (1+ e^{-2s}) $.} The claim that   $\tau_{\infty}(1/e) \le C d^{n-1} \log d$ now follows by \eqref{eq: Linftymixingtree} by substituting above $t=C d^{n-1} \log d$ and noting that $2\mu(x)=\deg(x)/|\EE_{d,n}|$.
\end{proof}     
\emph{Proof of Proposition \ref{cor:range}:}
We begin with the proof of \eqref{eq: Range3}. Fix $k,t,A$ and $y$ as in \eqref{eq: Range3}.  We first argue that there exist $C_1,c_1>0$ such that
\begin{equation}
 \label{eq: Range2}
\forall \, 8\|y\| +d^{\|y\|-1}\le t \le d^{n}, \quad  c_1 g(t) \le \E_y[|R_t \cap \cL |] \le C_1g(t). 
\end{equation}
 Clearly, it suffices to consider the case that $y \in \cL$ and that $d^{\|y\|-1}\le t \le d^{n}$.\footnote{Indeed, starting from $y$  the leaf set of $\cT_y(d^k) $ is hit by time $8\|y\|$ with probability bounded from below. This clearly implies that it suffices to show that $c_1 g(t) \le \E_y[|R_t \cap \cL |]  $ when $y  $ is a leaf, in order to obtain the same bound for general $y$ (up to a time shift by $8\|y\|$ time units). Conversely, conditioning on hitting the leaf set by time  $8\|y\|$ can only increase the mean of $|R_t \cap \cL |$.} Let $y \in \cL$. For every leaf $v \in \cL$ and $s \ge 0$ let \[e_{v}(s):=\sum_{i=0}^s p^i(y,v)=\sum_{i=0}^s p^i(v,y),\] \[a(s):=\sum_{i=0}^s p^i(y,y) \asymp \log_d (ds) \quad \text{and} \quad p_v(s):=\Pr_y[T_v \le s], \] where the estimate $a(s)\asymp \log_d (ds) $ follows from Lemma \ref{lem: SRWontree2}. As   $e_v(r)=\sum_{i=0}^{r} \Pr_y[T_v = i]a(r-i)$ (used in the second and third inequalities below), for all $v \in \cL$ and $r \in \N$ we have  
\[\frac{ce_v(t)}{1+\log_d t} \le \frac{e_v(t)}{a(t)} \le p_v(t) \le \frac{e_v(2t)}{a(t)} \le \frac{Ce_v(2t)}{1+\log_d t} ,  \]
 Using the fact that $\sum_{v \in \cL}e_v(s)=\E_y[|\{i:i \le s,\, X_i \in \cL \}|]=\Theta ( s)$ (as it is at most $s$ and by the paragraph following \eqref{e:leafmon} is at least $s \pi_n \ge s/4 $, where as before $\pi$ is the stationary distribution of $(|X_j|)_{j=0}^{\infty}$) summing over $v \in \cL$ concludes the proof of  \eqref{eq: Range2}.

We now prove \eqref{eq: Range3}. 
Again, it suffices to consider the case that $z \in \cL(\cT_y)$. Let $z \in \cL(\cT_y)$. We argue that 
\[ \forall s_1,s_2 \ge 1, \quad \Pr_z[|R_t \cap \cL | \ge s_{1}+s_{2}] \le \Pr_z[|R_t \cap \cL | \ge s_{1}]\Pr_z[|R_t \cap \cL | \ge s_{2}].   \]
To see this, let $\ell(j):=\inf \{i:|R_i \cap \cL |=j \} $. Then given that $\ell(s_1) \le t $, in order that also $\ell(s_1+s_2) \le t $ the walk, which at time $\ell(s_1)$ is at some leaf, has to visit $s_2$ new leafs in $t-\ell(s_1) \le t $ steps.
Thus $\Pr_z[|R_t \cap \cL | \ge j \lceil4 \E_z[|R_t \cap \cL |]\rceil ] \le 4^{-j} $ for all $j \ge 1$. In particular, 
 \[\E_z[|R_t \cap A |^{2}] \le \E_z[|R_t \cap \cL  |^{2} ] \le C_2 (\E_z[|R_t \cap \cL |])^{2} \le C_3 g^2(t).\] 
We  now argue that $|R_t \cap A |$ and $ 1_{X_{t} \in \cT_y(d^{k})} $ are positively correlated and that $\E_z[|R_t \cap A |] \succsim \delta g(t) $.  Using the fact that $\Pr_z[X_{t} \in \cT_y(d^{k})] \succsim 1 $, which follows from the choice $t \le d^{k}$ in conjunction with \eqref{eq: SRWtree5}, this implies that
\begin{equation}
\label{eq: Range4}
\begin{split}
\E_z[|R_t \cap A | \cdot 1_{X_{t} \in \cT_y(d^{k})}] & \ge \E_z[|R_t \cap A |]\Pr_z[X_{t} \in \cT_y(d^{k})] \succsim \delta g(t).
\end{split} 
\end{equation}
This implies \eqref{eq: Range3} by the Paley-Zygmund
inequality (e.g.~\cite[Lemma 4.1]{kallenberg2002foundations}). 

For the first inequality in \eqref{eq: Range4} we argue that the walk conditioned on $X_{t} \in \cT_y(d^{k})$, denoted by $(\widehat X_\ell)_{\ell=0}^{\infty}$, can be coupled with the unconditional walk  $( X_\ell)_{\ell=0}^{\infty}$, both started from $z \in \cL( \cT_y)$, such that (1) $| \widehat X_\ell|  \ge | X_\ell| $ for all $\ell$ and (2) whenever  $| \widehat X_\ell|  = | X_\ell| $ and $X_{\ell} \in \cT_y(d^{k})  $ we have that $X_{\ell}=\widehat X_\ell $. In fact, the same remains true when the walk conditioned on $\|X_t\| \le k $ takes the role of either $(X_{\ell})_{\ell=0}^{t}$ or of $(\widehat X_\ell)_{\ell=0}^{t}$ in the previous statement (and the other walk playing the same role it previously had). We leave this as an exercise. Hence in order to prove \eqref{eq: Range3} it remains to verify that indeed $\E_z[|R_t \cap A |] \succsim \delta g(t)$.

  By Lemma \ref{lem: SRWontree} with probability at least $c_6>0$  the walk reaches the root of $\cT_y(d^{k})$ and then by time $t/2$ returns to $\cL(\cT_y(d^{k})) $ before escaping  $\cT_y(d^{k})$.    On this event, the expected number of visits to $\cL(\cT_y(d^{k})) $ between time $t/2$ and time $t$ is at least $c_7 t$. By symmetry, on the aforementioned event, every $v \in \cL(\cT_y(d^{k})) $ has the same contribution to the aforementioned mean. Using the same reasoning as in the proof of \eqref{eq: Range2}, it follows that $\Pr_{z}[T_v \le t] \succsim \frac{t}{d^{k}\log_d (dt)} $ for every $v \in \cL(\cT_y(d^k))$. This concludes the proof of \eqref{eq: Range3} by summing over all $v \in A$.
We now prove \eqref{eq: Range3'}.

In the setup of part (ii), by Lemma \ref{lem: SRWontree} for all $z\in \cT_{x} $ and $32i \le  t \le d^{i-1} $ we have
\begin{equation*}
\begin{split}
\Pr_z[T_{\cL(\cT_y)}<t/2] & \ge \Pr_z[T_{x \wedge y }<t/4] \Pr_{x \wedge y}[T_{\cL(\cT_y)}<t/4] \\ & \succsim td^{-\|x \wedge y\|+1} \cdot d^{ \|y\|-\|x \wedge y\|}        = td^{-(2\|x \wedge y\|-\|y\|-1)}.
\end{split}
\end{equation*}
 Conditioned on $T_{\cL(\cT_y)}=s \le t/2 $ and on that $X_s \in \cT_{y,k}^{\ell}$ we have that $X_s$ is uniformly distributed on the leaf set of  $\cT_{y,k}$. The proof of  \eqref{eq: Range3'} can now be completed using the Markov property and similar calculations to the ones from the proof of \eqref{eq: Range3} in order to show that for all $s \le t/2$ and all $\ell$ we have that, given that  $T_{\cL(\cT_y)}=s  $ and that $X_s \in \cT_{y,k}^{\ell}$, the expected number of leafs of $ \cT_{y,k}^{\ell}$ visited by the walk by time $t$ is $\asymp g(t) $, while its second moment is $\asymp g^2 (t)$.  
\qed

\medskip

\emph{Proof of Corollary \ref{cor:coverleaf}:}
We first prove part (i). Denote the collection of particles occupying $B$ at time 0 whose position at time $s$ is in $\cT_y$ by $\W'_B:=\{w_1,\ldots,w_{|\W'_B|} \} $. Since (by Lemma \ref{lem: SRWontree}) each particle has chance at least $c_0>0$ to be at $\cT_y $ at time $s$, by Poisson thinning, if $C_2$ is taken to be sufficiently large, then the probability that $|\W'_{B}| \le \frac{C_2 c_0}{4} \log |\VV_{d,n} | $ is at most $|\VV_{d,n} |^{-4}$. 
Let $(\w_i(j))_{j=0}^s $ be the walk performed by $w_i$. Denote \[J_i:=\{\w_i(j):j \in ]s[ \} \cap \cL(\cT_y)  \quad \text{and} \quad F_i:=\cup_{j:j \in [i] }J_i .\] Then $F_{i} \setminus F_{i-1}$ is the set of ``new" leafs of $\cT_y$ discovered by $w_i$. Let $U_i$ be the event that  $|F_i  | \ge d^k/4$ or $|F_{i} \setminus F_{i-1}| \ge  cs/ \log_d s   $,
 for some $c>0$ to be determined later (where $F_0:=\eset$). Denote $Z_i:=1_{U_i} 1_{|\W'_{B}| \ge i}$. Observe that conditioned on $1_{|\W'_{B}| \ge i} $ we have that either  $|F_{i-1}  | \ge d^k/4$ and then also $|F_i| \ge d^k/4 $ and so  $Z_i=1$, or   $|F_{i-1}  | < d^k/4$ and then by \eqref{eq: Range3} with $\delta =3/4 $ we have that the probability that  $|F_{i} \setminus F_{i-1}| \ge  cs/ \log_d s   $ (and hence $Z_i=1$) is bounded from below.  It follows  that we can pick some $c'>0$ such that for all $i \in \N$  \[\E[Z_{i} \mid F_{i-1},|\W'_{B}| ] \ge c'1_{|\W'_{B}| \ge i}  \] $\as$.  It follows that conditioned on  $|\W'_{B}|= \ell > \frac{C_2c_{0}}{4} \log |\VV_{d,n} | $, the  distribution of $\sum_{i \in [\ell]} Z_i$ stochastically dominates the $\mathrm{Bin}(\ell,c')$ distribution. Hence if $C_2 \ge 128/(c_{0}c')$ and $n$ is sufficiently large so that $\frac{4\log_d s}{ c   } \cdot d \le 32 \log |\VV_{d,n} | $ then $\ell c' > 32 \log |\VV_{d,n} | \ge  \frac{4\log_{d} s}{c   } \cdot \frac{d^k}{s}$. Thus by Chernoff's bound  \[\Pr \left[\sum_{i \in [\ell]} Z_i \le  \frac{\log_{d} s}{cs   } \cdot \frac{d^k}{4} \mid |\W'_{B}|= \ell \right]< |\VV_{d,n} |^{-4}, \quad \text{for all} \quad \ell > \frac{C_2c_{0}}{4} \log |\VV_{d,n} | .\]  The proof of part (i) is concluded by noting that on the event $\sum_{i \in [|\W'_{B}|]} Z_i >  \frac{\log_{d} s}{cs   } \cdot \frac{d^k}{4} $, it must be the case that $|F_{|\W'_{B}|}| \ge d^{k}/4 $ (and recalling that the probability that $|\W'_{B}| \le \frac{C_2 c_0}{4} \log |\VV_{d,n} | $ is at most $|\VV_{d,n} |^{-4}$). 

\medskip

 We now prove part (ii). After reaching level $n-k$, conditioned on being at $\cL(\cT^i)$ (for some $i \in [d^{n-k}]$) the walk is equally likely to be at each leaf of $\cT^i$ and hence is in $A_i$ w.p.\  $|A_i|/d^k \ge 1/4 $ (where this inequality holds by assumption). Let $u \in \cL$. By \eqref{eq: SRWtree5''} starting from $u$ the walk has probability bounded from below to reach level $n-k$ and then return to the leaf set by time $s/2$.  On this event the expected number of visits to $A$ between time $s/2$ and $s$ is at least $(\sfrac{s}{2} \cdot \pi_n) \cdot \min_{i} \frac{|A_{i}|}{d^k}  \succsim s $ (where $\pi$ is as in \eqref{e:piofY}) and so
\begin{equation}
\label{e:Eatleafsbys}
\E_{u}[|\{i \in [s] :X_i \in A \}| ] \succsim s.
\end{equation}
Observe that if $a \in A $ then by reversibility (and the fact that $a,u \in \cL$)
\begin{equation}
\label{e:Eatleafsbys2}
\sum_{i \in [s]} P^{i}(u,a)=\sum_{i \in [s]} P^{i}(a,u) \le  \Pr_{a}[T_u \le s]\sum_{i \in ]s[} P^{i}(u,u) \asymp  \Pr_{a}[T_u \le s] \log_d s .
\end{equation}
 Denote by $Y_u$ the number of particles from $A$ which visit $u$ by time $s$. Then by Poisson thinning $Y_u$ has a Poisson distribution whose mean $\mu_u$ is by \eqref{e:Eatleafsbys2} and \eqref{e:Eatleafsbys}  \[\mu_u:=\sfrac{\la}{2} \sum_{a \in A}\Pr_{a}[T_u \le s] \ge  \frac{c_1 \la\sum_{a \in A}\sum_{i \in [s]} P^{i}(u,a)}{\log_d s }   \] \[ = \frac{c_{1}\la\E_{u}[|\{i \in [s] :X_i \in A \}| ]  }{\log_d s } \ge \frac{c_{2}\la s  }{\log_{d} s } \ge 5\log |\VV_{d,n} |,  \] 
where the last inequality holds provided that $C_1$ is sufficiently large. The proof is concluded by applying a union bound over $\cL$ (noting that $e^{-5\log |\VV_{d,n} |}|\cL| \le |\VV_{d,n}|^{-4}$).  

We now prove part (iii). We first note that if $\cL \subseteq \mathcal{R}_s $, then deterministically we must have that $\mathcal{R}_s= \VV_{d,n}$. That is, the last vertex to be activated must be a leaf. We partition the particles into two independent sets of density $\la/2$ and include the planted particle $\plant$ in the first set. First consider the dynamics only w.r.t.\ the first set (as if the second set of particles does not exist) in the case that the particle lifetime is $s$. Observe that this dynamics is exactly the frog model with particle density $\la/2$. 

Now consider the dynamics only w.r.t.\ the particles from the second set (as if the first set of particles does not exist), where initially the activated particles are the ones whose initial position is one which is    visited by at least one particle from the first set of particles in the aforementioned dynamics of the particles from the first set. Consider the case that the event $\mathrm{Hom}$ occurs w.r.t.\ the dynamics of the particles from the first set (the probability of this event is $\PP_{\la/2}[\mathrm{Hom}] $). Then we may apply part (ii) to the second set of particles to conclude the proof (using the aforementioned observation that it suffices that all leafs are activated). 
\qed

\section{Appendix B - Hitting times in birth and death chains}
\label{s:BD}
A classical result of Karlin and McGregor \cite{karlin} is that for a continuous-time birth and death chain on $]n[:=\{0,1,\ldots,n\}$  the law of the hitting time of the end-point $0$ starting from the other end-point $n$ is a convolution of Exponential distributions, whose parameters are the minuses of the eigenvalues of the generator of the process killed at $0$. This was  rediscovered by Keilson \cite{keilson}.   The discrete-time case is due to Fill \cite{Fill}. In discrete-time the Exponentials are replaced by Geometric distributions, provided that the eigenvalues of the restriction of the transition matrix to $[n]=\{1,2,\ldots,n\}$ are all positive.  Miclo \cite[Proposition 7.1]{miclo} showed that it suffices that all of the aforementioned eigenvalues will be non-negative. He also  gave an extremely elegant generalization of this result to arbitrary reversible Markov chains. 

The next proposition asserts that if $K(i,i)=0$ for all $i \neq 0 $ then essentially one does not have to require  the eigenvalues to be non-negative. 
\begin{prop}
\label{p:log-con}
Let $K$ be the transition matrix of a birth and death chain on $]n[:=\{0,1,\ldots,n\}$ such that $K(i,i)=0$ for all $i \in [n] $. If $n$ is even (resp.\ odd) then
\begin{itemize}
\item[(1)]
 The eigenvalues of the restriction of $I-K^2$ to $\{2,4,\ldots,n\}$ (resp.\ $\{2,4,\ldots,n-1\}$), denoted by $\gamma_1,\gamma_2,\ldots ,\gamma_{n/2}$ (resp.\  $\gamma_1',\gamma_2',\ldots ,\gamma_{\sfrac{n-1}{2}}'$) all lie in $(0,1]$.
\item[(2)] The law of $T_0/2$ (resp.\ $\half (T_0-1)$) started from $n$ is a convolution of $n/2$ (resp.\ $\sfrac{n-1}{2} $) Geometric distributions with parameters $\gamma_1,\gamma_2,\ldots ,\gamma_{n/2}$ (resp.\  $\gamma_1',\gamma_2',\ldots ,\gamma_{\sfrac{n-1}{2}}'$). In particular, it is log-concave and unimodal\footnote{A distribution $\mu$ on $\N$ is unimodal if $\mu(n)$ is non-decreasing in $n$ for $n \le m_*$ and non-increasing in $n$ for $n \ge m_*$, where $m_*$ is the mode of $\mu$ (i.e.\ $\mu(m_*)=\max_n \mu(n)$).}.
\item[(3)] $1/\gamma_1 \ge \half \mathbb{E}_2[T_0]= \frac{\pi(2)+\pi(4)+\cdots +\pi(n)}{ \pi(2)P^{2}(2,0)} $ (resp.\ $1/\gamma_1' \ge \half \mathbb{E}_2[T_0]= \frac{\pi(2)+\pi(4)+\cdots +\pi(n-1)}{ \pi(2)P^{2}(2,0)} $), where $\pi(i)$ is the sum of the edge weights of the two edges incident to $i$.
\end{itemize}
  
\end{prop}
\begin{proof}
As mentioned above Fill \cite[Theorem 1.2]{Fill} showed that for a discrete-time birth and death chain on $]n[$ whose transition  matrix $K$ is such that all of  the eigenvalues of the restriction of the Laplacian $I-K$ to $[n]$ all lie in $(0,1]$, the law of $T_0$ starting from $n$ is a convolution of Geometric distributions, whose parameters are the aforementioned eigenvalues.  In particular, since the Geometric distribution (with every parameter) is log-concave and the class of log-concave distributions is closed under convolutions, it follows that $T_0$ is log-concave (starting from $n$) under the aforementioned eigenvalue condition. Finally, every log-concave distribution is unimodal. We now reduce the general case to this case. We only treat the case that $n$ is even, as the case it is odd is completely analogous. 

Now, if $n$ is even and $K(i,i)=0$ for all $i \in ]n[ $ then instead of $K$ we can consider $\widehat K$ the restriction of $K^2$ to $D:=\{0,2,\ldots n \}$. Observe that because $K(x,x)=0$ for all $x \in [n]$ we have that $\widehat K$ is a birth and death chain on $\{0,2,\ldots n \}$ (after relabeling each state $2i$ by $i$).

We now verify that all of the eigenvalues of the restriction of $\widehat K $ to $\{2,4,\ldots,n \}$ are non-negative. Clearly all of the eigenvalues of $I-   \widehat K$ lie in $[0,1] $. Denote them by $0=\beta_1 < \beta_2 \le \cdots  \le \beta_{\frac{n}{2}+1} \le 1 $. By the interlacing eigenvalues Theorem the eigenvalues of the restriction of $I- \widehat K$ to $\{2,4,\ldots,n\}$, denoted by $\gamma_1<\gamma_2 \le \cdots \le \gamma_{n/2}$, satisfy that $ \beta_i \le \gamma_i \le \beta_{i+1}  $ for all $i \in [n/2]$ and also that  $\gamma_1>0 $ (since the restriction of $\widehat K $ to  $\{2,4,\ldots,n\}$ is strictly sub-stochastic). In particular, $\gamma_{n/2} \le 1 $.  This establishes the non-negativity of the eigenvalues of the restriction of $\widehat K $ to $\{2,4,\ldots,n \}$. 

 Let $\widehat T_0 $ (resp.\ $T_0$) be the hitting time of $0$ w.r.t.\ $K^2 = \widehat K $ (resp.\ $K$) starting from $n$. Then $\widehat T_0=T_0/2 $. Applying Fill's result to $\widehat K$ concludes the proof of part (2). 

We now prove part (3). The identity $\half \mathbb{E}_2[T_0]=\mathbb{E}_2[\widehat T_0]= \frac{\pi(2)+\pi(4)+\cdots +\pi(n)}{ \pi(2)P^{2}(2,0)} $ follows from a general formula for expected crossing times of an edge in birth and death chains in terms of bottleneck ratios (e.g.\ \cite[Eq.\ (5.14)]{basu}). The law of $\widehat T_0$ under both $\Pr_2$ and $\Pr_{\mu}$, where $\mu(2i):=\frac{\pi(2i)1_{i>0 }}{\pi(2)+\pi(4)+\cdots +\pi(n)}$,  is a mixture of Geometric distributions with parameters $0<\gamma_1<\gamma_2 \le \cdots \le \gamma_{n/2}$. For $\Pr_\mu $ this is a standard result in the theory of complete monotonicity (e.g.\ \cite[Lemma 3.7]{basu}). For  $\Pr_2$ this follows from that for $\Pr_{\mu}$ via Kac's formula which asserts that  $\Pr_2[\widehat T_0=t]\mathbb{E}_2[\widehat T_0] =\Pr_\mu[\widehat T_0 \ge t] $ (e.g.\ \cite[Eq.\ (5.13)]{basu}). 

Finally, it follows that there exists some convex combination $0 \le p_1,\ldots,p_{n/2} \le 1 $ such that $\sum_i p_i=1$ and $\mathbb{E}_2[\widehat T_0]=\sum_i p_i/\gamma_i \le 1/\gamma_1$. 
\end{proof}
\end{document}